\documentclass[12pt]{article}
\usepackage{apacite}
\usepackage[margin=0.6in]{geometry}
\usepackage{mathtools}
\usepackage{amssymb} 
\usepackage{enumitem}
\usepackage{amsthm}
\usepackage{amsopn}
\usepackage{bbm,hyperref}
\usepackage{natbib}
\usepackage{enumitem}
\DeclarePairedDelimiterX\Basics[1](){\let\given\sgiven #1}
\usepackage[utf8]{inputenc}
\usepackage{color}
\usepackage{caption}
\usepackage{subcaption}
\usepackage{enumitem}
\usepackage{pifont}
\usepackage{algorithm}
\usepackage[noend]{algpseudocode}
\usepackage[T1]{fontenc}
\usepackage{etoolbox}

\usepackage{float}
\newtheorem{theorem}{Theorem}

\newtheorem{lemma}[theorem]{Lemma}
\newtheorem{remark}[theorem]{Remark}
\newtheorem{definition}[theorem]{Definition}
\newtheorem{proposition}[theorem]{Proposition}

\numberwithin{equation}{section}
\newtheorem*{assumption*}{\assumptionnumber}
\newcommand{\innermid}{\nonscript\;\delimsize\vert\nonscript\;}
\newcommand{\activatebar}{%
  \begingroup\lccode`\~=`\|
  \lowercase{\endgroup\let~}\innermid 
  \mathcode`|=\string"8000
}
\providecommand{\assumptionnumber}{}
\makeatletter

\DeclareMathOperator*{\argmax}{\arg\max}

\newlist{steps}{enumerate}{1}
\setlist[steps, 1]{label = Step \arabic*:}

\let\given\givenbase

\newcommand{\interior}[1]{%
  {\kern0pt#1}^{\mathrm{o}}%
}

\begin{document}
\title{
Interbank lending with benchmark rates:\\
Pareto optima for a class of singular control games
\\
{\small {\it In memory of Mark H Davis, mentor and friend}}}
\author{
Rama Cont
\thanks{Mathematical Institute, University of Oxford. 
\ {\tt Rama.Cont@maths.ox.ac.uk, xur@maths.ox.ac.uk}}
\and Xin Guo
\thanks{Dept of Industrial Engineering and Operations Research, University of California, Berkeley.\ {\tt xinguo@berkeley.edu}}
\and Renyuan Xu $^*$
}
\date{First version: October 30, 2020.
This version: June 12, 2021.}
\maketitle
\begin{abstract}
We analyze a class of 
stochastic  differential games of singular control, motivated by the study of a dynamic model of interbank lending with benchmark rates. We describe   Pareto optima for this game and show how they may be achieved through the intervention of a regulator, whose policy is a solution to a singular stochastic  control problem. Pareto optima are  characterized  in terms of the solutions to a new class of  Skorokhod problems with piecewise-continuous  free boundary. 

Pareto optimal policies are shown to correspond to the enforcement of endogenous   bounds on interbank lending rates.
Analytical comparison between Pareto optima and Nash equilibria 
provides insight into the impact of regulatory intervention on the stability of  interbank rates. 
\end{abstract}
\tableofcontents

\section{Introduction}

The market for interbank lending offers an interesting example of strategic interaction among financial institutions in which players react to the {\it distribution} of the actions of other players.
One of the widely commented features of the interbank market is the fixing mechanism for interbank benchmark interest rates, the most well-known example of which is the London Interbank Offer Rate (LIBOR) which plays a central role in financial markets.
Historically these benchmarks have not been negotiated rates but a `trimmed' average of quotes collected daily from major banks. Every day, participating banks contribute a quote   representing their offered rate; a calculation agent then `trims' the tails of the distribution by removing the highest and lowest quotes and computes the value of the benchmark rate   as a weighted average of the remaining non-discarded quotes \citep{avellaneda2010}.
The resulting benchmark rate   --the LIBOR rate-- then serves as a reference for the valuation of interbank loans and debt contracts, as well as many other financial contracts indexed on the benchmark rate.
A deviation (spread) of a bank's rate from the benchmark may lead to a perception of credit risk and loss of market share -if the spread is positive- or an opportunity cost if the spread is negative, thus  incentivizing banks to align their offered rates with the benchmark.


This mechanism leads to strategic interactions among market participants in a dynamic setting, where interactions are mediated through an average action, or more generally through the distribution of actions of other participants and has been criticized for its vulnerability to manipulations \citep{avellaneda2010}, which have been extensively documented \citep{wheatley,duffie2015}. One of the lessons from the manipulation of LIBOR and other benchmarks is that insufficient attention had been paid to incentives, strategic interactions, mechanism design and the role of the regulator in such markets.

\subsection{A model of interbank lending with benchmark rates}\label{sec:LIBOR}
We shall now describe a stylized model of interbank rates which represents interactions among banks in terms of a stochastic dynamic game. 

Consider first an exogenous process $r_t$ representing a rate set by the central bank, with respect to which banks will position their lending rates. $r_t$ is typically modeled as a mean-reverting diffusion process driven by a multidimensional Brownian motion $\pmb{B}$ representing risk factors driving random macroeconomic shocks.
Each bank $i$ quotes a rate $r^i_t$ at a `spread' $X^i_t$ with respect to the reference rate $r_t$: 
 $r^i_t= r_t + X^i_t$. 
The spread of each bank $i$ is affected by  the macroeconomic shocks 
but the bank  may control its rate $r^i_t$ through positive or negative adjustments to its spread $X^i_t$, which we may represent by a pair $(\xi^{i,+}, \xi^{i,-})$ of non-decreasing processes representing increases (resp. decreases) in the spread: 
\begin{eqnarray}
dX_t^i =  {\pmb\sigma}^{i} \cdot d\pmb{B}_t+ d\xi_t^{i,+} -  d\xi_t^{i,-},\label{eq.interbank}
\end{eqnarray} 
where $\pmb{\sigma}^{i}$  is a volatility matrix representing the sensitivity of the spread $X^i$ to macroeconomic factors.
The benchmark (`LIBOR') rate $L_t$ is then defined as a weighted average of these offered rates:
$$ L_t=r_t+\overline{X}_t,\quad  \overline{X}_t =\sum_{i=1}^N a_i X^i_t\qquad {\rm and}\quad a_i\geq 0,\qquad \sum_{i=1}^N a_i=1.$$
Note that the `drift' term in the dynamics \eqref{eq.interbank} originates from the control. One may also consider  an additional drift term $\mu^idt$ in the uncontrolled dynamics, a positive drift corresponding to a bank whose creditworthiness is gradually deteriorating, leading to a steady increase of its spread. (See more general set-up in Section \ref{sec:formulation}.)

 We now turn to the incentives and costs faced by the banks. 
Each bank $i$ receives interest income from its lending activity, at rate  $r^i_t$.
The interest income of the bank over a short period $[t,t+ dt]$ is $r^i_t Q^i_t\ dt$ where $Q^i_t >0$ is the volume of lending activity (loan volume). Given that the bank can borrow at the interbank $L_t=r_t+\overline{X}_t,$ this represents an {\it opportunity cost} of
 $(\overline{X}_t -X^i_t) Q^i_t\ dt$.
In a competitive lending market, the loan volume $Q^i_t$ of bank $i$ will be a decreasing function $q_i(.)$ of its spread $r^i_t-L_t=X^i_t-\overline{X}_t$ relative to the benchmark rate:  $Q_t^i=q_i(X^i_t-\overline{X}_t)$. Assuming an inter-temporal discount rate of $\rho>0$, this leads to a running cost term $$\int_0^\infty e^{-\rho t}(\overline{X}_t -X^i_t)\  q_i(X^i_t-\overline{X}_t)\ dt.$$ For example, an affine dependence $q_i(x)=Q^i_0-\kappa_i x$, where $\kappa_i>0$ represents the sensitivity of loan volume to  the interest rate, leads to a linear-quadratic  cost
$\int_0^\infty {e}^{-\rho t}[Q^i_0 (\overline{X}_t -X^i_t)+ \kappa_i (\overline{X}_t -X^i_t)^2]  dt.$

These considerations only pertain to the relative costs of bank simultaneously engaging in borrowing and lending. Other constraints prevent the banks from deviating from the reference rate beyond a certain level; these are often `soft', rather than hard (i.e., inequality),  constraints and may be modeled by a  penalty on $|r^i_t|$, or equivalently a running cost $f_i(X^i_t)$ where $f_i$ is centered at some reference value and increases fast enough (e.g., quadratically) 
at infinity. As an example we shall use $f_i(x)= \nu_i (x-s_0)^2$ with $\nu_i>0$.

The benchmark fixing mechanism described above may be incorporated in the model through a cost term associated with the control $(\xi^{i,+},\xi^{i,-})$. Recall that
the LIBOR is computed as a {\it trimmed} average of quotes, discarding the highest and lowest `outliers'. This means an offered rate $X^i$ will not be taken into account if it lies too far from the mean. In absence of collusion between banks, this mechanism discourages them from making large daily adjustments to their offered rates, as a large upward or downward adjustment may result in their quotes being disregarded in the benchmark calculation. This may be modeled through a cost term which penalizes the  size  of the adjustment e.g., $K_i^{+}  d {\xi}_t^{i,+} +K_i^{-} d {\xi}_t^{i,-}$, with $K_i^+, K_i^- >0$,
where $1/K_i^-$ (resp. $1/K_i^+$) represents a typical distance  $(X^i-\overline{X})_+$ (resp. $(\overline{X}-X^i)_+$)  beyond which quotes are discarded. For instance one can take $K_i^+=K_i^-=1/\gamma $ where $\gamma$ represents a measure of dispersion (interquartile range or multiple of standard deviation) of the quote distribution. The case of an asymmetric penalty $K_i^+> K_i^-$ (resp. $K_i^+< K_i^-$) is useful to model the case of a bank $i$ systematically quoting above (or below) the benchmark.
This leads to an objective function
\begin{eqnarray}
J^i (\pmb{x};\pmb{\xi}) = \mathbb{E}\Bigg[ \int_0^{\infty} e^{-\rho t} \Bigg(\left(\overline{X}_t -X^i_t\right) q_i\left(X^i_t-\overline{X}_t\right)dt+ \nu_i \left(X^i_t-s_0\right)^2 dt+ K_i^{+}  d {\xi}_t^{i,+} +K_i^{-} d {\xi}_t^{i,-} \Bigg)\  \Bigg| {\pmb{X}}_{0-}=\pmb{x} \Bigg]\quad  \label{eq.payoff}
\end{eqnarray}
for bank $i$, where the control variable is
a pair of non-decreasing processes $\left(\xi^{i,+}, \ \xi^{i,-}\right)$
 representing the rate adjustments of bank $i$  and the expectation is taken with respect to the law of the controlled process \eqref{eq.interbank}.
 The controls $\xi^{i,+}, \ \xi^{i,-}$ are in general allowed to be right-continuous   with left limits (c\`adl\`ag) with possible jumps  as well as continuous adjustments to the rates.
  Such controls are called singular controls \citep{BSW1980, karatzas1982} and have been used for analyzing 
optimal investment policy and option pricing and hedging problems with transaction costs \citep{davis1990portfolio,davis1993european,kallsen2017general,zariphopoulou1992investment}.

In the case where $a_i=\frac{1}{N}$, $q_i=q_j$, $\nu_i=\nu_j$ and $K_i^{\pm}=K_j^{\pm}$ for $i\neq j$, the payoff structure is symmetric under permutation of indices and this can be formulated as mean field game \citep{LL2007, HMC2006}, which was studied under Nash equilibrium in \citep{GX2019}. However we shall not need this assumption and will treat below the case of a more general, not necessarily symmetric, cost function $h^i(\pmb{X}_t)$. This is more natural for the interbank lending problem.

\subsection{A class of stochastic differential  games of singular control}
Motivated by the example above, we study a class of $N$-player  stochastic differential  games,  where each player $i=1, \cdots, N$ controls a diffusive process $X^i_t$
 through  $\pmb{\xi}^i :=(\xi^{i,+}, \xi^{i,-})$ additive control terms
 \begin{eqnarray}\label{eq.evolution}
dX_t^i =  \mu^i dt   + \pmb{\sigma}^i \cdot d \pmb{B}_t + d\xi_t^{i,+} -  d\xi_t^{i,-},\quad X_{0-}^i = x^i,
\end{eqnarray}
and seeks to minimize  the sum of a discounted running cost and a proportional cost of intervention
\[ 
J^i (\pmb{x};\pmb{\xi}) = \mathbb{E}\Bigg[ \int_0^{\infty} e^{-\rho t} \left(h^i(\pmb{X}_t)dt+ K_i^{+}  d {\xi}_t^{i,+} +K_i^{-} d {\xi}_t^{i,-} \right)\  \Bigg| {\pmb{X}}_{0-}=\pmb{x} \Bigg].
\]
The first two terms in \eqref{eq.evolution} correspond to the `baseline' (uncontrolled) diffusion dynamics, and the last two term correspond to the 
 control $\pmb{\xi}^i=(\xi^{i,+}, \xi^{i,-})$, modeled as 
a pair of non-decreasing c\`adl\`ag  processes. 
Here we focus on Pareto-optimal outcomes.

\paragraph{Contribution.} The present work is a study of Pareto-optimal policies for the class of stochastic singular control games considered above, motivated by the interbank lending problem.
We relate the Pareto optima of this game to the solution of a regulator's problem, characterized as a high-dimensional singular stochastic control problem which we study in detail.
The regularity analysis of the value function, following the approach of \cite{SS1989}, for the regulator's problem enables us to  characterize  the optimal controls for this problem
and subsequently the Pareto-optimal policies for the $N$-player game.

We obtain a description of Pareto-optimal policies in terms of a multidimensional Skorokhod problem  for a `regulated diffusion' in a bounded region whose boundary  is  piece-wise smooth with 
possible corners. 
The state process follows a diffusion process in the interior, and the control intervenes only at the boundary to reflect it back into the interior.


Finally, we derive explicit descriptions of Pareto-optimal policies when
$N=2$. This complements the existing literature on Nash equilibrium for stochastic two player games \citep{DF2018,dianetti2020nonzero,hernandez2015zero,KZ2015}.
Analytical comparison between the Pareto-optimal and the Nash equilibrium solutions  demonstrates the role of regulator in the interbank lending market. 

Our analysis for the general case $(N \ge 2)$ provides insights for  regulatory intervention on  the interbank  market. In particular, it allows us to quantify the impact of a regulator on the stability of the benchmark rate.

\paragraph{Relation with previous literature.}

Stylized mean-field models of interbank borrowing and lending have been considered by \cite{CFS2013} and \cite{sun2018}, who focus on Nash equilibria in the case of a large  number of (indistinguishable) players.  Here we consider the case of a finite number of players, allowing them to be non-identical which is more realistic in terms of the  interbank problem at hand, and our focus is on Pareto optima and the role of a regulator.

A related strand of literature  consists of studies for central bank interventions on interest rates and exchange rates using an {\it impulse control} approach  \citep{bensoussan2012,cadenillas2000,jeanblanc1993impulse}. In these approaches, interventions are associated with a fixed cost. The singular control framework adopted here seems more natural  for modeling situations such as interbank markets 
where the cost of intervention is proportional to the action rather than fixed. Singular controls allow for discontinuities and include impulse controls as  special cases. 

Nash equilibria for stochastic games of singular control have been studied by \cite{CFR2013, DF2018, dianetti2020nonzero, hernandez2015zero}; on the other hand, there are few studies of Pareto-optimal strategies for such games.
\cite{ABP2017} consider a two-player game in an impulse control framework between a representative energy consumer and a representative electricity producer, and  derive an asymptotic Pareto-optimal policy. \cite{FL2016} solve explicitly a mean-variance portfolio optimization problem with $N$ stocks. \cite{FRS2017}  and \cite{WE2010}
 consider the problem of public good contribution and  analyze the Pareto-optimal policy for the $N$-player stochastic game under the framework of regular control and singular control, respectively.

The analysis of  Pareto optima in  stochastic  games is often through studying an auxiliary $N$-dimensional stochastic control problem. This approach can be traced back to the economic literature on  mechanism design and social welfare optimization in \cite{Bator1957} and \cite{Coleman1979}.
The  mathematical challenge lies in the associated high-dimensional Hamilton--Jacobi--Bellman (HJB) equations and characterizing the optimal control policy from the regulator.

\paragraph{Outline.} 
The remainder of the paper is organized as follows. Section \ref{sec:formulation} presents  the mathematical formulation of the $N$-player stochastic differential game, and describes its relation with the auxiliary control problem. Section \ref{sec:central_math} provides detailed analysis of the auxiliary control problem and  the construction of the  optimal strategies.  Section \ref{sec:po_strategy} characterizes the Pareto optima in terms of  a sequence of Skorokhod problems. Implications of our analysis for the interbank lending problem are discussed in Section \ref{sec.LIBOR}. Section \ref{sec:bank} provides  explicit   solutions in the case  $N=2$, and compares it with the Nash equilibrium.



\section{Mathematical formulation of the game} \label{sec:formulation} 
In this section, we describe the mathematical framework of the $N$-player game. 

\paragraph{Controlled dynamics.}
Let  $(X^i_t)_{t\ge 0} \in \mathbb{R}$ denote the state of player $i$ at time $t$, $1 \le i \le N$. 
With absence of
controls,   ${\pmb{X}}_t:=(X^1_t, \ldots, X^N_t)\in\mathbb{R}^N$ follows 
\begin{equation} 
\label{Eq. diffusion}
{\pmb{X}}_t= {\pmb{X}}_{0}+ \pmb{\mu}t
+\pmb{\sigma}  \pmb{B}_t, \quad  {\pmb{X}}_{0}=(x^1, \ldots, x^N),
\end{equation} 
where $\pmb{B}: =(B^1, \ldots, B^D)\in \mathbb{R}^D$ is a $D$-dimensional Brownian motion on a filtered probability space $(\Omega,
\mathcal{F}, \{\mathcal{F}_t\}_{t \ge 0}, \mathbb{P})$, and  $\pmb{\mu}:=(\mu_1,\ldots,\mu_N)\in \mathbb{R}^N$ and  $\pmb{\sigma}:=(\sigma_{ij})_{1 \le i  \le N, 1 \leq j \leq D}\in \mathbb{R}^{N \times D}$ are constants with  
  $\pmb{\sigma}\pmb{\sigma}^T \succeq \lambda \textbf{I}$ for some $\lambda>0$.

When player $i$ chooses a control $\pmb{\xi}^i:=(\xi^{i,+}, \xi^{i,-})$   from an admissible control set $\mathcal{U}_N^i$,  then $X^i_t$ evolves as
\begin{eqnarray}\label{Ndynamics}
d X_t^i = \mu^i dt + \pmb{\sigma}^i \cdot d \pmb{B}_t + d \xi_t^{i,+} - d \xi_t^{i,-},\quad X_{0-}^i = x^i.
\end{eqnarray}
Here  $\pmb{\xi}^i=(\xi^{i,+}, \xi^{i,-})$ 
is a pair of non-decreasing  c\`adl\`ag    
processes and $\pmb{\sigma}^i$ is the $i^{th}$ row of the  volatility matrix $\pmb{\sigma}$. We will denote by $\mathbb{P}^{\pmb{x}}$ the law of the process \eqref{Ndynamics} and $\mathbb{E}_{\pmb{x}} $ the expectation with respect to this law.


\paragraph{Admissible controls.}
The set $\mathcal{U}_N^i$ of admissible controls for player $i$ is defined as 
\begin{eqnarray}\label{A_N}
\begin{aligned}
 \mathcal{U}_{N}^i= & \left\{ (\xi_t^{i,+},\xi_t^{i,-})_{t \ge 0} \given  \xi_t^{i,+} \mbox{ and }\xi_t^{i,-} \mbox{ are } {\mathcal{F}_{t}}\mbox{-progressively measurable, c\`adl\`ag non-decreasing,} \right.\\
 & \hspace{50pt} \left.{} \mbox{ with }\mathbb{E} \left[ \int_0^{\infty}e^{-\rho t} d\xi_t^{i,+} \right] <\infty, \mathbb{E} \left[ \int_0^{\infty}e^{-\rho t}d\xi_t^{i,-} \right] <\infty, \xi_{0-}^{i,+}=0, \xi_{0-}^{i,-}=0 \right\}.
\end{aligned}
\end{eqnarray}

\paragraph{Objective functions.}
Each player $i$ chooses a control $(\xi^{i,+}, \xi^{i,-})$ in $\mathcal{U}^i_N$ to minimize
\[ 
J^i (\pmb{x};\pmb{\xi}) = \mathbb{E}_{\pmb{x}}  \int_0^{\infty} e^{-\rho t} \left[h^i(\pmb{X}_t)dt+ K_i^{+}  d {\xi}_t^{i,+} +K_i^{-} d {\xi}_t^{i,-} \right]. \label{N-game}  \tag{\textbf{N-player}}
\]
Here $\rho>0$ is a constant discount factor, $K_i^{+},K_i^{-} >0$ are the cost of controls, and $h^i(\pmb{x}): \mathbb{R}^N \rightarrow \mathbb{R}_+$ is the running cost function.



 \label{sec:central}
 We have focused on  characterizing  Pareto  optima of  the game  \eqref{N-game} subject to the dynamics \eqref{Ndynamics}. 
 
\begin{definition}[Pareto optimality] 
$\pmb{\xi}^* \in \mathcal{U}_N := \Pi_{i=1}^N \mathcal{U}_{N}^i$  is a Pareto-optimal policy for the game \eqref{N-game} if and only if there does not exist $\pmb{\xi} \in \mathcal{U}_N $ such that, for all $\pmb{x}\in\mathbb{R}^N$,
\begin{eqnarray*}
\forall i\in\{1,\ldots,N\},\quad J^i \left(\pmb{x}; \pmb{\xi} \right) \leq J^i \left(\pmb{x}; \pmb{\xi}^* \right); \quad{\rm and}\quad \exists j \in\{1,\ldots,N\},\quad J^j \left(\pmb{x}; \pmb{\xi} \right) <J^j \left(\pmb{x}; \pmb{\xi}^* \right).
\end{eqnarray*}
\end{definition}
Pareto optima correspond to  {\it efficient} outcomes of a game, which may or may not come from decentralized optimization by $N$ players. The intervention of a regulator may be necessary to enforce a Pareto-optimal policy.

\section{Regulator's problem }\label{sec:central_math}

To study Pareto optima for game \eqref{N-game}, we introduce a `welfare function' defined as an aggregate cost:
\begin{eqnarray}
J(\pmb{x}; \pmb{ \xi}) & = &  \sum_{i=1}^N L_i\  J^i(\pmb{x},\pmb{\xi})\label{eq.regulator}\\ 
& = & \mathbb{E}_{\pmb{x}}   \int_0^{\infty} e^{-\rho t} \left[ H(\pmb{X}_t)dt+ \sum_{i=1}^N L_i K_i^{+}  d {\xi}_t^{i,+} +\sum_{i=1}^N L_i K_i^{-} d {\xi}_t^{i,-} \right],\nonumber
\end{eqnarray}
where 
the dynamics of $\pmb{X}_t$ is given by (\ref{Ndynamics}), and  \begin{eqnarray}\label{H}
H(\pmb{x}) := \sum_{i=1}^N L_i h^i(\pmb{x}), \,\,\mbox{with}\,\, L_i >0 \,\,\mbox{and}\,\, \sum_{i=1}^N L_i=1.
\end{eqnarray}
We will show that  Pareto optima of \eqref{N-game} correspond to
solutions of the following auxiliary    stochastic control problem 
\[
v(\pmb{x}) = \min_{\pmb{\xi} \in \mathcal{U}_N} J(\pmb{x};\pmb{\xi}),\label{centrol_controller}\tag{\textbf{Regulator}}
\] 
which may be interpreted as the problem facing a market regulator seeking to optimize the aggregate cost \eqref{eq.regulator}.

To ensure the well-definedness of the game, the following assumptions will be made throughout, unless otherwise specified. 
\paragraph{Assumptions.} 
There exist $C>c>0$  such that 
\begin{enumerate}[font=\bfseries,leftmargin=3\parindent]
\item [A1.] $\forall \pmb{x} \in \mathbb{R}^N$, $0 \leq H(\pmb{x}) \leq C (1+\|\pmb{x}\|^2).$
\item [A2.] $\forall \pmb{x} ,\pmb{x}^{\prime} \in \mathbb{R}^N$, $|H(\pmb{x}) - H(\pmb{x}^{\prime})| \leq C (1+\|\pmb{x}\| +\|\pmb{x}^{\prime}\|)\|\pmb{x}-\pmb{x}^{\prime}\|.$
\item [A3.] $H(\pmb{x}) \in \mathcal{C}^{2}(\mathbb{R}^N)$, $H$ is convex, with $0<c \leq \partial_{\pmb{z}}^2 H(\pmb{x}) \leq C$ for all unit direction $\pmb{z} \in \mathbb{R}^N$.
\end{enumerate}
{For example, for the payoff described in the interbank lending problem in Section \ref{sec:LIBOR},  \begin{equation}
    H(\pmb{x})=\sum_{i=1}^NL_i \left[\kappa_i\left(x^i-{\sum_{j\neq i}a_j x^j}\right)^2 
+\nu_i (x^i)^2 \right]\quad {\rm with}\quad \kappa_i,\nu_i>0.
\end{equation}
Then $H$ satisfies  \textbf{A1}-\textbf{A3} for any choice of weight $L_i>0$.}


We shall first  analyze the regularity  of the value function $v$,  which is necessary for  subsequently establishing  the existence and uniqueness of the optimal control. As we shall see,  the optimal control for \eqref{centrol_controller}  yields a Pareto-optimal policy for game (\ref{N-game}).

The regularity analysis of the value function  involves several steps.  The first step is to show that the value function for \eqref{centrol_controller} is a viscosity solution to the following  HJB equation 
\begin{eqnarray}\label{po_hjb}
\max \{  \rho u - \mathcal{L} u -H(\pmb{x}) ,  \beta(\nabla u)-1  \}=0,
\end{eqnarray}
with the operator $\mathcal{L} = \frac{1}{2}\sum_{i,j=1}^N \pmb{\sigma}^{i} \cdot \pmb{\sigma}^j\,\partial ^2_{ x^i  x^j}+\sum_{i=1}^N \mu^i \,\partial_{ x^i }, $
and 
\begin{eqnarray}\label{gamma_function}
\beta(\pmb{q}) = \max_{1 \leq i \leq N} \left[ \left( \frac{q^i}{L_i K_i^{-}} \right)^+  \lor \left( \frac{q^i}{L_i K_i^{+}} \right)^- \right],
\end{eqnarray}
where $\pmb{q}:=(q^1,\cdots,q^N)$, $(a)^+ = \max\{0,a\}$ and  $(a)^- = \max\{0,-a\}$ for any $a \in \mathbb{R}$.  The second step is to show that the value function for \eqref{centrol_controller} is $\mathcal{W}_{loc}^{2,\infty}$. 

Let us start with the following property of the value function $v$ for  \eqref{centrol_controller}. Throughout the paper, $K$ will be used in the proof for generic positive constants which may represent different values for different estimates.
\begin{proposition}\label{prop:estimation}
Under Assumptions {\bf A1-A2}, there exists $K>0$ such that
\begin{enumerate}[font=\bfseries,leftmargin=3\parindent]
\item[$(i)$] $0  \leq v(\pmb{x}) \leq K(1+\|\pmb{x}\|^2)$, $\forall \pmb{x} \in \mathbb{R}^N$;
\item[$(ii)$] $|v(\pmb{x})-v(\pmb{x}')| \leq K(1+\|\pmb{x}\|+\|\pmb{x}^{\prime}\|)\|\pmb{x}-\pmb{x}^{\prime}\|$, $\forall \pmb{x},\pmb{x}^{\prime}\in \mathbb{R}^N$.
\end{enumerate}
\end{proposition}

\begin{proof}
First, $v(\pmb{x}) \geq 0$ is clear by the non-negativity of $H(\pmb{x})$. Moreover, by the property that $\pmb{\sigma}\pmb{\sigma}^T \succeq \lambda \textbf{I}$ with $\lambda>0$, it follows from a known estimate and martingale argument \citep[(2.15)]{MR1983} that the solution {$\{\tilde{\pmb{X}}_t\}_{t \geq 0} := \{\pmb{x}+\pmb{\mu}t +  \pmb{\sigma}\pmb{B}_t\}_{t \geq 0}$} with $\pmb{\xi} = \pmb{0}$ satisfies
$$\mathbb{E}_{\pmb{x}} \int_0^{\infty} e^{-\rho t} \|\tilde{\pmb{X}}_t\|^2dt \leq K(1+\|\pmb{x}\|^2) , \quad \forall \pmb{x} \in \mathbb{R}^N,$$
for some constant $K>0$. By Assumption \textbf{A1},  there exists a constant $K>0$ such that 
$$v(\pmb{x}) \leq J(\pmb{x},\pmb{0}) \leq K(1+\|\pmb{x}\|^2), \ \ \forall \pmb{x} \in \mathbb{R}^N.$$
Thus $(i)$ of Proposition \ref{prop:estimation} is established. 

For each fixed $\pmb{x} \in \mathbb{R}^N$, let 
\begin{eqnarray}\label{adm_x}
\mathcal{U}_{\pmb{x}} = \{\pmb{\xi} \in \mathcal{U}: J(\pmb{x},\pmb{\xi}) \leq J(\pmb{x};\pmb{0})\}.
\end{eqnarray}
By Assumption \textbf{A1}, 
\begin{eqnarray} \label{bound1}
\mathbb{E}_{\pmb{x}} \int_0^{\infty} e^{-\rho t} \|\pmb{X}_t\|^2 dt \le K(1+\|\pmb{x}\|^2), \qquad \forall \pmb{x} \in \mathbb{R}^N, \pmb{\xi} \in  \mathcal{U}_{\pmb{x}}.
\end{eqnarray}
For $\pmb{\xi} \in \mathcal{U}_{\pmb{x}}$, it is easy to verify
\begin{eqnarray} \label{bound2}
\mathbb{E}_{\pmb{x}}\int_0^{\infty} e^{-\rho t} \|\pmb{\xi}_t\|^2 dt \leq K(1+\|\pmb{x}\|^2),
\end{eqnarray}
and
\begin{eqnarray*}
|v(\pmb{x})-v(\pmb{x}')| \leq \sup \left\{ |J(\pmb{x};\pmb{\xi}) - J(\pmb{x}';\pmb{\xi})| : \pmb{\xi} \in \mathcal{U}_{\pmb{x}} \cup\mathcal{U}_{\pmb{x}'}  \right\}, \forall \pmb{x},\pmb{x}^{\prime} \in \mathbb{R}^N.
\end{eqnarray*}
Meanwhile,
$$ |J(\pmb{x};\pmb{\xi}) - J(\pmb{x}';\pmb{\xi})|  \leq \mathbb{E} \int_0^{\infty} e^{-\rho t} |H(\pmb{X}_t^{\pmb{x}}) - H(\pmb{X}_t^{\pmb{x}'})|dt.$$
Statement (ii) for $v$ follows by Assumption \textbf{ A2}, along with the facts that $\pmb{X}_t^{\pmb{x}}-\pmb{X}_t^{\pmb{x}^{\prime}}=\pmb{x}-\pmb{x}^{\prime}$ and  that for any $\pmb{\xi} \in \mathcal{U}_{\pmb{x}}\cup \mathcal{U}_{\pmb{x}^{\prime}}$,
\begin{eqnarray}\label{eqn1}
\mathbb{E}_{\pmb{x}} \int_0^{\infty} e^{-\rho t}\|\pmb{X}^{\pmb{x}}_t\| dt &\leq& K(1+\|\pmb{x}\|+\|\pmb{x}^{\prime}\|),\\
\mathbb{E}_{\pmb{x}^{\prime}} \int_0^{\infty} e^{-\rho t}\|\pmb{X}^{\pmb{x}^{\prime}}_t\| dt &\leq& K(1+\|\pmb{x}\|+\|\pmb{x}^{\prime}\|).\nonumber
\end{eqnarray}
In fact, if $\pmb{\xi} \in \mathcal{U}_{\pmb{x}}$, \eqref{eqn1} follows immediately from \eqref{bound2} by the H\"older inequality. Meanwhile, if $\pmb{\xi} \in \mathcal{U}_{\pmb{x}^{\prime}}$, \eqref{eqn1} holds because
 $$\|\pmb{X}_t^{\pmb{x}}\| \leq \|\pmb{X}_t^{\pmb{x}^{\prime}}\|+\|\pmb{x}-\pmb{x}^{\prime}\| \leq \|\pmb{X}_t^{\pmb{x}^{\prime}}\|+\|\pmb{x}\|+\|\pmb{x}^{\prime}\|.$$
\end{proof}

Next, we establish the viscosity property of the value function in the following sense.
\begin{definition}[Continuous viscosity solution]
The value function $v$ for problem (\ref{centrol_controller}) is a continuous viscosity solution to \eqref{po_hjb} on $\mathbb{R}^N$ if 
\begin{itemize}
\item 
$\forall \pmb{x}_0 \in \mathbb{R}^N$, $\forall \phi \in \mathcal{C}^2(\mathbb{R}^N)$ such that $\pmb{x}_0$ is a local minimum of $(v-\phi) (\pmb{x})$ with $v(\pmb{x}_0)=\phi(\pmb{x}_0)$,
$$\max \{\rho \phi-\mathcal{L}\phi-H(\pmb{x}),\beta(\nabla \phi)-1\} \ge 0.$$
\item 
$\forall \pmb{x}_0 \in \mathbb{R}^N$, $\forall \phi \in \mathcal{C}^2(\mathbb{R}^N)$ such that $\pmb{x}_0$ is a local maximum of $(v-\phi) (\pmb{x})$ with $v(\pmb{x}_0)=\phi(\pmb{x}_0)$,
$$\max \{\rho \phi-\mathcal{L}\phi-H(\pmb{x}),\beta(\nabla \phi)-1\} \le 0.$$
\end{itemize}
\end{definition}

\begin{theorem}[Viscosity solution] \label{thm:viscosity} Under Assumptions {\bf A1} - {\bf A3}, the value function $v$ to the control problem \eqref{centrol_controller}
is convex and a continuous viscosity solution  of the HJB equation \eqref{po_hjb}.
\end{theorem}
 \begin{proof}
The convexity of $v$ follows from the joint convexity of $J(\pmb{x};\pmb{\xi})$ in the following sense:
\begin{eqnarray}\label{convexity_J}
J(\theta \pmb{x}_1 +(1-\theta) \pmb{x}_2;\theta \pmb{\xi}_1 +(1-\theta) \pmb{\xi}_2) \leq \theta J (\pmb{x}_1; \pmb{\xi}_1) + (1- \theta) J (\pmb{x}_2 ; \pmb{\xi}_2),
\end{eqnarray}
holds for any $\pmb{x}_1,\pmb{x}_2\in \mathbb{R}^N$ and any $\pmb{\xi}_1,\pmb{\xi}_2 \in \mathcal{U}_N$.  To see this, $\pmb{X}_t^{\pmb{x}}$ depends linearly on $(\pmb{x},\pmb{\xi})$, and  both the set $\mathcal{U}_N$ and the function $H$ are convex. {  Under Assumption {\bf A1} - {\bf A3}, the existence of the optimal control to problem \eqref{centrol_controller}
follows from Theorem 4.5 and Corollary 4.11 in \citep{MT1989}.
The convexity of $v$ is verified as below, which follows the standard argument \citep{GP2005,WCM1994}. Take $\pmb{\xi}^*_1 = \arg \min_{\pmb{\xi}\in \mathcal{U}_N}J(\pmb{x}_1;\,\pmb{\xi})$ and 
$\pmb{\xi}^*_2 = \arg \min_{\pmb{\xi}\in \mathcal{U}_N}J(\pmb{x}_2\,;\,\pmb{\xi})$, then  by definition,
\begin{eqnarray}\label{convexity_J2}
 \theta J (\pmb{x}_1; \pmb{\xi}^*_1) + (1- \theta) J (\pmb{x}_2 ; \pmb{\xi}_2^*) = \theta v( \pmb{x}_1) +(1-\theta)v(  \pmb{x}_2).
\end{eqnarray}
Note that $\theta  \pmb{\xi}^*_1 +(1-\theta) \pmb{\xi}^*_2 \in \mathcal{U}_N$ by the convexity of $\mathcal{U}_N$, therefore
\begin{eqnarray}\label{convexity_J3}
v(\theta \pmb{x}_1 +(1-\theta) \pmb{x}_2) = \min_{\pmb{\xi}\in \mathcal{U}_N} J(\theta \pmb{x}_1 +(1-\theta) \pmb{x}_2;\pmb{\xi}) 
\leq J(\theta \pmb{x}_1 +(1-\theta) \pmb{x}_2;\theta \pmb{\xi}^*_1 +(1-\theta) \pmb{\xi}^*_2).
\end{eqnarray}
Combining  \eqref{convexity_J}, \eqref{convexity_J2}, and \eqref{convexity_J3}, 
\[
v(\theta \pmb{x}_1 +(1-\theta) \pmb{x}_2) \leq \theta v( \pmb{x}_1) +(1-\theta)v(  \pmb{x}_2).
\]
}

We now show that $v$ is both a viscosity super-solution and a viscosity sub-solution to the HJB equation \eqref{po_hjb}.\\

\noindent \underline{Sub-solution.}
Consider the following controls:  $  \xi_t^{i,-}=0$ and
\[
\xi_t^{i,+}=\left\{
                \begin{array}{ll}
                0,\quad t=0-,\\
                \eta^{i,+}, \quad  t \ge 0,
                                \end{array}
              \right.
              \]         
    where $0\le \eta^{i,+} \leq \epsilon$. Define the exit time 
    \[
    \tau_{\epsilon} :=\inf \{t \ge 0, \pmb{X}_t \notin \bar{B}_{\epsilon}(\pmb{x}_0)\}.
    \]        
    Note that $\pmb{X}$ has at most one jump at $t=0$ and is continuous on $[0,\tau_{\epsilon})$. 
 { The dynamic programming principle states that for any $\pmb{x}\in \mathbb{R}^N$,
 \begin{eqnarray}\label{eq:dpp}
 v(\pmb{x}) =\inf_{\pmb{\xi}\in \mathcal{U}_N}\mathbb{E}_{\pmb{x}}\left[\int_0^{\theta}\left( e^{-\rho t}H(\pmb{X}_t)dt +\sum_{i=1}^N L_iK_i^{+}d\xi_t^{i,+}+\sum_{i=1}^N L_iK_i^{-}d\xi_t^{i,-}\right)+e^{-\rho \theta}v(\pmb{X}_{\theta}) \right],     
 \end{eqnarray}
 for any $\theta\in \mathcal{F}$ possibly depending on $\pmb{\xi}$ in the infimum of \eqref{eq:dpp}. Therefore,
 }  
    \begin{eqnarray}\label{vscosity_dpp}
    \phi(\pmb{x}_0) = v(\pmb{x}_0) \leq \mathbb{E}_{\pmb{x}_0}\int_0^{\tau_{\epsilon}\wedge h} e^{-\rho t}\left[H(\pmb{X}_t)dt +\sum_{i=1}^N L_iK_i^{+}d\xi_t^{i,+}\right]+\mathbb{E}_{\pmb{x}_0}\left[e^{-\rho (\tau_{\epsilon}\wedge h)}\phi(\pmb{X}_{\tau_{\epsilon}\wedge h}) \right].
    \end{eqnarray}  
    Applying It\^{o}'s formula to the process $e^{-\rho t} \phi(\pmb{X}_t)$ between $0$ and $\tau_{\epsilon}\wedge h$, and taking expectation, we obtain
    \begin{eqnarray}\label{viscosity_ito}
    \mathbb{E}_{\pmb{x}_0}\left[ e^{-\rho (\tau_{\epsilon}\wedge h)}\phi(\pmb{X}_{\tau_{\epsilon}\wedge h})\right] = \phi(\pmb{x}_0)&+&\mathbb{E}_{\pmb{x}_0}\left[ \int_0^{
    \tau_{\epsilon}\wedge h}e^{-\rho t} (-\rho \phi + \mathcal{L}\phi)(\pmb{X}_t)dt\right]\nonumber\\
    &+&\mathbb{E}_{\pmb{x}_0}\left[\sum_{0\leq t \leq \tau_{\epsilon}\wedge h} [\phi(\pmb{X}_t)-\phi(\pmb{X}_{t-})]\right].
    \end{eqnarray}
    Combining \eqref{vscosity_dpp} and \eqref{viscosity_ito}, we have
    \begin{eqnarray}\label{viscosity_inter}
    \mathbb{E}_{\pmb{x}_0}\left[ \int_0^{\tau_{\epsilon}\wedge h} e^{-\rho t} (\rho \phi -\mathcal{L}\phi- H)(\pmb{X}_t)dt\right]&-&\mathbb{E}_{\pmb{x}_0}\left[ \int_0^{\tau_{\epsilon}\wedge h} e^{-\rho t}(\sum_{i=1}^N L_iK_i^+d \xi_t^{i,+})\right]\nonumber\\
    &-&\mathbb{E}_{\pmb{x}_0}\left[\sum_{0\leq t \leq \tau_{\epsilon}\wedge h}\phi(\pmb{X}_t)-\phi(\pmb{X}_{t-})\right]\leq 0.
    \end{eqnarray}
    \begin{itemize}
    \item Taking first $\eta^{i,+}=0$ for all $i=1,2,\cdots,N$, i.e.,  $\xi^{i,+}=\xi^{i,-}=0$, we see that $\pmb{X}$ is continuous and that only the first term in the LHS of \eqref{viscosity_inter} is nonzero. Dividing the above inequality \eqref{viscosity_inter} by $h$ and letting $h \rightarrow 0$, then by the dominated convergence theorem,
    \[
    \rho  \phi(\pmb{x}_0)-\mathcal{L}\phi(\pmb{x}_0)-H(\pmb{x}_0)\leq 0.
    \]
    \item Now, by taking $\eta^{i,+}>0$ and $\eta^{j,+}=0$ for $j \neq i$ in \eqref{viscosity_inter}, and noting that $\xi^{i,+}$ and $\pmb{X}$ jump only at $t=0$ with size $\eta^{i,+}$, we get
    \[
  \mathbb{E}_{\pmb{x}_0} \left[ \int_0^{\tau_{\epsilon}\wedge h} e^{-\rho t} (\rho \phi -\mathcal{L}\phi- H)(\pmb{X}_t)dt\right]-L_iK_i^+\eta^{i,+}-\phi(\pmb{x}_0+\eta^{i,+}\pmb{e}_i)+\phi(\pmb{x}_0) \leq 0.
    \]
    Taking $h \rightarrow 0$, then dividing by $\eta^{i,+}$ and letting $\eta \rightarrow 0$, we have
    \[-L_iK_i^+\leq \partial_{x^i}\phi(\pmb{x}).\]
    \item Meanwhile, taking an admissible control such that $\xi^{i,+}=0$ and 
    \[
\xi_t^{i,-}=\left\{
                \begin{array}{ll}
                0,\quad t=0-,\\
                \eta^{i,-}, \quad t \ge 0,
                                \end{array}
              \right.
              \]         
    where $0\le \eta^{i,-} \leq \epsilon$. By a similar argument, we have
    \[
 \forall i=1,2,\cdots,N,\quad \partial_{x^i}\phi(\pmb{x})\leq L_i K_i^-.
    \]
    \end{itemize}
    This proves the sub-solution viscosity property
    $$\max \{\rho \phi-\mathcal{L}\phi-H(x),\beta(\nabla \phi)-1\} \le 0.$$
\underline{Super-solution.}
This part is proved by contradiction. Suppose otherwise. Then there exist $\pmb{x}_0 \in \mathbb{R}^N$, $\epsilon >0$, $\phi(\pmb{x})\in \mathcal{C}^2(\mathbb{R}^N)$ with $\phi(\pmb{x}_0)=v(\pmb{x}_0)$,  $v \geq \phi$ in $\bar{B}_{\epsilon}(\pmb{x}_0)$ and $\nu>0$ such that for all $\pmb{x} \in \bar{B}_{\epsilon}(\pmb{x}_0)$,
\begin{eqnarray}\label{viscosity_contra_1}
\rho \phi(\pmb{x}_0)-\mathcal{L}\phi(\pmb{x}_0)-H(\pmb{x}_0) \le -\nu,
\end{eqnarray}
and for all $i=1,2,\cdots,N$,
\begin{eqnarray}\label{viscosity_contra_2}
-L_i K_i^{+}+\nu \leq \partial_{x^i}\phi \leq L_i K_i^{-}-\nu.
\end{eqnarray}

Given any admissible control $\pmb{\xi}$, consider the exit time $\tau_{\epsilon}=\inf \{t \ge 0, \pmb{X}_t \notin \bar{B}_{\epsilon}(\pmb{x}_0)\}$.
Applying It\^o's formula \citep[Theorem 21]{Meyer76} to $e^{-\rho t}\phi(\pmb{x})$ and any semi-martingale $\{\pmb{X}_t\}_{t \ge 0}$ under admissible control $(\xi^{i,+},\xi^{i,-})_{i=1}^N$ leads to
\begin{eqnarray*}
\mathbb{E}_{\pmb{x}_0}\left[e^{-\rho \tau_{\epsilon}}\phi(\pmb{X}_{\tau_{\epsilon}-}) \right] =&& \phi(\pmb{x}_0)+\mathbb{E}_{\pmb{x}_0} \left[\int_0^{\tau_{\epsilon}}e^{-\rho t}(-\rho \phi +\mathcal{L}\phi)(\pmb{X}_t)dt \right]\\
&+&\mathbb{E}_{\pmb{x}_0} \left[ \int_0^{\tau_{\epsilon}} e^{-\rho t} \sum_{i=1}^N \partial_{x^i}\phi(\pmb{X}_t)[(d\xi_t^{i,+})^c-(d\xi_t^{i,-})^c]\right]\\
&+&\mathbb{E}_{\pmb{x}_0} \left[\sum_{0\leq t <\tau_{\epsilon}} e^{-\rho t} [\phi(\pmb{X}_t)-\phi(\pmb{X}_{t-})]\right].
\end{eqnarray*}
Note that for all $0\leq t <\tau_{\epsilon}$, $\pmb{X}_t \in \bar{B}_{\epsilon}(\pmb{x}_0)$. Then, by \eqref{viscosity_contra_1}, and noting that $\Delta X_t^i = \Delta \xi_t^{i,+}-\Delta \xi_t^{i,-}$, we have for all $0\leq t <\tau_{\epsilon}$,
\[
\phi(\pmb{X}_t)-\phi(\pmb{X}_{t-})=\sum_{i=1}^N \Delta X_t^i \int_0^1 \partial_{x^i}\phi(\pmb{X}_t+z\Delta \pmb{X}_t)dz \leq \sum_{i=1}^N \left[(L_iK_i^--\nu)\Delta \xi_t^{i,+} +(L_iK_i^+-\nu)\Delta \xi_t^{i,-}\right].
\]
Similarly,
\begin{eqnarray}\label{viscosity_fc}
\phi(\pmb{X}_t)-\phi(\pmb{X}_{t-})\geq \sum_{i=1}^N \left[ -(L_iK_i^--\nu)\Delta \xi_t^{i,-} -(L_iK_i^+-\nu)\Delta \xi_t^{i,+}\right].
\end{eqnarray}
In light of relations \eqref{viscosity_contra_1}-\eqref{viscosity_fc},
\begin{eqnarray}\label{sup-big}
\mathbb{E}_{\pmb{x}_0}\left[e^{-\rho \tau_{\epsilon}}\phi(\pmb{X}_{\tau_{\epsilon}-}) \right] \geq \phi(\pmb{x}_0)&+&\mathbb{E}_{\pmb{x}_0} \left[\int_0^{\tau_{\epsilon}}e^{-\rho t}(-H+\nu)(\pmb{X}_t)dt \right]\nonumber\\
&+&\mathbb{E}_{\pmb{x}_0} \left[ \int_0^{\tau_{\epsilon}-} e^{-\rho t} \sum_{i=1}^N - (L_iK_i^{+}-\nu)d\xi_t^{i,+}- (L_iK_i^{-}-\nu)d\xi_t^{i,-}\right]\nonumber\\
=\phi(\pmb{x}_0)&-& \mathbb{E}_{\pmb{x}_0} \int_0^{\tau_{\epsilon}}e^{-\rho t}\left[H(\pmb{X}_t)dt +\sum_{i=1}^N L_iK_i^{+}d\xi_t^{i,+}+\sum_{i=1}^N L_iK_i^{-}d\xi_t^{i,-}\right]\nonumber\\
&+&\sum_{i=1}^N\left(\mathbb{E}_{\pmb{x}_0}\left[e^{-\rho \tau_{\epsilon}}L_iK_i^{+}\Delta \xi^{i,+}_{\tau_{\epsilon}}\right]+\mathbb{E}_{\pmb{x}_0}\left[e^{-\rho \tau_{\epsilon}}L_iK_i^{-}\Delta \xi^{i,-}_{\tau_{\epsilon}}\right]\right)\nonumber\\
&+&\nu \left\{\mathbb{E}_{\pmb{x}_0}\left[\int_0^{\tau_{\epsilon}}e^{-\rho t}dt \right]+\mathbb{E}_{\pmb{x}_0}\left[\int_0^{\tau_{\epsilon}-}e^{-\rho t}(d\xi_t^{i,+}+d\xi_t^{i,-})\right]\right\}.
\end{eqnarray}
Note that $\pmb{X}_{\tau_{\epsilon}-}\in \overline{B}_{\epsilon}(\pmb{x}_0)$, $\pmb{X}_{\tau_{\epsilon}}$ is either on the boundary $\partial B_{\epsilon}(\pmb{x}_0)$ or out of $\overline{B}_{\epsilon}(\pmb{x}_0)$. However, there is some random variable $\delta$ valued in $[0,1]$ such that
\[
\pmb{x}_{\delta} = \pmb{X}_{\tau_{\epsilon}-}+\delta \Delta \pmb{X}_{\tau_{\epsilon}} = \pmb{X}_{\tau_{\epsilon}-}+ \delta (\Delta \pmb{\xi}^{+}_{\tau_{\epsilon}}-\Delta \pmb{\xi}^{-}_{\tau_{\epsilon}})\in \partial \bar{B}_{\epsilon}(\pmb{x}_0).
\]
Then similar to \eqref{viscosity_fc}, we have
\begin{eqnarray}\label{sup-1}
\phi(\pmb{x}_{\delta})-\phi(\pmb{X}_{\tau_{\epsilon}-})\geq \delta \sum_{i=1}^N \left[ -(L_iK_i^--\nu)\Delta \xi_{\tau_{\epsilon}}^{i,-} -(L_iK_i^+-\nu)\Delta \xi_{\tau_{\epsilon}}^{i,+}\right].
\end{eqnarray}
Note that $\pmb{X}_{\tau_{\epsilon}}=\pmb{x}_{\delta}+(1-\delta)(\Delta \pmb{\xi}^{+}_{\tau_{\epsilon}}-\Delta \pmb{\xi}^{-}_{\tau_{\epsilon}})$, thus
\begin{eqnarray}\label{sup-2}
v(\pmb{x}_{\delta}) \leq (1-\delta)\sum_{i=1}^N\left(L_iK_i^{+}\Delta \xi_{\tau_{\epsilon}}^{i,+}+L_iK_i^{-}\Delta \xi_{\tau_{\epsilon}}^{i,-}\right)+v(\pmb{X}_{\tau_{\epsilon}}).
\end{eqnarray}
Recalling that $v(\pmb{x}_{\delta})\geq \phi(\pmb{x}_{\delta})$, inequalities \eqref{sup-1}-\eqref{sup-2} imply
\[
(1-\delta)\sum_{i=1}^N\left(L_iK_i^{+}\Delta \xi_{\tau_{\epsilon}}^{i,+}+L_iK_i^{-}\Delta \xi_{\tau_{\epsilon}}^{i,-}\right)+v(\pmb{X}_{\tau_{\epsilon}})\ge\phi(\pmb{X}_{\tau_{\epsilon}-})+ \delta \sum_{i=1}^N \left[ -(L_iK_i^--\nu)\Delta \xi_{\tau_{\epsilon}}^{i,-} -(L_iK_i^+-\nu)\Delta \xi_{\tau_{\epsilon}}^{i,+}\right].
\]
Therefore,
\[
\sum_{i=1}^N\Big((L_iK_i^{+}-\delta\nu)\Delta \xi_{\tau_{\epsilon}}^{i,+}+(L_iK_i^{-}-\delta\nu)\Delta \xi_{\tau_{\epsilon}}^{i,-}\Big) +v(\pmb{X}_{\tau_{\epsilon}})\ge\phi(\pmb{X}_{\tau_{\epsilon}-}).
\]
Plugging the last inequality into \eqref{sup-big}, along with $\phi(\pmb{x}_0)=v(\pmb{x}_0)$, yields
\begin{eqnarray*}
&&\mathbb{E}_{\pmb{x}_0}\quad e^{-{\rho}\tau_{\epsilon}}\left[\sum_{i=1}^N\left((L_iK_i^{+}-\delta\nu)\Delta \xi_{\tau_{\epsilon}}^{i,+}+(L_iK_i^{-}-\delta\nu)\Delta \xi_{\tau_{\epsilon}}^{i,-}\right) +v(\pmb{X}_{\tau_{\epsilon}})\right]\nonumber\\
&&\ge v(\pmb{x}_0)-\mathbb{E}_{\pmb{x}_0} \int_0^{\tau_{\epsilon}}e^{-\rho t}\left[H(\pmb{X}_t)dt +\sum_{i=1}^N L_iK_i^{+}d\xi_t^{i,+}+\sum_{i=1}^N L_iK_i^{-}d\xi_t^{i,-}\right]\nonumber\\
&+&\sum_{i=1}^N\left(\mathbb{E}_{\pmb{x}_0}\left[e^{-\rho \tau_{\epsilon}}L_iK_i^{+}\Delta \xi^{i,+}_{\tau_{\epsilon}}\right]+\mathbb{E}_{\pmb{x}_0}\left[e^{-\rho \tau_{\epsilon}}L_iK_i^{-}\Delta \xi^{i,-}_{\tau_{\epsilon}}\right]\right)\nonumber\\
&+&\nu \left\{\mathbb{E}_{\pmb{x}_0}\left[\int_0^{\tau_{\epsilon}}e^{-\rho t}dt \right]+\mathbb{E}_{\pmb{x}_0}\left[\int_0^{\tau_{\epsilon}-}e^{-\rho t}(d\xi_t^{i,+}+d\xi_t^{i,-})\right]\right\}.
\end{eqnarray*}
Hence
\begin{eqnarray*}
&&\mathbb{E}_{\pmb{x}_0}e^{-{\rho}\tau_{\epsilon}} v(\pmb{X}_{\tau_{\epsilon}}) +\mathbb{E}_{\pmb{x}_0} \int_0^{\tau_{\epsilon}}e^{-\rho t}\left[H(\pmb{X}_t)dt +\sum_{i=1}^N L_iK_i^{+}d\xi_t^{i,+}+\sum_{i=1}^N L_iK_i^{-}d\xi_t^{i,-}\right]\\
&\ge& v(\pmb{x}_0)+\nu \left\{\mathbb{E}_{\pmb{x}_0}\left[\int_0^{\tau_{\epsilon}}e^{-\rho t}dt \right]+\mathbb{E}_{\pmb{x}_0}\left[\int_0^{\tau_{\epsilon}-}e^{-\rho t}(d\xi_t^{i,+}+d\xi_t^{i,-})\right]+\delta\mathbb{E}_{\pmb{x}_0}\left[e^{-{\rho}\tau_{\epsilon}}\Delta \xi_{\tau_{\epsilon}}^{i,+}+e^{-{\rho}\tau_{\epsilon}}\Delta \xi_{\tau_{\epsilon}}^{i,-}\right]\right\}.
\end{eqnarray*}

We now claim that there exists a constant $g_0>0$ such that for all admissible control $\pmb{\xi}$,
\begin{eqnarray}\label{g0}
\mathbb{E}_{\pmb{x}_0}\left[\int_0^{\tau_{\epsilon}}e^{-\rho t}dt \right]+\mathbb{E}_{\pmb{x}_0}\left[\int_0^{\tau_{\epsilon}-}e^{-\rho t}(d\xi_t^{i,+}+d\xi_t^{i,-})\right]+\delta\mathbb{E}_{\pmb{x}_0}\left[e^{-{\rho}\tau_{\epsilon}}\Delta \xi_{\tau_{\epsilon}}^{i,+}+e^{-{\rho}\tau_{\epsilon}}\Delta \xi_{\tau_{\epsilon}}^{i,-}\right] \ge g_0.
\end{eqnarray}

\noindent Indeed, one can always find some constant $G_0$ such that the $\mathcal{C}^2$ function
\[
\psi(\pmb{x})=G_0((\pmb{x}-\pmb{x}_0)^2-\epsilon^2)
\]
satisfies
\begin{eqnarray*}
\begin{cases}
&\min_i\{\rho\psi-\mathcal{L}\psi+1,1-|\partial_{x^i}\psi|\} \ge 0, \,\,\mbox{ on }\,\, \overline{B}_{\epsilon}(\pmb{x}_0),\\
&\psi=0 \,\, \mbox{on} \,\, \partial \overline{B}_{\epsilon}(\pmb{x}_0).
\end{cases}
\end{eqnarray*}

Applying It\^o's formula  to $e^{-\rho t}\psi(\pmb{x})$ and any semi-martingale $\{\pmb{X}_t\}_{t \ge 0}$ under admissible control $(\xi^{i,+},\xi^{i,-})_{i=1}^N$ leads to
\begin{eqnarray}\label{sup-5}
\mathbb{E}_{\pmb{x}_0}\left[ e^{-\rho \tau_{\epsilon}}\psi(\pmb{X}_{\tau_{\epsilon}-})\right] \leq \psi(\pmb{x}_0)+\mathbb{E}_{\pmb{x}_0}\left[ \int_0^{\tau_{\epsilon}}e^{-\rho t}dt\right]+\sum_{i=1}^N\mathbb{E}_{\pmb{x}_0}\left[ \int_0^{\tau_{\epsilon}-}e^{-\rho t}(d \xi_t^{i,+}+d \xi_t^{i,-})\right].
\end{eqnarray}
Since $\partial_{x^i}\psi(\pmb{x}_0)\geq -1$ for all $i=1,2,\cdots,N$, 
\[
\psi(\pmb{X}_{\tau_{\epsilon}-})-\psi(\pmb{x}_{\delta }) \geq -\nabla \psi(\pmb{X}_{\tau_{\epsilon}-}-\pmb{x}_{\delta })\geq -\delta \sum_{i=1}^N\Delta \xi^{i,-}_{\tau_{\epsilon}},
\]
which, combined with \eqref{sup-5}, yields
\begin{eqnarray*}
&&\mathbb{E}_{\pmb{x}_0} \left[ \int_0^{\tau_{\epsilon}}e^{-\rho t}dt\right]+\sum_{i=1}^N\mathbb{E}_{\pmb{x}_0} \left[\int_0^{\tau_{\epsilon}-}(d\xi_t^{i,+}+d\xi_t^{i,-})\right]+\mathbb{E}_{\pmb{x}_0}\left[ e^{-\rho \tau_{\epsilon}}\delta\sum_{i=1}^N\Delta \xi^{i,-}_{\tau_{\epsilon}}\right] \\
&\ge & \mathbb{E}_{\pmb{x}_0} \left[e^{-\rho \tau_{\epsilon}}\psi(\pmb{x}_{\delta})\right] -\psi(\pmb{x}_0)=G_0 \epsilon^2.
\end{eqnarray*}
Hence \eqref{g0} holds with $g_0=G_0\epsilon^2$.
 \end{proof}

We can further show that  the value function is a   $\mathcal{W}^{2,\infty}_{\text{loc}}(\mathbb{R}^N)$ solution to the HJB equation  \eqref{po_hjb}. 
\begin{theorem}[Regularity]\label{prop:regularity_uniqueness}
Under Assumption {\bf{A3}}, the value function $v$ defined by \eqref{centrol_controller}   belongs to $\mathcal{W}^{2,\infty}_{{\rm loc}}(\mathbb{R}^N)$ and is a  solution to the HJB equation \eqref{po_hjb}. In addition,
\begin{eqnarray}\label{second_order_derivative_estimation}
    0 \leq  \partial ^2 _{\pmb{z}} v(\pmb{x}) \leq C, \,\,a.e.\,\, \text{ for } \pmb{x} \in \mathbb{R}^N,
\end{eqnarray}
with $C>0$ defined in Assumption {\bf A3}. Furthermore the continuation region 
\begin{eqnarray}\label{set_PO}
\mathcal{C}_N := \left\{\pmb{x} \given \beta(\nabla v(\pmb{x}))<1\right\}
\end{eqnarray}
is bounded and non-empty. In addition, we have $v\in \mathcal{C}^{4,\alpha}(\mathcal{C}_N)$.
\end{theorem}
Note that  $\mathcal{W}_{\rm loc}^{2,\infty}(\mathbb{R}^N)\subset \mathcal{C}^1(\mathbb{R}^N)$ by the Sobolev embedding (see Corollary 9.15 in Chapter 9 from \cite{brezis2010functional}).
{
\begin{remark}[Uniqueness]{\em
Our primary goal  is to identify and characterize Pareto optimal policies. To this end, it suffices to show that the value function of \eqref{centrol_controller} is in $\mathcal{W}^{2,\infty}_{{\rm loc}}(\mathbb{R}^N)$ and a convex solution to the HJB equation \eqref{po_hjb}. Uniqueness of the HJB solution, although not essential,  can be established by a verification argument  as discussed in Appendix \ref{app:verification}: the regularity, the convexity, and the bounded second-order derivative  of the value function allow to  apply an It\^o-Tanaka-Meyer Formula. }
\end{remark}
}
\begin{proof}
To prove \eqref{second_order_derivative_estimation}, let $\Delta_i(\eta):=(0,\cdots,0, \eta,0,\cdots,0)$ be the N-dimensional row vector with  the $i$-th entry being $\eta$ for $i=1,2,\cdots,N$. For any function $F: \mathbb{R}^N \rightarrow \mathbb{R}$, define the second difference of $F$ in the $x^i$ direction by
\begin{eqnarray} \label{second_order_difference}
 \delta^2_i \left(F,\pmb{x},\eta \right) := F\left( \pmb{x} + \Delta_i(\eta) \right) + F\left(\pmb{x}-\Delta_i(\eta)\right) - 2 F(\pmb{x}).
 \end{eqnarray}
 It is easy to check 
\begin{eqnarray} \label{2_diff_v}
\delta_i^2 (v,\pmb{x},\eta) \leq \sup \{\delta_i^2(J \left(\cdot;\pmb{\xi} \right),\pmb{x},\eta): \pmb{\xi} \in \mathcal{U}_{\pmb{x}}\}.
\end{eqnarray}
Since $H \in \mathcal{C}^2(\mathbb{R}^N)$,  for $\pmb{x} \in \mathbb{R}^N$,
\begin{eqnarray}\label{2_diff_h}
\delta_i^2 (H,\pmb{x},\eta) = (\eta)^2 \int_0^1 \int_{-\lambda}^{\lambda} \partial^2_{x^i} H (x^1,\ldots, x^i +\mu \eta, \ldots, x^N) d\mu d\lambda.
\end{eqnarray}
By Assumption \textbf{A3}, 
\begin{eqnarray} \label{2_diff_h_bound}
\delta_i^2 (H,\pmb{x},\eta) \leq C\,\eta^2 \int_0^1 \int_{-\lambda}^{\lambda} d\mu d\lambda = \eta^2  C.
\end{eqnarray}
Hence
\begin{eqnarray}\label{2_diff_v_2}
0 \leq \delta_i^2 (v,\pmb{x},\eta) \leq C  \eta^2, \,\,  \pmb{x} \in \mathbb{R}^N, |\eta| \leq 1.
\end{eqnarray}
The lower bound of (\ref{2_diff_v_2}) follows from the convexity of  $v$ by  Theorem \ref{thm:viscosity}.

To prove $v \in \mathcal{W}_{loc}^{2 , \infty}$, let ${G\subseteq \mathbb{R}^N}$ be any open ball and let $\psi \in C^{\infty}_0(\mathbb{R}^N)$ be any test function {such that ${\rm supp}(\psi) \subset G$.}
According to \eqref{2_diff_v_2}, we have
\[
|\eta^{-2} \delta_i^2 (v,\pmb{x},\eta)| \leq C \text{ for } \pmb{x}\in G  \text{ and } |\eta| \leq 1.
\]
Therefore by Theorem 1.1.2 in \citep{evans1990weak}, there is a sequence $\eta_k \rightarrow 0+$ as $k \rightarrow \infty$ such that, {denoting by $g_k(\pmb{x}):=\eta_k^{-2} \delta_i^2 (v,\pmb{x},\eta_k)$}, we have $g_k(\pmb{x}) \rightarrow Q$ weakly in $L^p (G)$ for some $p$ with $1<p < \infty$. It is then easy to see that 
\begin{eqnarray} 
\int_{\mathbb{R}^N} \psi (\pmb{x}) Q(\pmb{x}) d\pmb{x} = \int_{\mathbb{R}^N} \partial^2_{x^i} \psi v (\pmb{x}) d\pmb{x}, \quad \forall \psi \in \mathcal{C}_0^{\infty} (G),
\end{eqnarray}
where $Q = \partial^2_{x^i} v$. The existence and local boundedness of  second order  derivatives is now immediate: for $k=1,2,\ldots,N$, let $\pmb{e}_k$ denote the unit vector in the direction of the positive $x_k$ axis; for any fixed $i \neq j$ with $1 \leq i,j \leq N$, let $\pmb{y}$ be a new coordinate whose axis points to the $\frac{\pmb{e}_i +\pmb{e}_j}{\sqrt{2}}$ direction, then $\partial^2_{x_ix_j} v = \partial^2_{\pmb{y}} v - \frac{1}{2}(\partial_{x^i}^2v + \partial_{x^j}^2 v)$.

 Since $|{\partial_{x^i}} v(\pmb{x})| \leq L_i \max\{K_i^+,K_i^-\}$ ($i=1,2,\cdots,N$) on $\mathbb{R}^N$ but $H$ grows at least quadratically by Assumption {\bf A3}, $\mathcal{C}_N$ must be bounded.

Finally, let $G$ be any open ball such that $\overline{G}\in \mathcal{C}_N$. By Theorem 6.13 in \citep{GT2015}, the Dirichlet problem in $G$,
\begin{eqnarray} 
   \left\{
                \begin{array}{ll}
        \rho \tilde{v} - \mathcal{L} \tilde{v}  = H(\pmb{x}) , \qquad &\forall x \in G,\\
\tilde{v} = v,  \qquad &\forall x \in \partial G, \label{direchlet_boundary}       
                               \end{array}
              \right.
\end{eqnarray}
has a solution $\tilde{v} \in \mathcal{C}^0(\bar{G}) \cap \mathcal{C}^{2,\alpha}(G)$. In particular, $\tilde{v}-v \in \mathcal{W}^{2,\infty}(G)$, therefore by  (\ref{direchlet_boundary}),  $\tilde{v}-v \in \mathcal{W}^{1,2}_0(G)$. By Theorem 8.9 in \citep{GT2015}, $v = \tilde{v}$ in $G$, thus $v \in \mathcal{C}^{2,\alpha}(G)$. 
By Theorem 6.17 in \citep{GT2015}, $v \in \mathcal{C}^{4,\alpha}(G)$ thus $v \in  \mathcal{C}^{4,\alpha}(\mathcal{C}_N)$ for all $\alpha \in (0,1)$.

\end{proof}



\begin{remark}{\rm
The proof of Theorem \ref{prop:regularity_uniqueness} is inspired by the approach in  \citep[Theorem 4.5]{SS1989} {and \citep[Theorem 3.1]{WCM1994}}. In \citep{SS1989}, the following HJB equation \eqref{po_hjb_comparison}
(See Eqn. (3.1) in \citep{SS1989}) has been studied for an $N$-dimensional control problem  
\begin{eqnarray}\label{po_hjb_comparison}
\max \left\{  \rho u - \mathcal{L} u -H(\pmb{x}) ,  \sqrt{\sum_{i=1}^N(\partial_{x^i} u)^2}-1  \right\}=0.
\end{eqnarray}
Comparing the gradient constraints in \eqref{po_hjb_comparison} with  \eqref{po_hjb}, it is clear that the operator $\beta$ in \eqref{po_hjb} is less regular than $\|\nabla u\|_2$ in (\ref{po_hjb_comparison}) as  $\|\nabla u(\cdot)\|_2$ has  smoother and gradual changes in the state space $\mathbb{R}^N$.   In contrast, $\beta$ in \eqref{po_hjb} involves a maximum operator as a result of game interactions.

 The   HJB equation \eqref{po_hjb} has appeared  in \cite{MT1989} for analyzing the convergence of  finite variation controls. To our best knowledge, our  characterization of the optimal control and regularity results are novel.
}
\end{remark}

\section{Pareto-optimal policies}\label{sec:po_strategy}
The regularity analysis of the value function for problem (\ref{centrol_controller})  enables us to establish  the existence and the uniqueness of its  optimal control, for any given weight $(L_1, \cdots, L_N)$ such that $L_i>0$ and $\sum_{i=1}^N L_i=1$ (Section \ref{subsec:optimal_policy}). The optimal control in  \eqref{centrol_controller} is then shown to lead to a Pareto-optimal policy for game (\ref{N-game}) (Theorem \ref{connection_1}) for each choice of weights $(L_1, \cdots, L_N)$.

\subsection{Optimal policy for the regulator}\label{subsec:optimal_policy}
To  ensure the  uniqueness of the Pareto-optimal policy, we impose the following assumption on  the value function $v$. 
\begin{enumerate}[font=\bfseries,leftmargin=3\parindent]
\item [A4.]  The diagonal dominates the row/column in the  Hessian  matrix $\nabla^2 v$. That is, 
\begin{eqnarray}
\partial^2_{x^i}v(\pmb{x}) {>} \sum_{j \neq i}\left|\partial^2_{x^ix^j}{v}(\pmb{x})\right|, \forall i,=1,2,\cdots,N \,\,\mbox{ and }\,\,\pmb{x}\in \overline{\mathcal{C}}_N.
\end{eqnarray}
\end{enumerate}
Note that a similar assumption has been used in \citep[Assumption 3]{GMS2010} to analyze Nash equilibrium strategies. This assumption guarantees that the reflection direction of the Skorokhod problem is not parallel to the boundary, and that the controlled dynamics are continuous when $\pmb{x}\in \mathcal{C}_N$. Assumption {\bf A4} can be  relaxed using  techniques of  \cite{kruk2000} to deal with possible  jumps  at the reflection boundary.

Given this additional assumption and the regularity of the value function,  we are now ready to  characterize   the Pareto-optimal policy to game \eqref{N-game}. 

We shall show that when  $\pmb{x} \in \overline{\mathcal{C}}_N$, the optimal policy may be constructed by  solving a sequence of Skorokhod problems with  piecewise $\mathcal{C}^1$ boundaries,  then  passing to the limit of this sequence of $\epsilon$-optimal policies.
We shall also  show that the reflection field of the Skorokhod problem can be extended to the entire state space under appropriate conditions, completing the construction of  the Pareto-optimal policy when $\pmb{x}$ is outside $\overline{\mathcal{C}}_N$.

\paragraph{Optimal policy for $\pmb{x} \in \overline{\mathcal{C}}_N$.}
 \label{sec:thm1_proof2}

First, recall the definition of the Skorokhod problem in \citep{ramanan2006}.
\begin{definition}[Skorokhod problem]\label{def:skorokhod}
Let G be an open domain in $\mathbb{R}^N$ with $S = \partial G$.
Let $\Gamma(\pmb{a},b)=\{\pmb{x}\in \mathbb{R}^N : |\pmb{x}-\pmb{a}|=b\}$. To each point $\pmb{x} \in S$, we will associate a set $\pmb{r}(\pmb{x}) \subset\Gamma(\pmb{0},1)$ called the {\it directions of reflection}. 
We say that a continuous process 
\begin{eqnarray}
\pmb{\xi}_t = \int_0^t \pmb{N}_s  d \eta_s,
\end{eqnarray}
with $\eta_t = \bigvee_{[0,t]} \pmb{\xi}$ {the total variation up to time $t$}, is a solution to a Skorokhod problem with data $(\pmb{x}+\pmb{\mu}t + \pmb{\sigma} \pmb{B}_t,G, \pmb{r},\pmb{x})$ if
\begin{enumerate}[font=\bfseries,leftmargin=3\parindent]
\item[(a)] $|\pmb{N}_t|=1$, $\eta_t$ is continuous and nondecreasing;
\item[(b)] the process $\pmb{X}_t = \pmb{x}+\pmb{\mu}t + \pmb{\sigma} \pmb{B}_t +\int_0^t \pmb{N}_s d \eta_s$ satisfies $\pmb{X}_t \in \overline{G}$, $0\leq t <\infty$, a.s; 
\item[(c)] for every $0 \leq t <\infty$, 
\begin{eqnarray*}
\eta_t = \int_0^t \textbf{1} _{\left(\pmb{X}_{s} \in \partial G, \pmb{N}_s \in\pmb{r}(\pmb{X}_{s}) \right)} d \eta_s.
\end{eqnarray*}
\end{enumerate}
\end{definition}

{Now let us introduce some notations for the Skorokhod problem associated with  the continuation region $\mathcal{C}_N$ defined in \eqref{set_PO}. By definition, 
\begin{eqnarray}\label{set_PO_redefine}
\mathcal{C}_N = \left\{\pmb{x} \given \beta(\nabla v(\pmb{x}))<1\right\}=\cap_{j=1}^{2N} G_{j},
\end{eqnarray}
where for $i=1,2,\cdots,N$,
\begin{eqnarray}\label{G}
G_{i} = \{\pmb{x}\,\,\vert\,\, \partial_{x^i} v(\pmb{x}) < { L_i}K_i^{-} \},\qquad
G_{i+N} = \{\pmb{x}\,\,\vert\,\, \partial_{x^i} v(\pmb{x}) > -{ L_i}K_{i}^{+} \}.
\end{eqnarray}
Denote $\mathcal{S}=\partial \mathcal{C}_{N}$ as the boundary of $ \mathcal{C}_{N}$, denote
  $I(\pmb{x})=\left\{j\,\, \vert \,\, \pmb{x} \notin G_j ,\,\, j=1,2,\cdots,2N\right\}$ as the boundary that $\pmb{x}$ lies on, and 
define the vector field $\gamma_j$ on each face $G_j$ as
\begin{eqnarray}\label{eq:gamma_direction}
\gamma_{i} =-\pmb{e}_i,\qquad
\gamma_{i+N} =\pmb{e}_i,
\end{eqnarray}
where $\pmb{e}_i=(0,\cdots,0,1,0,\cdots,0)$ with the $i^{th}$ component being $1$. Then the directions of the reflection is defined as
\begin{eqnarray}\label{cone_original}
\pmb{r}(\pmb{x}) = \left\{\sum_{j \in I(\pmb{x})} c_j \gamma_j (\pmb{x})\,\, : \,\, c_i \geq 0 \,\mbox{ and }\, \left\| \sum_{j \in I(\pmb{x})} c_j \gamma_j(\pmb{x}) \right\| =1  \right\}.
\end{eqnarray}

}


\begin{theorem}[$\epsilon$-policy]\label{thm:epsilon_policy} Assume Assumptions {\bf A1-A4} and $\pmb{x} \in \mathcal{C}_N$. For any $\epsilon >0$, 
{there exist $\mathcal{C}_{\epsilon}\subseteq \mathcal{C}_N$ non-empty and $\pmb{r}_{\epsilon}$ such that the}
 unique solution to the Skorokhod problem with data $(\pmb{x}+\pmb{\mu}t+ \pmb{\sigma}\pmb{B}_t,\mathcal{C}_{\epsilon}, \pmb{r}_{\epsilon},\pmb{x})$ is an $\epsilon$-optimal (admissible) policy of the control problem \eqref{centrol_controller} with
\begin{eqnarray}\label{epsilon_control}
\pmb{\xi}_t^{\epsilon} = \int_0^t \pmb{N}_s^{\epsilon} \cdot d \eta_s^{\epsilon},
\end{eqnarray}
and $\pmb{N}^{\epsilon}_s {\in} \pmb{r}_{\epsilon}(\pmb{X}^{\epsilon}_s)$ on $\mathcal{S}_{\epsilon}$, where $\pmb{X}^{\epsilon}_t = \pmb{x}+\pmb{\mu}t+ \pmb{\sigma}\pmb{B}_t +\pmb{\xi}_t^{\epsilon}$. That is, {
\[
(1-C_0\epsilon)J(\pmb{x}, \pmb{\xi}^{\epsilon}) \le  v(\pmb{x}),
\]}
for some constant $C_0$ that is independent of $\epsilon$. Here $ \mathcal{C}_{\epsilon} \subseteq \mathcal{C}$ has piecewise smooth boundaries.
\end{theorem}

\begin{proof} The proof consists of two steps. We first construct an approximation $\mathcal {C}_{\epsilon}$ of  $\mathcal{C}_N$ with piecewise smooth boundaries. Clearly, if $\partial \mathcal {C}_N$ {itself is piecewise smooth}, the $\mathcal {C}_{\epsilon} =\mathcal{C}_N$. We then show that the solution to the Skorokhod problem  with piecewise smooth boundary provides an $\epsilon$-policy to the  control problem \eqref{centrol_controller}.

\paragraph{Step 1: Skorokhod problem with piecewise smooth boundary.}
Let $\phi^{\delta}(\pmb{x}) \in {C}^{\infty}(\mathbb{R}^N,\mathbb{R}_+)$ be such that $\phi^{\delta}(\pmb{x}) = 0$ for $|\pmb{x}| \geq \delta$ and 
\begin{eqnarray}\label{smooth_function}
\int_{\mathbb{R}^N} \phi^{\delta} (\pmb{x}) d \pmb{x} =1.
\end{eqnarray}
Since $v \in \mathcal{W}_{loc}^{2,\infty}(\mathbb{R}^N)$, consider a regularization of $v(\pmb{x})$ via $\phi^{\epsilon}$ such that
\begin{eqnarray}\label{smooth_function-v}
v^{\delta}(\pmb{x}) := \phi^{\delta} * v(\pmb{x}).
\end{eqnarray}
{Similarly define $H^{\delta}(\pmb{x}) := \phi^{\delta} * H(\pmb{x})$.} The boundedness of $v$, $\nabla v$, $D^2 v$  on $B_R(0)$, with $\overline{\mathcal{C}}_N \subset B_{R-1}(\pmb{0})$, implies that $H^{\delta}$ and $v^{\delta}$ are bounded uniformly on $\overline{\mathcal{C}}_N$ for $\delta<1$, and
\begin{eqnarray*}
v^{\delta} \rightarrow v, \quad \nabla v^{\delta} \rightarrow \nabla v, \quad H^{\delta} \rightarrow H \quad \mbox{ uniformly in } \overline{\mathcal{C}}_N.
\end{eqnarray*}
Denote $K_{\max}=\max_{i=1,2,\cdots,N}\{{ L_i}K_i^{+},{ L_i}K_i^{-}\}$, $K_{\min}=\min_{i=1,2,\cdots,N}\{{ L_i}K_i^{+},{ L_i}K_i^{-}\}$ and recall $C$ in \eqref{second_order_derivative_estimation} such that  $0 \leq \partial ^2 _{ \pmb{z}} v(\pmb{x}) \leq C$ for any second order directional derivative $\partial ^2 _{ \pmb{z}}$.
Then, for any $\epsilon_k\in (0,\frac{1}{4})$, there exists $\delta_k:=\delta_k(\epsilon_k) \in \left(0,\frac{\epsilon_k K_{\min}}{C}\right)$ such that for all $\delta \in [0,\delta_k]$, $\|\nabla v^{\delta}-\nabla v\|_{1}<K_{\min}\epsilon_k$.
Take a non-negative and non-increasing sequence $\{\epsilon_k\}_k$ such that $\lim_{k \rightarrow \infty}\epsilon_k=0$. Denote $w^{\delta_k}(\pmb{x}) = \beta(\nabla v^{\delta_k}(\pmb{x}))$ and $\mathcal{C}_{\epsilon_k} := \{\pmb{x}\,\, \vert \,\,w^{\delta_k}(\pmb{x}) < 1-2\epsilon_k\}= \cap_{j=1}^{2N} G_{j}^{\epsilon_k} $, where $i=1,2,\cdots,N$,
\begin{eqnarray}\label{G_epsilon}
G_{i}^{\epsilon_k} &=& \{\pmb{x}\,\,\vert\,\, \partial_{x^i} v^{\delta_k}(\pmb{x}) < (1-2\epsilon_k){ L_i}K_i^{-} \},\nonumber\\
G_{i+N}^{\epsilon_k} &=& \{\pmb{x}\,\,\vert\,\,\partial_{x^i}  v^{\delta_k}(\pmb{x}) > (-1+2\epsilon_k){ L_i}K_i^{+} \}.
\end{eqnarray}
Since $\|\nabla v^{\delta_k}-\nabla v\|_{1}<K_{\min}\epsilon_k$ in $\mathcal{C}_N$ and by the definition in \eqref{G_epsilon}, we have $\mathcal{C}_{\epsilon_k} \subset \mathcal{C}_N$.

{First, let us show $\mathcal{C}_{\epsilon_k}$ is non-empty when $\epsilon_k\in (0,\frac{1}{4})$. We claim that $v$ attains its minimum in $\mathcal{C}_N$. To see this, let $\theta\in\left(0,\frac{K_{\rm min}}{2}\right)$ be given, and choose $\pmb{x}^{\theta}\in \mathbb{R}^N$ such that
$v(\pmb{x}^{\theta}) \leq v(\pmb{x})+\theta, \quad \forall \pmb{x}\in \mathbb{R}^N$. Define
$
\psi_{\theta}(\pmb{x}) := v(\pmb{x}) +\theta \|\pmb{x}^{\theta}-\pmb{x}\|^2$ for all  $\pmb{x}\in \mathbb{R}^N$,
and note that $\psi_{\theta}$ attains its minimum over $\mathbb{R}^N$ at some point $\pmb{y}^{\theta}$. In particular,
\begin{eqnarray}\label{eq:gradient-0}
    0 = \nabla \psi_{\theta}(\pmb{y}^{\theta}) = \nabla v (\pmb{y}^{\theta})+2\theta (\pmb{y}^{\theta}-\pmb{x}^{\theta}).
\end{eqnarray}
But also
\[
v(\pmb{y}^{\theta}) +\theta \|\pmb{y}^{\theta}-\pmb{x}^{\theta}\|^2 = \psi_{\theta}(\pmb{y}^{\theta})\leq \psi_{\theta}(\pmb{x}^{\theta}) = v(\pmb{x}^{\theta}) \leq v(\pmb{y}^{\theta})+\theta.
\]
It follows that $\|\pmb{x}^{\theta}-\pmb{y}^{\theta}\|\leq 1$. Returning to \eqref{eq:gradient-0}, we have $\|\nabla v(\pmb{y}^{\theta})\|^2 \leq 4\theta^2< K^2_{\rm min}.$ Hence $|\partial_{x^i}v(\pmb{y}^{\theta})| < K_{\rm min}$ $(i=1,2,\cdots,N)$ and $\pmb{y}^{\theta} \in \mathcal{C}_N$ for all $\theta \in (0,\frac{K_{\rm min}}{2})$. Since $\mathcal{C}_N$ is bounded, there exists a sequence $\theta_k\in (0,\frac{K_{\rm min}}{2})$ with $\lim_{k\rightarrow \infty}\theta_k = 0$ such that  $\lim_{k\rightarrow \infty}\pmb{y}^{\theta_k} =\pmb{y}^0$ for some $\pmb{y}^0\in \overline{\mathcal{C}}_N$.  From \eqref{eq:gradient-0} we have $\nabla v(\pmb{y}^0)=0$ and hence $\pmb{y}^0\in {\mathcal{C}}_N$. In addition,
the convexity of $v$ implies that $v$  attains its minimum at $\pmb{y}^0$. We now show that $B(\pmb{y}^0,\frac{K_{\min}}{4C})\subseteq \mathcal{C}_{\epsilon_k}$ for all $\epsilon_k \in (0,\frac{1}{4})$.
For any $ \pmb{x} \in B(\pmb{y}^0,\frac{1}{K_{\max}})$ and $i=1,2,\cdots,N$,
\[
|\partial_{x^i}v^{\delta_k}(\pmb{x})| \leq |\partial_{x^i}v(\pmb{x} )|+K_{\min}\epsilon \leq K\|\pmb{x}-\pmb{y}^0\|+K_{\min}\epsilon \leq \frac{1}{2}K_{\min}.
\]
The first inequality holds by the definition of $\delta_k$ and the second inequality holds since $\|\nabla^2 v\|\leq C$. By definition of $\mathcal{C}_{\epsilon_k}$ in \eqref{G_epsilon}, we have $B(\pmb{y}^0,\frac{K_{\min}}{4C})\subseteq \mathcal{C}_{\epsilon_k}$ and hence $\mathcal{C}_{\epsilon_k}\neq \emptyset$ holds for all $\epsilon_k \in (0,\frac{1}{4})$. }

Also notice that $\partial{G}_{j}^{\epsilon_k} \cap \overline{\mathcal{C}}_{\epsilon_k} \in \mathcal{C}^2$ because $v^{\delta_k}$ is smooth.
Now, take any $\epsilon=\epsilon_l$ from the sequence $\{\epsilon_k\}_k$ and take $\delta \in [0,\delta_{l}]$, and denote $\mathcal{S}_{\epsilon}=\partial \mathcal{C}_{\epsilon}$ as the boundary of $ \mathcal{C}_{\epsilon}$ and  $I_{\epsilon}(\pmb{x})=\left\{j\,\, \vert \,\, \pmb{x} \notin G_j^{\epsilon} ,\,\, j=1,2,\cdots,2N\right\}$.
Define the vector field $\gamma_j$ on each face $G_j^{\epsilon}$ as \eqref{eq:gamma_direction} and 
the directions of reflection by
\begin{eqnarray}\label{cone}
\pmb{r}_{\epsilon}(\pmb{x}) = \left\{\sum_{j \in I_{\epsilon}(\pmb{x})} c_j \gamma_j (\pmb{x})\,\, : \,\, c_i \geq 0 \,\mbox{ and }\, \left\| \sum_{j \in I_{\epsilon}(\pmb{x})} c_j \gamma_j(\pmb{x}) \right\| =1  \right\}.
\end{eqnarray}
When $\epsilon=0$, denote $I(\pmb{x}):=I_{0}(\pmb{x})$ and $\pmb{r}(\pmb{x}) := \pmb{r}_0(\pmb{x})$ for the index set and reflection cone of region $\mathcal{C}_N$, respectively. Then define the normal direction on face $G_{j}^{\epsilon}$  as $n^{\epsilon}_j$ ($j=1,2,\cdots,2N$) with
\begin{eqnarray*}
n^{\epsilon}_i=-\frac{\nabla (\partial_{x^i} v^{\delta})}{\|\nabla (\partial_{x^i} v^{\delta})\|_2},&&
n^{\epsilon}_{i+N}=\frac{\nabla (\partial_{x^i} v^{\delta})}{\|\nabla (\partial_{x^i} v^{\delta})\|_2}, \qquad i=1,2,\cdots,N.
\end{eqnarray*}
Note that the normal direction $n^{\epsilon}_j$ $(j=1,2,\cdots,2N)$ is well-defined by the construction of \eqref{G_epsilon}.  

Next we show that $n^{\epsilon}_{i} \cdot \gamma_i=\frac{\partial^2_{x^i}v^{\delta}}{\|\nabla(\partial_{x^i} v^{\delta})\|_2}>0$  and $n^{\epsilon}_{i+N} \cdot \gamma_{i+N}=\frac{\partial^2_{x^i}v^{\delta}}{\|\nabla(\partial_{x^i} v^{\delta})\|_2}>0$ for $i=1,2,\cdots,N$. 

To do so, we shall show that $B_{\delta}(\pmb{x})\in \mathcal{C}_N$ for $\pmb{x}\in \mathcal{S}_{\epsilon}$. Note that
$(-1+2\epsilon) L_i K_{i}^{+}\leq \partial_{x^i}v(\pmb{x})\leq (1-2\epsilon)L_i K_{i}^{-}$ for $\pmb{x} \in \bar{\mathcal{C}}_{\epsilon}$. For any $\pmb{y}\in B_{\delta}(\pmb{x})$, $| \partial_{x^i}v(\pmb{x})-  \partial_{x^i}v(\pmb{y})|\leq C \|\pmb{x}-\pmb{y}\|\leq C \delta\leq \epsilon K_{\min}$. Therefore, 
$(-1+\epsilon) { L_i}K_{i}^{+} \leq (-1+2\epsilon)L_iK_{i}^{+}-\epsilon K_{\min}\leq  \partial_{x^i} v(\pmb{y})\leq (1-2\epsilon)L_iK_{i}^{-}+\epsilon K_{\min} \leq (1-\epsilon){ L_i}K_{i}^{-}. $
Thus, $\pmb{y}\in \mathcal{C}_N$ for all $\pmb{y} \in B_{\delta}(\pmb{x}) $ and $\pmb{x}\in \mathcal{S}_{\epsilon}$. Moreover, under {Assumption \textbf{A4}}, $ \partial^2_{x^i}v^{\delta}(\pmb{x})=\int_{\pmb{y}\in B_{\delta}(\pmb{x})} \partial^2_{x^i}v(\pmb{y})\phi^{\delta}(\pmb{x}-\pmb{y})d\pmb{y}> 0$ for all $\pmb{x}\in \mathcal{S}_{\epsilon}$.

Furthermore, at each point $\pmb{x}\in S_{\epsilon}$, there exists $\gamma \in \pmb{r}_{\epsilon}(\pmb{x})$ pointing into $\mathcal{C}_{\epsilon}$. This is because
there is no $\pmb{x}\in \partial \mathcal{C}_{\epsilon}$ such that $\{i,i+N\} \in I_{\epsilon}(\pmb{x})$ for any $i=1,2,\cdots,N$, and this implies $|I_{\epsilon}(\pmb{x})|\leq N$ for all $\pmb{x}\in \partial \mathcal{C}_{\epsilon}$. Now Assumption \textbf{A4} implies the  following condition (3.8) in~\citep{DI1993}: the existence of scalars $b_j \geq 0$ $j\in I_{\epsilon}(\pmb{x})$, such that
\begin{eqnarray*}
b_j \left<\gamma_j(\pmb{x}),n_j(\pmb{x}) \right> > \sum_{k \in I_{\epsilon(\pmb{x})\setminus\{i\}}} b_k \left|\left<\gamma_k(\pmb{x}),n_k(\pmb{x}) \right>\right|. 
\end{eqnarray*}
{ Here we can simply take $b_j=1$ for all $j \in I_{\epsilon}(\pmb{x})$.}
Therefore, by Theorem 4.8 and Corollary 5.2 of \citep{DI1993}, there exists a unique strong solution  to the Skorokhod problem with data $\left(\{\pmb{x}+\pmb{\mu}t+\pmb{\sigma}B_t\}_{t \ge 0},\mathcal{C}_{\epsilon}, \pmb{r}_{\epsilon},\pmb{x}\right)$.

\paragraph{Step 2. $\epsilon$-optimal policy.} Now we shall show that  the solution to the Skorokhod problem with data $(\pmb{x}+\pmb{\mu}t+ \pmb{\sigma}\pmb{B}_t,\mathcal{C}_{\epsilon}, \pmb{r}_{\epsilon},\pmb{x})$ is an $\epsilon$-optimal policy of the control problem \eqref{centrol_controller} with
\begin{eqnarray}\label{epsilon_control}
\pmb{\xi}_t^{\epsilon} = \int_0^t \pmb{N}_s^{\epsilon} \cdot d \eta_s^{\epsilon},
\end{eqnarray}
and $\pmb{N}^{\epsilon}_s \in \pmb{r}_{\epsilon}(\pmb{X}^{\epsilon}_s)$ on $\mathcal{S}_{\epsilon}$, with $\pmb{X}^{\epsilon}_t =\pmb{x}+\pmb{\mu}t +  \pmb{\sigma}\pmb{B}_t+\pmb{\xi}_t^{\epsilon}$.
By Theorem 4.8 of \citep{DI1993}, $\pmb{X}^{\epsilon}$ is a continuous process. Since $v\in \mathcal{C}^{4,\alpha}(\mathcal{C}_N)$, applying  It\^o formula to the  semi-martingale $\pmb{X}^{\epsilon}$ yields
\begin{eqnarray}\label{eq:inter_33}
v(\pmb{x}) &=& \mathbb{E}_{\pmb{x}} \int_0^{\infty} e^{-\rho t}\left[  H(\pmb{X}^{\epsilon}_t)dt+  \nabla v(\pmb{X}_t^{\epsilon})\cdot \pmb{N}_t^{\epsilon}d{\eta}_t^{\epsilon} \right]\nonumber\\
&\geq & \mathbb{E}_{\pmb{x}} \int_0^{\infty}e^{-\rho t} \left[  H(\pmb{X}^{\epsilon}_t)dt+(1-3\epsilon)\left[(\pmb{N}_t^{\epsilon})^+ \cdot \pmb{K_{L}^+}+(\pmb{N}_t^{\epsilon})^- \cdot \pmb{K_{L}^+} \right]d\eta_t^{\epsilon} \right]\nonumber\\
&\geq & (1-3\epsilon)\mathbb{E}_{\pmb{x}} \int_0^{\infty}e^{-\rho t} \left[  H(\pmb{X}^{\epsilon}_t)dt+\left[(\pmb{N}_t^{\epsilon})^+ \cdot \pmb{K_{L}^+}+(\pmb{N}_t^{\epsilon})^- \cdot \pmb{K_{L}^+} \right]d\eta_t^{\epsilon} \right]\nonumber\\
& = & (1-3\epsilon) J (\pmb{x}; \pmb{\xi}^{\epsilon})
\end{eqnarray}
where $\pmb{N}^{\epsilon}(\pmb{x}) \in \pmb{r}_{\epsilon}(\pmb{x})$ on $\mathcal{S}_{\epsilon}$, 
\begin{eqnarray}\label{eq:K_alpha}
\pmb{K_{L}^+}:=(L_1K_1^{+},\cdots,L_N K_N^{+}),\,\,\, \pmb{K_{L}^-}:=(L_1K_1^{-},\cdots,L_N K_N^{-}),\,\,\, {\rm and } \,\,\,K_{\max} = \max_{1\leq i \leq N}\{L_i K_i^{+},L_i K_i^{-}\}.
\end{eqnarray}
The first inequality of \eqref{eq:inter_33} holds since $\|\nabla v^{\delta}-\nabla v\|_{L_1}<K_{\min}\epsilon$ for $\delta \in [0,\delta(\epsilon)]$ and  \eqref{G_epsilon}. The second inequality of \eqref{eq:inter_33} holds since $H(\pmb{x})\ge 0$.

\end{proof}

Now we are ready to establish the main theorem when $\pmb{x} \in \overline{\mathcal{C}}_N$.

\begin{theorem}[Existence and uniqueness of optimal control]\label{thm:skorokhod}
Take $\pmb{x} \in \overline{\mathcal{C}}_N$ and assume  \textbf{A1}- {\textbf{A4}}.
Then there exists a unique optimal control $\pmb{\xi}^*$ to problem (\ref{centrol_controller}), which is a solution to the Skorokhod problem 
(\ref{def:skorokhod})  with data $( \pmb{x}+\pmb{\mu}t + \pmb{\sigma} \pmb{B}_t, \mathcal{C}_N,\pmb{r},\pmb{x})$ such that $\pmb{X}_t^{*}\in\overline{\mathcal{C}}_N$ under control $\pmb{\xi}^*$. 
\end{theorem}
\begin{proof}
\noindent {\bf Step 1: Optimality.} {
{The existence of the optimal control to problem \eqref{centrol_controller}
follows from Theorem 4.5 and Corollary 4.11 in \citep{MT1989}.} According to Corollary 4.11 of \cite{MT1989}, if $(\pmb{N}^{\epsilon_k},{\eta}^{\epsilon_k})$ is a sequence of $\epsilon_k-$optimal policies for $\pmb{x}$ and $\lim_{k \rightarrow \infty}\epsilon_k \rightarrow 0$, then one can extract a subsequence  $\epsilon_{k^{\prime}}$ such that
\begin{eqnarray}\label{convergence_epsilon_control}
\pmb{\xi}_t^{\epsilon_{k^{\prime}}}(\omega) = \int_0^t \pmb{N}_s^{\epsilon_{k^{\prime}}}(\omega) d \eta_s^{\epsilon_{k^{\prime}}}(\omega) \mathop{\rightarrow}^{k\to\infty} {\pmb{\xi}}^*_t(\omega)\quad dt \times d\mathbb{P} -a.e.
\end{eqnarray}
where $\pmb{\xi}^*$, defined in \eqref{convergence_epsilon_control}, is optimal, i.e., $\pmb{\xi}^* \in \arg \min_{\pmb{\xi}\in \mathcal{U}_N} J(\pmb{x};\pmb{\xi})$.} By the analysis in Theorem \ref{thm:epsilon_policy}, there exits  a sequence of $\epsilon_k-$optimal policy and $\epsilon_k \rightarrow 0$ when $k \rightarrow \infty$. Therefore, the optimal control exists, which is the limit of $\pmb{\xi}_t^{\epsilon_{k^{\prime}}}(\omega)$ defined in \eqref{convergence_epsilon_control}.



\noindent {\bf Step 2: Skorokhod condition.}  We next show that the limiting control ${\pmb{\xi}^*}$ in \eqref{convergence_epsilon_control} is a solution to the Skorokhod problem (\ref{def:skorokhod})  with data $( \pmb{x}+\pmb{\mu}t + \pmb{\sigma} \pmb{B}_t, \mathcal{C}_N,\pmb{r},\pmb{x})$ such that $\pmb{X}_t^{*}\in\overline{\mathcal{C}}_N$.
Let us first check Property {\bf (b)} of the Skorokhod problem (Definition \ref{def:skorokhod}).
Denote
$$A = \left\{ \omega \given \pmb{X}_t^{\epsilon_{k^{\prime}}}  (\omega) \in \overline{\mathcal{C}}_{\epsilon_{k^{\prime}}} \mbox{ for all } 0 \leq t <\infty \mbox{ and all } {k^{\prime}} \geq 0\right\},$$
then by definition (\ref{epsilon_control}), $P(A)=1$. Also define
$$B = \left\{ \omega \given \pmb{X}_t^{\epsilon_{k^{\prime}}} \rightarrow  \pmb{X}_t \mbox{ a.e. } {\rm Leb} \mbox{ on } [0,\infty) \right\},$$
then by (\ref{convergence_epsilon_control}),  $P(B)=1$. For all $\omega \in A \cap B$, since $\overline{\mathcal{C}}_N$ is closed,
$$\pmb{X}_t(\omega) \in \overline{\mathcal{C}}_N \,\,\mbox{Leb a.e. on} \,\,[0,\infty).$$

Now we check property {\bf (c)} in Definition \ref{def:skorokhod}, i.e., the optimal policy  acts only on $\partial \mathcal{C}_N$, and its reflection direction is in $\pmb{r}(\pmb{x})$.
 
Take the smooth function $\phi^{\epsilon}$ in \eqref{smooth_function} and the smooth version of value function $v^{\epsilon}$ in \eqref{smooth_function-v}.
 Let $H^{\epsilon}(\pmb{x}) =  \phi^{\epsilon}* H(\pmb{x})$. From the HJB Equation (\ref{po_hjb}),
\begin{eqnarray}\label{eq:property_HJB}
\rho v -\mathcal{L}v \leq H, \quad \beta(\nabla v) \leq 1 \mbox{ in } \mathbb{R}^N,
\end{eqnarray}
and
\begin{eqnarray} \label{property_smooth_fun}
\rho v^{\epsilon} -\mathcal{L}v^{\epsilon} \leq H^{\epsilon}, \quad \beta(\nabla v^{\epsilon}) \leq 1 \mbox{ in } \mathbb{R}^N.
\end{eqnarray}
{To see this, take $\pmb{x}\in \mathbb{R}^N$. By \eqref{eq:property_HJB},
\[
  \rho v^{\epsilon}(\pmb{x}) -\mathcal{L}v^{\epsilon}(\pmb{x}) = \int_{B(0,\epsilon)} \phi^{\epsilon}(\pmb{y})[\rho(\pmb{x}-\pmb{y})-\mathcal{L}v(\pmb{x}-\pmb{y})]d\pmb{y} \leq \int_{B(0,\epsilon)} \phi^{\epsilon}(\pmb{y})H(\pmb{x}-\pmb{y})d\pmb{y}=H^{\epsilon}(\pmb{x}),\]
  where $B(0,\epsilon) = \{\pmb{x}\in \mathbb{R}^N : \|\pmb{x}\| \leq \epsilon\}$.
  For any $i=1,2,\cdots,N$ and $\pmb{x}\in \mathbb{R}^N$, by \eqref{eq:property_HJB} we have 
  \begin{eqnarray*}
  \partial_{x^i} v^{\epsilon}(\pmb{x}) &=&\partial_{x^i} \left( \int_{B(0,\epsilon)} \phi^{\epsilon}(\pmb{y}) v(\pmb{x}-\pmb{y})d\pmb{y}\right)=\int_{B(0,\epsilon)} \phi^{\epsilon}(\pmb{y})\partial_{x^i}v(\pmb{x}-\pmb{y})d\pmb{y}\\
  &\leq&  \int_{B(0,\epsilon)} \phi^{\epsilon}(\pmb{y})L_iK^{-}_{i}d\pmb{y} = L_iK^{-}_{i}. 
\end{eqnarray*}
Similarly $-L_iK^{+}_{i} \leq  \partial_{x^i} v^{\epsilon}(\pmb{x})$ holds for all $i=1,2,\cdots,N$ and $\pmb{x}\in \mathbb{R}^N$. Hence $\beta(\nabla v^{\epsilon}) \leq 1 \mbox{ in } \mathbb{R}^N$.
}

Letting $T>0$ and applying the It\^o formula \citep[Theorem 21]{Meyer76} to $e^{-\rho t}v^{\epsilon}(\pmb{x})$ and the semi-martingale $\{\pmb{X}_t\}_{t \ge 0}$ under any admissible control $(\xi^{i,+},\xi^{i,-})_{i=1}^N$ yields
\begin{eqnarray*}
\mathbb{E}_{\pmb{x}} \left[e^{-\rho t}v^{\epsilon}({\pmb{X}}_T)\right] =&\quad & v^{\epsilon}(\pmb{x}) + \mathbb{E}_{\pmb{x}} \int_0^T  e^{-\rho t} \left( \mathcal{L} v^{\epsilon}-\rho  v^{\epsilon}\right)({\pmb{X}}_t)dt \\
\quad &+& \mathbb{E}_{\pmb{x}} \int_0^T  e^{-\rho t}  \nabla v^{\epsilon} ({\pmb{X}}_t) \cdot d \pmb{\xi}_t\\
\quad &+&  \mathbb{E}_{\pmb{x}} \int_0^T  \sum_{ 0 \leq t <T}e^{-\rho t} (v^{\epsilon} ({\pmb{X}}_t) -v^{\epsilon} ({\pmb{X}}_{t-}) -\nabla v^{\epsilon}\cdot({\pmb{X}}_t) (\pmb{\xi}_t-\pmb{\xi}_{t-} )),
\end{eqnarray*}
with the last term coming from the jumps of ${\pmb{X}}_t$.
By (\ref{property_smooth_fun}),
\begin{eqnarray}\label{relation}
\begin{aligned}
&\quad&\mathbb{E}_{\pmb{x}} \left[e^{-\rho T}v^{\epsilon}({\pmb{X}}_T)\right]   + \mathbb{E}_{\pmb{x}} \int_0^T  e^{-\rho t} H^{\epsilon}({\pmb{X}}_t)dt 
-\mathbb{E}_{\pmb{x}} \int_0^T  e^{-\rho t}  \nabla v^{\epsilon} ({\pmb{X}}_t) \cdot d{\pmb{\xi}}_t \\
&\quad& \quad +  \mathbb{E}_{\pmb{x}} \int_0^T  \sum_{ 0 \leq t <T}e^{-\rho t} (-v^{\epsilon} ({\pmb{X}}_t) +v^{\epsilon} ({\pmb{X}}_{t-}) +\nabla v^{\epsilon}({\pmb{X}}_t) \cdot (\pmb{\xi}_t-\pmb{\xi}_{t-}) ) \geq v^{\epsilon}(\pmb{x}).
\end{aligned}
\end{eqnarray}
 Moreover, $H^{\epsilon}$, $v^{\epsilon}$ are bounded uniformly on $\overline{\mathcal{C}}_N$ for $\epsilon<1$ because $v$, $\nabla v$, $D^2 v$ are bounded on $B(0,R)$, with $  \overline{\mathcal{C}}_N \subset B(0,R-1)$, thus
\begin{eqnarray*}
v^{\epsilon} \rightarrow v, \quad \nabla v^{\epsilon} \rightarrow \nabla v, \quad H^{\epsilon} \rightarrow H \quad \mbox{ uniformly in } \overline{\mathcal{C}}_N.
\end{eqnarray*}
Meanwhile, for $\forall \pmb{x} \in \overline{\mathcal{C}}_N$,
\begin{eqnarray}\label{value_function_expression}
v({\pmb{x}}) = \mathbb{E}_{\pmb{x}} \int_0^{\infty} e^{-\rho t}\left[ H(\pmb{X}^*_t)dt + \left[({\pmb{N}}_t^*)^+ \cdot \pmb{K_{L}^{+}} + ({\pmb{N}}_t^*)^- \cdot \pmb{K_{L}^{-}} \right] d \eta_t^*      \right],
\end{eqnarray}
where ${\pmb{X}}_t^* = {\pmb{x}} + \pmb{\mu}t +\pmb{\sigma}{\pmb{B}}_t + \pmb{\xi}_t^{*}$ with $\pmb{\xi}_t^{*}:= \int_0^t {\pmb{N}}_s^* d \eta_s^*$ the optimal control, and $\pmb{K_{L}^{+}}$ and $\pmb{K_{L}^{-}}$ are defined in \eqref{eq:K_alpha}. In particular,
\begin{eqnarray}\label{bounded_control}
 \mathbb{E}_{\pmb{x}} \int_0^{\infty} e^{-\rho t} d \eta^*_t < \infty,
\end{eqnarray}
which leads to
\begin{eqnarray*}
 \mathbb{E}_{\pmb{x}} \int_0^{T} e^{-\rho t} \left[({\pmb{N}}_t^*)^+ \cdot \pmb{K^{+}_{L}}+ ({\pmb{N}}_t^*)^- \cdot \pmb{K^{-}_{L}} \right] d \eta^*_t < \infty.
\end{eqnarray*}
By the bounded convergence theorem and (\ref{relation}),
\begin{eqnarray}\label{relation_limit}
\begin{aligned}
&\quad&\mathbb{E}_{\pmb{x}} \left[e^{-\rho T}v(\pmb{X}^*_T)\right]   + \mathbb{E}_{\pmb{x}} \int_0^T  e^{-\rho t} H(\pmb{X}^*_t)dt 
-\mathbb{E}_{\pmb{x}} \int_0^T  e^{-\rho t}  \nabla v (\pmb{X}^*_t) \cdot \pmb{N}^*_t d \eta^*_t\\
&\quad& \quad +  \mathbb{E}_{\pmb{x}} \int_0^T  \sum_{ 0 \leq t <T}e^{-\rho t} \left(-v (\pmb{X}^*_t) +v^{\epsilon} (\pmb{X}^*_{t-}) +\nabla v(\pmb{X}^*_t) \cdot \pmb{N}_t^*(\eta^*_t-\eta^*_{t-} )\right) \geq v(\pmb{x}).
\end{aligned}
\end{eqnarray}
The last term on the left-hand side is nonpositive because of convexity of $v$, hence
\begin{eqnarray*}
\begin{aligned}
&\quad&\mathbb{E}_{\pmb{x}} \left[e^{-\rho T}v(\pmb{X}^*_T)\right]   + \mathbb{E}_{\pmb{x}} \int_0^T  e^{-\rho t} H(\pmb{X}^*_t)dt 
-\mathbb{E}_{\pmb{x}} \int_0^T  e^{-\rho t}  \nabla v (\pmb{X}^*_t) \cdot \pmb{N}^*_t d \eta^*_t \geq v(\pmb{x}).
\end{aligned}
\end{eqnarray*}
Letting $T \rightarrow \infty$, by the boundedness of ${\pmb{X}}_t^*$, $\beta(\nabla v) \leq 1$, $|{\pmb{N}}_t^*|=1$, and (\ref{bounded_control}), 
\begin{eqnarray*}
\begin{aligned}
 \mathbb{E}_{\pmb{x}} \int_0^{\infty}  e^{-\rho t} H(\pmb{X}^*_t)dt 
-\mathbb{E}_{\pmb{x}} \int_0^{\infty}  e^{-\rho t}  \nabla v (\pmb{X}^*_t) \cdot \pmb{N}^*_t d \eta^*_t \geq v(\pmb{x}).
\end{aligned}
\end{eqnarray*}
Along with (\ref{value_function_expression}), we have
\begin{eqnarray*}
0 \geq \mathbb{E}_{\pmb{x}} \int_0^{\infty} e^{-\rho t} \left( \left[\nabla v({\pmb{X}}_t^*) + \pmb{K_{L}^+}\right] \cdot ({\pmb{N}}_t^*)^+ d \eta^*_t+ \left[-\nabla v({\pmb{X}}_t^*) + \pmb{K_{L}^-} \right]\cdot ({\pmb{N}}_t^*)^-  d \eta^*_t  \right).
\end{eqnarray*}
Given $\beta(\nabla v) \leq 1$, we have
$-K^{+}_i\leq v_{x^i}(\pmb{x}) \leq K^{-}_i$, $\forall x \in \mathbb{R}^N$  and $ i=1,2,\cdots,N.$
Hence
\begin{eqnarray*}
0 \geq \mathbb{E}_{\pmb{x}} \int_0^{\infty} e^{-\rho t} \left( \left[\nabla v({\pmb{X}}_t^*) + \pmb{K_{L}^+}\right] \cdot ({\pmb{N}}_t^*)^+ d \eta^*_t+ \left[-\nabla v({\pmb{X}}_t^*) + \pmb{K_{L}^-} \right]\cdot ({\pmb{N}}_t^*)^-  d \eta^*_t  \right)\geq 0.
\end{eqnarray*}
This implies
$d \eta_t^* =0$ when $\beta(\nabla v({\pmb{X}}_t^*))<1$ a.e. in $t$. Also, when $d \eta^*_t \neq 0$, ${\pmb{N}}_t^*(\pmb{x}) \in \pmb{r}(\pmb{x})$ for $\pmb{x} \in \mathcal{S}$ a.e. for  $t\in[0,\infty)$, where the reflection cone $\pmb{r}(\pmb{x})$ is defined in \eqref{cone_original}.

By Assumption {\bf A4}, for any $\pmb{x} \in \partial\mathcal{C}_N$ and $\gamma(\pmb{x}) \in \pmb{r}(\pmb{x})$,  $\gamma(\pmb{x})$ is not parallel to $\partial\mathcal{C}_N$ at $\pmb{x}$. Hence, property {\bf(a)} holds, i.e., the optimal control is continuous.

{\bf Step 3: Uniqueness.} It remains to show the uniqueness of the optimal control. This is done by a contradiction argument. Suppose that there are two optimal controls $\{\pmb{\xi}^*\}_{t \geq 0}$ and $\{\pmb{\xi}^{**}\}_{t \geq 0}$ such that  $\pmb{\xi}^* \neq \pmb{\xi}^{**}$ almost surely. 
Let $\{ \pmb{X}_t^* \}_{t \geq 0}$ and $\{\pmb{X}_t^{**}\}_{t \geq 0}$ be the corresponding trajectories. Let $\pmb{\xi}_t = \frac{\pmb{\xi}^*_t+\pmb{\xi}_t^{**}}{2}$ and $\pmb{X}_t = \frac{\pmb{X}^*_t+\pmb{X}_t^{**}}{2}$. Then by Assumption \textbf{A3},
\begin{eqnarray*}
v(\pmb{x}) - J(\pmb{x};\pmb{\xi}_t) &=& \frac{(J(\pmb{x}; \pmb{\xi}^*)+J(\pmb{x}; \pmb{\xi}^{**}))}{2}  - J (x;\pmb{\xi})\\
&\geq& \mathbb{E}_{\pmb{x}} \int_0^{\infty} e^{-\rho t} \left[ \frac{H(\pmb{X}_t^*)+H(\pmb{X}_t^{**})}{2} - H\left(\frac{\pmb{X}_t^*+\pmb{X}_t^{**}}{2}\right) \right]dt>0.
\end{eqnarray*}
Therefore $v(\pmb{x}) > J(\pmb{x}; \pmb{\xi})$, which contradicts the optimality of $\{\pmb{\xi}_t^*\}_{t \geq 0}$ and $\{\pmb{\xi}_t^{**}\}_{t \geq 0}$. Hence the uniqueness of the optimal control.
\end{proof}

\paragraph{Optimal policy for $\pmb{x} \notin  \overline{\mathcal{C}}_N$.}
 \label{sec:thm1_proof3} 
 When $\pmb{x} \notin \overline{\mathcal{C}}_N$, the optimal policy is to jump immediately to some point $\hat{\pmb{x}} \in \overline{\mathcal{C}}_N$ and then follows the optimal policy in $\overline{\mathcal{C}}_N$.
We will need the following assumption so that the reflection field of the Skorokhod problem is extendable to the $\mathbb{R}^N$ plane \citep{DI1991}. Note that when $N=2$,  \textbf{A5} follows directly from Assumptions \textbf{A1}-\textbf{A3}.
  
 \begin{enumerate}[font=\bfseries,leftmargin=3\parindent]
\item [A5.]
 There is a map $\pi: \mathbb{R}^N \rightarrow \mathcal{C}_N$ satisfying $\pi(\pmb{x}) = \pmb{x}$ for all $\pmb{x} \in \mathcal{C}_N$ and $\pi(\pmb{x}) -\pmb{x} \in \pmb{r}(\pi(\pmb{x}))$.
\end{enumerate}
 This assumption was also adopted in \citep[Assumption 3.1]{DI1991}.
\begin{theorem}\label{epsilon_out}
Given \textbf{A1}-\textbf{A3}, and \textbf{A5}. For any $\pmb{x}\notin \overline{\mathcal{C}}_N$, there exists {an optimal policy $\pi$ such that} $\pi(\pmb{x}) \in \partial \mathcal{C}_N$ at time $0$ and 
$$v(\pmb{x}) = v(\pi(\pmb{x}))+l(\pmb{x}-\pi(\pmb{x})),$$
with $l(\pmb{y}) = \sum_i l_i(y_i)$, where 
\begin{eqnarray} \label{l_function}
    l_{i}(y_i)=\left\{
                \begin{array}{ll}
                 L_iK_i^- y_i, \quad \mbox{ if } y_i \geq 0,\\
                 -L_iK^+_i y_i, \quad \mbox{ if } y_i <0.
                \end{array}
              \right.
\end{eqnarray}
\end{theorem}

\begin{proof}
Notice that $l(\pmb{y})$ is  convex  and
\begin{eqnarray*}
l_i(y_i) = \max_{-L_iK_i^+ \leq k \leq L_iK_i^-} \{k y_i\} = \max\{ -L_iK_i^+y_i,L_iK_i^-y_i\} \mbox{ for } y_i \in \mathbb{R}.
\end{eqnarray*}

\noindent Here we define two linear approximations which correspond to the lower and the upper bounds of the value function $v(\pmb{x})$, respectively.

For ${\pmb{x}}\not\in  \overline{\mathcal{C}}_N$,  define
\begin{eqnarray}\label{pi1_and_pi2}
 u_1(\pmb{x}) &=& v(\pi({\pmb{x}})) + \nabla v (\pi{(\pmb{x})})\cdot  (\pmb{x}-\pi{(\pmb{x})}),\nonumber \\
 u_2(\pmb{x}) &=& v(\pi({\pmb{x}})) + l(\pmb{x}-\pi({\pmb{x}})).
 \end{eqnarray}
Then $u_{2}(\pmb{x})  \geq v (\pmb{x})$ by the sub-optimality of the policy, and 
 $u_{1}(\pmb{x}) \leq v (\pmb{x})$ by convexity. Thus,
 \begin{eqnarray}\label{sandwich_value_fun}
 u_{1}(\pmb{x}) \leq v (\pmb{x}) \leq u_{2}(\pmb{x}).
 \end{eqnarray}
 
 We now show $u_{1}(\pmb{x})= u_{2}(\pmb{x})$. By Assumption \textbf{ A5}, $u_1$ and $u_2$ in (\ref{pi1_and_pi2}) can be rewritten as
\begin{eqnarray*}
 u_1(\pmb{x}) &=& v(\pi({\pmb{x}})) + \nabla v (\pi({\pmb{x}})) \cdot d (\pi({\pmb{x}})) \|\pmb{x}-\pi({\pmb{x}})\|,\nonumber \\
u_2(\pmb{x}) &=& v(\pi({\pmb{x}})) + \pmb{P} (\pi({\pmb{x}}))\cdot d  (\pi({\pmb{x}}))\|\pmb{x}-\pi({\pmb{x}})\|,
\end{eqnarray*}
where $d (\pi({\pmb{x}})) \in \pmb{r} (\pi({\pmb{x}}))$ and $\pmb{P}(\pmb{x}) = (P_1,\cdots,P_N)(\pmb{x})$, with $${P}_i(\pmb{x}) =L_i K_i^+ \textbf{1}(\partial_{x^i} v(\pmb{x})<0) +  L_i K_i^- \textbf{1}(\partial_{x^i} v(\pmb{x})>0).$$
Therefore $u_1(\pmb{x})=u_2(\pmb{x})$. 
\end{proof}

\subsection{Pareto-optimal policies}
Pareto-optimal policies for \eqref{N-game} may  be constructed from the  optimal control for problem  \eqref{centrol_controller} as described below.
\begin{theorem}\label{connection_1}
The optimal control for the regulator's problem (\ref{centrol_controller})  yields a Pareto-optimal policy for the game (\ref{N-game}).
\end{theorem}
\begin{proof} To see this, 
 take the payoff function $J^i$  in (\ref{N-game}),  ${v}(\pmb{x})$ the value function  in (\ref{centrol_controller}), and  the optimal control $\pmb{\xi}^*:=(\pmb{\xi}^{1*},\ldots,\pmb{\xi}^{N*})$, if exists, to problem (\ref{centrol_controller}), then for any $\pmb{\xi}:= (\pmb{\xi}^1,\ldots, \pmb{\xi}^N) \in \mathcal{U}_N$ and $L_i$, with $L_i>0, \sum_{i=1}^N L_i=1$,
\begin{eqnarray}
\sum_{i=1}^N L_i {J}^i (\pmb{x};\pmb{\xi})   \geq {v}(\pmb{x}),
\end{eqnarray}
where value $ {v}(\pmb{x})$ is reached when player $i$ takes the control $\pmb{\xi}_t^{i*}$ ($ i=1,2,\ldots,N$).

If there is another $\pmb{\xi}{'}:=(\pmb{\xi}^{1'},\ldots,\pmb{\xi}^{N'})\in \mathcal{U}_N$ and $k  \in \{1,\ldots,N\}$ such that
$$ {J}^k(\pmb{x};\pmb{\xi}^{1'},\ldots, \pmb{\xi}^{N'}) < {J}^k(\pmb{x};\pmb{\xi}^{1*},\ldots, \pmb{\xi}^{N*}),$$
then given $L_i>0$ for all $i$, there must exists $j \in \{1,\ldots,N\}$ such that
$$ {J}^j(\pmb{x};\pmb{\xi}^{1'},\ldots, \pmb{\xi}^{N'}) > {J}^j(\pmb{x};\pmb{\xi}^{1*},\ldots, \pmb{\xi}^{N*}). $$
Hence the control $\pmb{\xi}^*$ is a Pareto-optimal policy by definition. 
\end{proof}

\noindent Combining Theorems  \ref{thm:skorokhod}, \ref{epsilon_out} and \ref{connection_1} yields    the following  result which summarizes the structure of the set of Pareto optima:
\begin{theorem}[Pareto-optimal policies]\label{PO_solution}
 Under Assumptions \textbf{A1}-\textbf{A5}, 
for  any set of weights $\pmb{L}=(L_1, \cdots, L_N)$ with $L_i>0 $  and $\sum_{i=1}^NL_i =1$, the unique solution $\pmb{\xi_L}\in {\cal U}_N$ 
to the regulator's problem \eqref{centrol_controller} yields a
 Pareto-optimal policy for the game (\textbf{N-player}). 
\end{theorem}

The analytical structure of  the continuation region \eqref{set_PO_redefine} and the Pareto-optimal policy suggest the following description: $\pmb{X}_t$ evolves according to the uncontrolled diffusion process inside the interior of ${\cal C}_N$ and when it hits boundary at a point belonging to $\partial G_i$ or $\partial G_{i+N}$, then bank $i$ will adjust its rate to push it back instantaneously inside ${\cal C}_N$.  In particular the optimal policies lead to continuous controls $\xi^i$. 

\subsection{Pareto-optimal policies for interbank lending}\label{sec.LIBOR}

Let us now translate these results in the setting of the interbank lending model described in Section \ref{sec:LIBOR}.

Theorem \ref{PO_solution} implies that Pareto optima for the interbank lending market may be  described in terms of the policy of a regulator facing the  optimization problem \eqref{centrol_controller} with an aggregate  payoff function \eqref{H} representing a weighted average of payoffs of individual banks. 

Under a Pareto-optimal policy, the interbank rates may be described as a `regulated diffusion' in a bounded region ${\cal C}_N$  defined by \eqref{set_PO_redefine}.
The boundedness of ${\cal C}_N$ implies that   the payoff structure \eqref{eq.payoff} leads to  {\it endogenous} bounds on the interbank rates: the regulator only intervenes when the rates reach these bounds, represented by the boundary of the continuation region ${\cal C}_N$.

The Pareto-optimal policy leads $\pmb{X}_t$ to remain confined in the bounded region ${\cal C}_N$, which implies in particular that the spread ${X}^i$ remains bounded.
In the context of the LIBOR mechanism, this can be seen as the impact of `trimmed' averaging, which is the origin of the terms $K_i^+,K_i^-$, as explained in Section \ref{sec:LIBOR}: as banks internalize the risk of being `outliers' in the benchmark fixing, they confine their rates to a bounded region.

The process $\pmb{X}_t$ diffuses in the interior of ${\cal C}_N$, following the random shocks banks are subjected to, and is pushed into the interior when it reaches the boundary. More precisely,  the boundary   $\partial{\cal C}_N$  is composed of  $2N$ `faces'  corresponding to the saturation of the constraints in
\eqref{G}. Edges correspond to intersections of two or more faces. When $\pmb{X}_t$ reaches a point $\pmb{x}\in \partial{\cal C}_N$, action is taken by all banks $i$ such that $\pmb{x}\notin G_i\cup G_{i+N}$: 
if $\pmb{x}\notin G_i $ then $X^i$ is reduced i.e. $d\xi^{i,-}>0$ and if  $\pmb{x}\notin  G_{i+N}$ then  $X^i$ is increased i.e. $d\xi^{i,+}>0$.
 When $\pmb{X}_t$ reaches the interior of such a face, only bank $i$ adjusts its rate in order to 
  push back $\pmb{X}_t$  to the interior. Similarly, if $\pmb{X}_t$ reaches an edge, two or more banks need to simultaneously adjust their rates.
 The rate at which such simultaneous adjustments occur is given by the {\it intersection local time} \citep{rosen1987} of $(X^1,...,X^N)$ on the boundary.
 Therefore  Pareto-optimal policy  rarely leads to more than one bank's rate to be adjusted; a simultaneous rate adjustment by several banks is most likely {\it not} associated   with a Pareto-optimal policy and is thus a signature of a non-optimal behavior by banks.

We also note that, our admissible controls allow for  discontinuous adjustments of rates, and Pareto-optimal policies correspond to {\it instantaneously}  pushing the process to the interior. As discussed in Theorem \ref{epsilon_out},  Pareto-optimal policies may involve  an initial push  at $t=0$ to bring the initial condition into ${\cal C}_N$, which we may interpret as the entry of a new bank into the interbank market.

The set of all such Pareto optima is parameterized by the set of allocations $L=(L_1,...,L_N)$ with $L_i>0$ and $ \sum_{i=1}^N L_i=1$. 
These allocations lead to different outcomes across banks.
A natural choice is to take $L_i$  proportional to the  loan volume of bank $i$;   \eqref{H} then represents an aggregate wealth maximization problem and this policy leads to the same pro-rata cost across banks.
As is clear from \eqref{gamma_function}, choosing a higher weight $L_i$ leads to a tighter control on the rates of bank $i$.

\section{Explicit solution for two players}\label{sec:bank}
{We now study in more detail the structure of the optimal strategies for the case of $N=2$. Our analytical results illustrate the difference between  Nash equilibria and  Pareto optima    and demonstrate the impact of  regulatory intervention in this game.}

\subsection {Pareto-optimum for $N=2$}\label{sec:bank_explicit}
For the special case of  $N=2$, we can derive explicitly its Pareto-optimal solution. 
For  ease of exposition, we shall assume the following conditions in the case of $N=2$.
\begin{itemize}
    \item[{\bf B1.}]$a_1=a_2$ {and $L_1=L_2$}. 
    In other words, the regulator allocates equal weights to the banks.
    \item[{\bf B2.}]$h^1(x^1,x^2)=h^2(x^1,x^2)=h(x^1-x^2)$, $h \in \mathcal{C}^3(\mathbb{R})$ is symmetric, and there exist $0<c<C$ such that $c<h^{\prime \prime} <C$, and $h^{\prime \prime}$ is non-decreasing and bounded away from $0$.
    \item[{\bf B3.}] $\mu^1=\mu^2=0$, $K_1^+=K_1^-=:K_1>0$ and $K_2^+=K_2^-=:K_2>0$.  
\end{itemize}

Note that Assumption {\bf B2} is more general than  Assumptions {\bf A1-A3}. As a result, we will see in Proposition \ref{prop:PO} that the non-action region may not  necessarily be bounded and the Pareto-optimal policy for the game may not be unique with fixed weights $L_1=L_2$.

Under Assumption {\bf B3}, the  rates  $X_t^1$ and $X_t^2$ are assumed to be 
\begin{eqnarray}\label{both-dynamic}
 X_t^i = \pmb{\sigma}^i\cdot d\pmb{B}_t   +d\xi_t^{i,+} -  d\xi_t^{i,-}, \text{ with } x_{0-}^i=x^i,\,\,i=1,2.
\end{eqnarray}
The  value function  $v(x^1,x^2)$  of \eqref{centrol_controller} becomes
\begin{eqnarray}\label{single-game}
{v}(x^1,x^2) &=& \inf_{(\pmb{\xi}^1,\pmb{\xi}^2) \in \mathcal{U}_2} {J}(x^1,x^2,\pmb{\xi}^1,\pmb{\xi}^2) 
=\inf_{(\pmb{\xi}^1,\pmb{\xi}^2) \in \mathcal{U}_2} \frac{1}{2}\left[ J^1(x^1,x^2,\pmb{\xi}^1,\pmb{\xi}^2) +  J^2(x^1,x^2,\pmb{\xi}^1,\pmb{\xi}^2) \right] \\
&=& \inf_{(\pmb{\xi}^1,\pmb{\xi}^2) \in \mathcal{U}_2}  \mathbb{E}_{(x^1,x^2)} \left [\int_0^{\infty} e^{-\rho t} \left( h\left({X_t^1-X_t^2}\right)dt + \frac{K_1}{2}d \xi_t^{1,+}+\frac{K_1}{2}d\xi_t^{1,-} + \frac{K_2}{2} d \xi_t^{2,+} +\frac{K_2}{2} d\xi_t^{2,-} \right) \right],\nonumber
\end{eqnarray}
subject to (\ref{both-dynamic}).

\begin{lemma}
\label{lemma:zero_player_1} Assume $K_2<K_1$ and {\bf B1-B3}. Then for any $(\pmb{\xi}^{1*},\pmb{\xi}^{2*})\in \arg \inf_{(\pmb{\xi}^{1},\pmb{\xi}^{2})\in \mathcal{U}_2} J(x^1,x^2,\pmb{\xi}^1,\pmb{\xi}^2)$, 
\[
({\xi}_t^{1,+*},{\xi}_t^{1,-*}) = (0,0)
\text{  for any } t \ge 0 \,\,\text{ a.s.}.\]
\end{lemma}
\begin{proof}
The statement is proved by contradiction. Assume there exists an optimal policy $(\pmb{\xi}^{1*},\pmb{\xi}^{2*})\in\arg\inf_{(\pmb{\xi}^{1},\,\pmb{\xi}^{2})\in \mathcal{U}_2}{J}(x^1,x^2,(\pmb{\xi}^{1},\pmb{\xi}^{2}))$  and $t_0 \ge 0$ such that
\[
{\xi}_{t_0}^{1*,+}>0.
\]
Since ${\xi}^{1,+*}$ is a non-decreasing process, we have ${\xi}_{t}^{1,+*}>0$ for all $t \ge t_0.$ Now construct the following admissible policy $(\overline{\pmb{\xi}}^1,\overline{\pmb{\xi}}^2)$ such that, $\forall t \ge0$,
\begin{eqnarray}
 \begin{cases}
 \overline{\xi}_{t}^{2,-} =  {\xi}_{t}^{1*,+}+{\xi}_{t}^{2*,-},\\
  \overline{\xi}_{t}^{1,+} = 0,\\
   \overline{\xi}_{t}^{1,-} = {\xi}_{t}^{1*,-},\,\, \overline{\xi}_{t}^{2,+} = {\xi}_{t}^{2*,+}.\\
 \end{cases}
\end{eqnarray}
 Then
\[
{J}({x}^1,x^2,{\pmb{\xi}}^{1*},{\pmb{\xi}^{2*}}) - {J}({x}^1,x^2,{\overline{\pmb{\xi}}}^{1},\overline{{\pmb{\xi}}}^{2}) = \mathbb{E}_{(x^1,x^2)}\left[\int_0^{\infty}e^{-\rho t}\frac{K_1-K_2}{2} d \xi_t^{1*,+} \right]>0,
\]
which contradicts  the optimality of the control process $({\pmb{\xi}}^{1*},{\pmb{\xi}^{2*}})$.
\end{proof}

We now show that solving the control problem \eqref{both-dynamic}-\eqref{single-game} is equivalent to the following control problem  \eqref{single-game2}-\eqref{both-dynamic-reduce} when $K_1>K_2$,

\begin{eqnarray}\label{single-game2}
{u}(y) &=& \inf_{\pmb{\eta} \in \mathcal{U}_1} \widehat{J}(y,\pmb{\eta})  = \inf_{\pmb{\eta}\in \mathcal{U}_1}\mathbb{E}_y \left [\int_0^{\infty} e^{-\rho t} \left( h\left({Y_t}\right)dt +\frac{K_2}{2} d\eta_t^{+}  +\frac{K_2}{2} d\eta_t^{-} \right) \right],\\
{\rm where}\qquad\label{both-dynamic-reduce}
    dY_t  &=& ( \pmb{\sigma}^1 - \pmb{\sigma}^2)\cdot d\pmb{B}_t  - d\eta_t^{+} +d \eta_t^{-}, \textit{ with }\,\, Y_{0-} = y.
\end{eqnarray}

\begin{lemma}[Equivalence]\label{lemma:equivalence}
Assume {\bf B1-B3}  and $K_1>K_2$, then
\begin{itemize}
    \item[(i)] $v(x^1,x^2)=u(x^1-x^2)$;
    \item[(ii)] If $(\pmb{\xi}^{1*},\pmb{\xi}^{2*})\in \arg\inf_{(\pmb{\xi}^{1},\pmb{\xi}^{2})\in \mathcal{U}_2} J(x^1,x^2,(\pmb{\xi}^{1},\pmb{\xi}^{2}))$,  then $({\xi}_t^{1*,+},{\xi}_t^{1*,-}) = (0,0)$ $\forall t$ a.s., and\\ $ \pmb{\xi}^{2*}\in \arg\inf_{\pmb{\eta}\in \mathcal{U}_1} \widehat{J}(x^1-x^2,\pmb{\eta})$;
    \item[(iii)] If  $\pmb{\eta}^*\in \arg\inf_{\pmb{\eta}\in \mathcal{U}_1} \widehat{J}(x^1-x^2,\pmb{\eta})$, then $((\pmb{0},\pmb{0}),\pmb{\eta})\in \arg\inf_{(\pmb{\xi}^{1},\pmb{\xi}^{2})\in \mathcal{U}_2} J(x^1,x^2,(\pmb{\xi}^{1},\pmb{\xi}^{2}))$.
\end{itemize}
\end{lemma}

\begin{proof}
By Lemma \ref{lemma:zero_player_1}, $({\xi}_t^{1*,+},{\xi}_t^{1*,-}) = (0,0)
\text{  for any } t \ge 0 \,\,\text{ a.s.}.$ Therefore, we can consider  a smaller class of admissible control set where $(\xi_t^{1,+},\xi_t^{1,-})=(0,0)$ $\forall t \ge 0$ and $\pmb{\xi}^2 \in \mathcal{U}_1$. Note that with $(\xi_t^{1,+},\xi_t^{1,-})=(0,0)$, we have
\begin{eqnarray}\label{inter:object}
 X_t^1-  X_t^2  = ( \pmb{\sigma}^1-\pmb{\sigma}^2)\cdot \pmb{B}_t   - \xi_t^{2,+} + \xi_t^{2,-}+(x^1-x^2),
\end{eqnarray}

and
\begin{eqnarray}\label{inter:dynamics}
J(x^1,x^2,\pmb{\xi}^1,\pmb{\xi}^2) = \mathbb{E}_{(x^1,x^2)} \left [\int_0^{\infty} e^{-\rho t} \left( h\left({X_t^1-X_t^2}\right)dt + \frac{K_2}{2} d \xi^{2,+} +\frac{K_2}{2} d\xi^{2,-} \right) \right].
\end{eqnarray}
Clearly problem \eqref{inter:object}-\eqref{inter:dynamics} is equivalent to the one-dimensional control problem \eqref{single-game2}-\eqref{both-dynamic-reduce} with $y=x^1-x^2$. Hence the claim.
\end{proof}

\begin{proposition}[Pareto-optimal solution when $N=2$] \label{prop:PO}
Assume   \textbf{B1}-\textbf{B3}.
 \begin{itemize}
     \item[{\rm(i)}] If $K_1 = K_2=K$, then
     the following control yields one Pareto-optimal policy to game \eqref{both-dynamic}-\eqref{single-game}:
\begin{eqnarray}\label{distribution}
\begin{aligned}
\pmb{\xi}_t^{1*} = (\xi_t^{1*,+},\xi_t^{1*,-}) &=\left(0,\max \left\{0,\max_{0 \leq u \leq t} \left \{(x^1-x^2) + (\pmb{\sigma}^1- \pmb{\sigma}^2)\cdot \pmb{B}_u+ \xi_u^{2*,-} - {c}_1 \right\} \right\}\right),\\
\pmb{\xi}_t^{2*} = (\xi_t^{2*,+},\xi_t^{2*,-}) &= \left(0, \max \left\{0,\max_{0 \leq u \leq t}  \left \{-(x^1-x^2) + (\pmb{\sigma}^2- \pmb{\sigma}^1)\cdot \pmb{B}_u+\xi_u^{1*,-}  - {c}_1 \right \}\right \}\right),
\end{aligned}
\end{eqnarray}
where $c_1$ is the unique solution to
\begin{eqnarray}\label{c1}
\frac{\widetilde{\sigma}}{\sqrt{2 \rho}} \tanh \left(\frac{\sqrt{2 \rho}}{\widetilde{\sigma}} x\right) = \frac{p_1^{\prime}(x)-\frac{K}{2}}{p_1^{\prime \prime}(x)},
\end{eqnarray}
 and
 \begin{eqnarray}\label{p1}
p_1(x) = \mathbb{E} \left[\int_0^{\infty} e^{-\rho t} h \left({x}+\widetilde{\sigma}B_t \right)dt \right],
\end{eqnarray}
with $\widetilde{\sigma} = \sqrt{\sum_{i=1}^2\sum_{j=1}^D\sigma_{ij}^2}$.
The associated Pareto-optimal value is
\begin{eqnarray}\label{single-solution}
    v(x^1,x^2)=\left\{
                \begin{array}{ll}
                  -\frac{\widetilde{\sigma}^2p_1^{\prime \prime}(c_1) \cosh \left(\left(x^1-x^2\right)\frac{\sqrt{2 \rho}}{\widetilde{\sigma}}\right)}{2 \rho \cosh \left(c_1 \frac{\sqrt{2 \rho}}{\widetilde{\sigma}} \right)} + p_1(x^1-x^2), \qquad \ \  0 \leq & x^1-x^2 \leq c_1,\\
                  v(x^2+c_1,x^2)+\frac{K}{2}(x^1-x^2-c_1), \qquad &x^1-x^2 \geq c_1,\\
                 v(-x^1,-x^2), \qquad & x^1-x^2 <0.
                \end{array}
              \right.
  \end{eqnarray}

\item[{\rm(ii)}] If $K_1 > K_2$
 then the following control yields a Pareto-optimal policy  to game \eqref{both-dynamic}-\eqref{single-game},
\begin{eqnarray}
\pmb{\xi}_t^{1*} &=& (0,0),\,\, \text{ and }\,\, \pmb{\xi}_t^{2*} = (\xi_t^{2*,+},\xi_t^{2*,-}) \text{ with } \label{distribution1}\\
\xi_t^{2*,-} &=& \max \left\{0,\max_{0 \leq u \leq t}  \left \{-(x^1-x^2) +(\pmb{\sigma}^2- \pmb{\sigma}^1)\cdot \pmb{B}_u  +\xi_u^{2*,+}  - \widetilde{c}_1 \right \}\right \},\label{distribution2}\\
\xi_t^{2*,+} &=& \max \left\{0,\max_{0 \leq u \leq t} \left \{(x^1-x^2) + (\pmb{\sigma}^1-  \pmb{\sigma}^2)\cdot \pmb{B}_u +\xi_u^{2*,-}- \widetilde{c}_1 \right\} \right\},\label{distribution3}
\end{eqnarray}
where $\widetilde{c}_1 $ is the unique solution to
\begin{eqnarray}\label{c1_tilde}
\frac{\widetilde{\sigma}}{\sqrt{2 \rho}} \tanh \left(\frac{\sqrt{2 \rho}}{\widetilde{\sigma}} x\right) = \frac{p_1^{\prime}(x)-\frac{K_2}{2}}{p_1^{\prime \prime}(x)},
\end{eqnarray}
 and the associated Pareto-optimal value is
\begin{eqnarray}\label{single-solution_tilde}
    v(x^1,x^2)=\left\{
                \begin{array}{ll}
                  -\frac{\widetilde{\sigma}^2p_1^{\prime \prime}(\widetilde{c}_1) \cosh \left((x^1-x^2)\frac{\sqrt{2 \rho }}{\widetilde{\sigma}}\right)}{2 \rho  \cosh \left(\widetilde{c}_1 \frac{\sqrt{2 \rho }}{\widetilde{\sigma}} \right)} + p_1(x^1-x^2), \qquad \ \  0 \leq & x^1-x^2 \leq \widetilde{c}_1,\\
                  v(x^2+\widetilde{c}_1,x^2)+\frac{K_2}{2}((x^1-x^2)-\widetilde{c}_1), \qquad &x^1-x^2 \geq \widetilde{c}_1,\\
                 v(-x^1,-x^2), \qquad & x^1-x^2 <0.
                \end{array}
              \right.
  \end{eqnarray}
 \end{itemize}
\end{proposition}
\begin{remark}

Note that  under {\bf B1-B3}, the Pareto-optimal policy is no longer unique with fixed $L_1=L_2=\frac{1}{2}$. 
For instance, when $K_1=K_2=K$, the following control yields another Pareto-optimal policy with the same   value function defined in \eqref{single-solution}:
 \begin{eqnarray}\label{distribution_not_unique}
\begin{aligned}
\pmb{\xi}_t^{1*} &= (\xi_t^{1*,+},\xi_t^{1*,-}) = \left(0, 0\right) \,\, \text{ and }\,\,\pmb{\xi}_t^{2*} = (\xi_t^{2*,+},\xi_t^{2*,-}), \text{ with }\\
\xi_t^{2*,-}  & = \max \left\{0,\max_{0 \leq u \leq t}  \left \{-(x^1-x^2) + (\pmb{\sigma}^2- \pmb{\sigma}^1)\cdot \pmb{B}_u +\xi_u^{2*,+}   - {c}_1 \right \}\right \},\\
\xi_t^{2*,+} &=\max \left\{0,\max_{0 \leq u \leq t} \left \{(x^1-x^2) + (\pmb{\sigma}^1- \pmb{\sigma}^2)\cdot \pmb{B}_u +\xi_u^{2*,-}  - {c}_1 \right\} \right\}.
\end{aligned}
\end{eqnarray}
\end{remark} 
\begin{remark}{\em
Under the Pareto-optimal policy,  the controlled dynamics  $X_t^
{1*}$ and $X_t^{2*}$ are such that $\mathbb{P}(\|X_t^{1*}-X_t^{2*}\| \leq c_1, \forall t \ge 0)=1$. This suggests that there should be a mechanism, such as `trimming', to maintain the dispersion of rates within a certain range. In addition,
this solution form indicates that  it is socially optimal for  the more efficient bank (i.e., the one with the lower cost of adjustment) to take the lead in lending rate adjustment. The other banks then become  `free riders'.  }
\end{remark}
\begin{proof}
First let us prove the case when $K_1>K_2$. 
By Lemma \ref{lemma:equivalence}, it is sufficient to focus on the single-agent problem \eqref{single-game2}-\eqref{both-dynamic-reduce} with $y=x^1-x^2$. Following the standard analysis \citep{BSW1980,karatzas1982}, the HJB equation for the one-dimensional control problem follows \eqref{single-game2}-\eqref{both-dynamic-reduce} is
\begin{eqnarray}\label{hjb-w}
\max \left\{ \rho  {u}(x) - h(x) - \frac{\widetilde{\sigma}^2}{2}{u}^{\prime \prime}(x),\,{u}^{\prime}(x)-\frac{K_2}{2},\,-{u}^{\prime}(x)-\frac{K_2}{2} \right\} = 0.
\end{eqnarray}
There is a $\mathcal{C}^2$ solution \citep{BSW1980,karatzas1982} given by
\begin{eqnarray} \label{hjb-solution-central}
{u}(x) =\left\{
                \begin{array}{ll}
                  -\frac{\widetilde{\sigma}^2p_1^{\prime \prime}(\widetilde{c}_1) \cosh \left(x\frac{\sqrt{2 \rho }}{\widetilde{\sigma}}\right)}{2 \rho  \cosh \left(\widetilde{c}_1 \frac{\sqrt{2 \rho }}{\widetilde{\sigma}} \right)} + p_1(x), \qquad \ \  0 \leq & x \leq \widetilde{c}_1,\\
                  u(\widetilde{c}_1)+\frac{K_2}{2}(x-\widetilde{c}_1), \qquad &x \geq \widetilde{c}_1,\\
                 u(-x), \qquad & x <0,
                \end{array}
              \right.
\end{eqnarray}
where $\widetilde{c}_1$ is the unique positive solution to \eqref{c1_tilde} and $p_1(x)$ is defined as in \eqref{p1}. The corresponding control of the regulator is a bang-bang type such that \eqref{distribution2}-\eqref{distribution3} hold. Furthermore, it is easy to see that $v(x^1,x^2):=u(x^1-x^2)$, with $u(x)$ defined in (\ref{hjb-solution-central}), is indeed the value function of problem \eqref{single-game}.

Next when $K_1=K_2$, $\xi^{1,+}$ and $\xi^{2,-}$ controls $Y_t$ in the same direction with  the same cost.  The same holds for  $\xi^{2,+}$ or $\xi^{1,-}$, hence the Pareto-optimal policy \eqref{distribution} {and \eqref{distribution_not_unique}.}

\end{proof}

\subsection{Benefits of regulation: Pareto optimum vs Nash equilibrium} 
We now use the above analytical results to compare the Pareto-optimal strategies with   Nash equilibrium strategies, whose definition we  recall:

\begin{definition}[Nash equilibrium]\label{nash} {\rm
 $ {\pmb{\eta}} = \left({\eta}^{1},\ldots,{\eta}^{N}\right) \in \mathcal{U}_N$ is a Nash equilibrium strategy of the stochastic game
({\bf{N-Player}}), if for any $ i=1,\ldots, N$, $\pmb{X}_{0-} = \pmb{x}$, and  any $\left( {\pmb{\eta}}^{-i},{\xi}^i\right) \in  \mathcal{U}_N$,
 the following inequality holds,
\begin{eqnarray*}
J^i \left(\pmb{x};  {\pmb{\eta}}\right)  \leq J^i \left(\pmb{x}; \left( {\pmb{\eta}}^{-i}, {\xi}^{i}\right)\right).
\end{eqnarray*} 
 $v^i(\pmb{x}):=J^i \left(\pmb{x};  {\pmb{\eta}}\right)$ is called the Nash equilibrium value for player $i$ associated with  $ {\pmb{\eta}}$.}
\end{definition}

\begin{proposition}[Pareto optimum vs Nash equilibrium solutions for $N=2$ players] \label{POvsNE}
Assume   \textbf{B1}-\textbf{B3} and $K_1=K_2=K$.
\begin{enumerate}[font=\bfseries,leftmargin=1.5\parindent]
\item[{\rm (i)}] The following controls give a Nash equilibrium policy to  game \eqref{both-dynamic}-\eqref{single-game}:
\begin{eqnarray}\label{nash-opt-control}
\begin{aligned} 
({\eta}_t^{1,+} ,{\eta}_t^{1,-} )&=\left(0, \max \left\{0,\max_{0 \leq u \leq t} \left \{(x^1-x^2) + (\pmb{\sigma}^1- \pmb{\sigma}^2)\cdot \pmb{B}_u  +{\eta}_u^{2,-} - {c}_2 \right\} \right\} \right),\\
({\eta}_t^{2,+} ,{\eta}_t^{2,-} )&= \left(0,\max \left\{0,\max_{0 \leq u \leq t}  \left \{-(x^1-x^2) +( \pmb{\sigma}^2- \pmb{\sigma}^1)\cdot \pmb{B}_u+{\eta}_u^{1,-}  - {c}_2 \right \}\right \} \right),
\end{aligned}
\end{eqnarray}
 where $c_2>0$ is the unique positive solution to
\begin{eqnarray}\label{c0}
\frac{\widetilde{\sigma}}{\sqrt{2\rho }} \tanh\left(\frac{\sqrt{2\rho }}{\widetilde{\sigma}}x \right)= \frac{p_1^{\prime}(x)-K}{p_1^{\prime \prime}(x)},
\end{eqnarray}
with $p_1$ defined in \eqref{p1}.
The  value functions $v^1$ and $v^2$ corresponding to the Nash equilibrium $\left( {\pmb{\eta}}^{1} , {\pmb{\eta}}^{2}\right)$ defined in (\ref{nash-opt-control}) are
  \begin{eqnarray}\label{v1}
    v^{1}(x^1,x^2)=\left\{
                \begin{array}{ll}
                v^1( x^2-c_2,x^2) , &   x^1 \leq x^2-c_2, \\
                -\frac{\widetilde{\sigma}^2p_1^{\prime \prime}(c_2) \cosh\left(\frac{\sqrt{2\rho }}{\widetilde{\sigma}}(x^1-x^2)\right)}{2\rho   \cosh \left(c_2\frac{\sqrt{2\rho }}{\widetilde{\sigma}}\right)} + p_1(x^1-x^2), & x^2-c_2 \leq x^1 \leq x^2 + c_2,\\
        K (x^1 -x^2- c_2)  +v^1(x^2+c_2,x^2),    & x^1 \geq x^2  +c_2,
                \end{array}
              \right.
 \end{eqnarray}
and
\begin{eqnarray}\label{v2}
    v^{2}(x^1,x^2)=\left\{
                \begin{array}{ll}
                v^2( x^1,x^1-c_2), & x^2 \leq x^1-c_2, \\
                  -\frac{\widetilde{\sigma}^2p_1^{\prime \prime}(c_2) \cosh \left(\frac{\sqrt{2\rho }}{\widetilde{\sigma}}(x^2-x^1)\right)}{2\rho   \cosh \left(c_2\frac{\sqrt{2\rho }}{\widetilde{\sigma}}\right)} + p_1(x^2-x^1),& x^1 -c_2 \leq x^2 \leq x^1  + c_2,\\
         K( x^2 -x^1- c_2)  +v^2(x^1,x^1+c_2),       & x^2 \geq x^1   +c_2;
          
                \end{array}
              \right.
  \end{eqnarray} 
 \item[{\rm (ii)}]  $c_2>c_1$, where $c_1$ is the unique positive solution to \eqref{c1} and $c_2$ is the unique positive solution to \eqref{c0}.  
 \end{enumerate}
\end{proposition}
That is, a Pareto-optimal policy yields a tighter threshold for spreads,  hence reduces volatility of  interbank rates compared to the Nash equilibrium (see Figure \ref{figure3}).
\begin{figure}[H]  
       
     \centering \includegraphics[width=0.6\columnwidth]{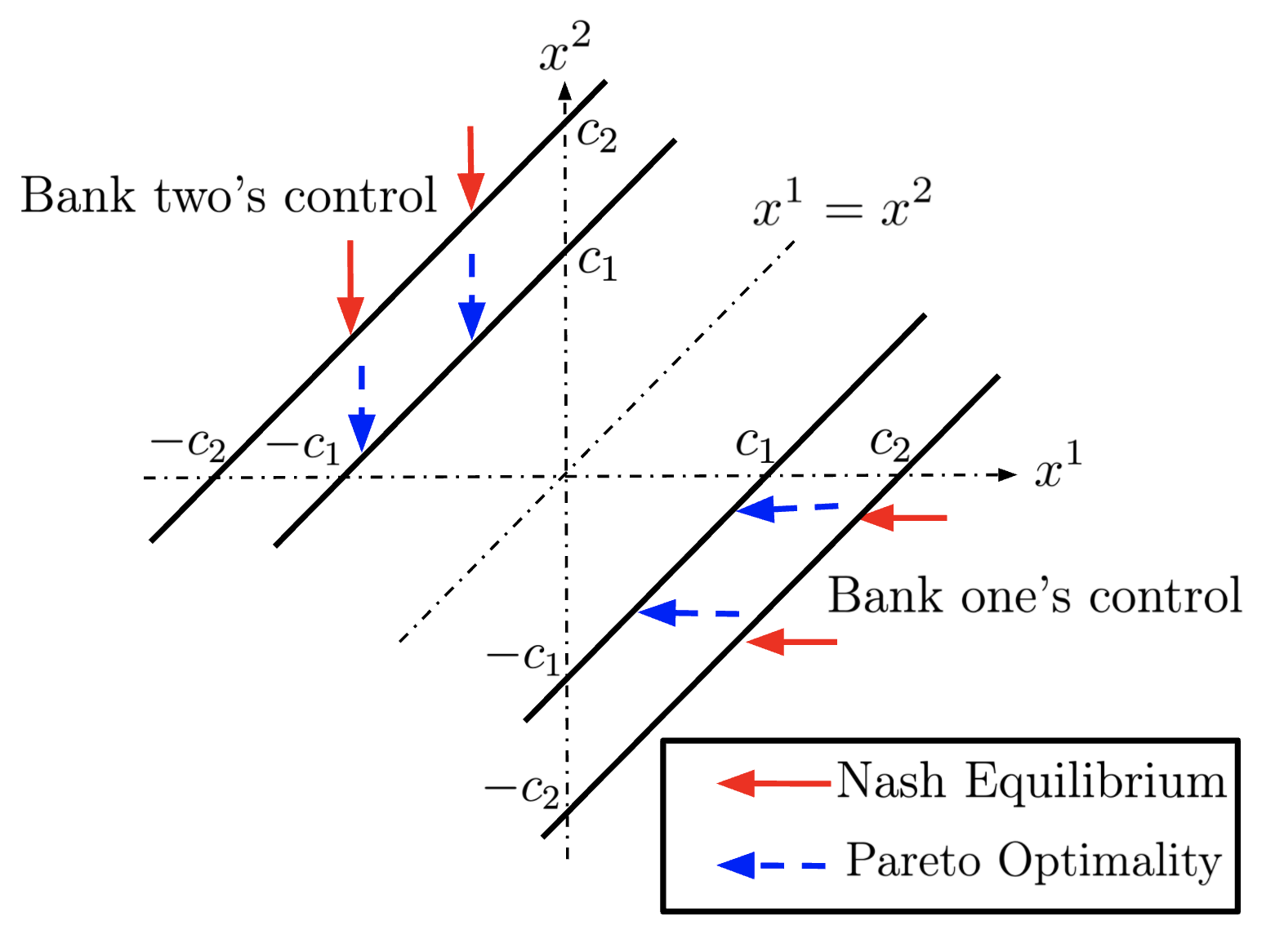}
       
        \caption{
                \label{figure3} 
               Comparison: Nash and Pareto ($K_1=K_2$).
        }
\end{figure}

\begin{proof}
Similar to the derivation in \citep{GX2019}, we have the following quasi-variational inequalities for the Nash equilibrium of game \eqref{both-dynamic} with $J^1$ and $J^2$ and $K_1=K_2=K$,
\begin{eqnarray}\label{qvi}
    \begin{cases} 
   & \max\left\{\rho  v^i(x^1,x^2) - h(x^1-x^2)-\frac{\widetilde{\sigma}^2}{2}\Big(\partial^2_{x^1}v^i(x^1,x^2)+\partial^2_{x^2}v^i(x^1,x^2)\Big),\right.\\
 &\hspace{150pt} \partial_{x^i}v^i(x^1,x^2)-K,-\partial_{x^i}v^i(x^1,x^2)-K\Big\} = 0,\nonumber \\
    &\hspace{210pt} \mbox{on} \hspace{10pt}\left\{(x^1,x^2): -K<\partial_{x^j}v^j (x^1,x^2)<K\right\},\\
  & \partial_{x^i} v^i(x^1,x^2) = 0, \hspace{120pt} \mbox{on} \hspace{10pt} \left\{(x^1,x^2): \partial_{x^j}v^j(x^1,x^2) = K \text{ or }\partial_{x^j}v^j (x^1,x^2)= -K\right\},
\end{cases}
\end{eqnarray}
for $i\neq j$ and $1\leq i,j \leq 2$.
Moreover, one can show that \eqref{v1}-\eqref{v2} are the solution to \eqref{qvi}. Applying a verification theorem \citep[
Theorem 3]{GX2019}, some further calculations can verify that  \eqref{v1}-\eqref{v2} are the game values associated with the Nash equilibrium policy \eqref{nash-opt-control}.

Now we provide the proof for Claim (ii). Define $g(x) =  \frac{\widetilde{\sigma}}{\sqrt{2\rho }} \tanh \left(\frac{\sqrt{2\rho }}{\widetilde{\sigma}}x\right)$, $g_1(x) = \frac{p_1^{\prime}(x)-\frac{K}{2}}{p_1^{\prime\prime}(x)}$ and $g_2(x) = \frac{p_1^{\prime}(x)-K}{p_1^{\prime\prime}(x)}$, where $p_1$ is defined in (\ref{p1}). Then $g(0)=0$,
$g^{\prime}(x)>0$ $\textit{ for any } x \in \mathbb{R^+}$, and $lim_{x \rightarrow \infty}g(x) = \frac{\widetilde{\sigma}}{\sqrt{2\rho }}$. Thanks to Assumption $({\bf B2})$, $$0<\frac{c}{\rho } \leq p_1^{\prime \prime}(x) = \mathbb{E} \int_0^{\infty} e^{-\rho t} h^{\prime \prime}({x}+\widetilde{\sigma}B_t)dt \leq \frac{C}{\rho }.$$
The function $p^{\prime}_1(x)$ is negative at $x=0$ and increases monotonically to $\infty$ on $\mathbb{R}^+$. Hence there exists an unique positive zero ${c}_0$.  Moreover, for any $ x >{c}_0$,
$g_1^{\prime}(x) = 1-\frac{p^{\prime\prime\prime}_1(x)}{p^{\prime\prime}_1(x)} g_1(x) \geq 1$. This is because $p^{\prime\prime\prime}_1(x) \leq 0 $ for $x \geq 0$.
We conclude that there exists a unique point ${c}_0<c_1<\infty$ such that $g(c_1) = g_1(c_1)$.

Now by similar analysis, $c_2$ is the unique solution to $g(x) = g_2(x)$ such that $0<c_2<\infty$. Notice that, $g_1(x)-g_2(x) = \frac{K}{2p_1^{\prime \prime}(x)}>0$ because  $p_2^{\prime \prime}(x)>0$. Hence $c_2>c_1$.
\end{proof}

\bibliographystyle{apacite}
\bibliography{PO.bib}

\begin{thebibliography}{}

\bibitem [\protect \citeauthoryear {%
A{\"\i}d%
, Basei%
\BCBL {}\ \BBA {} Pham%
}{%
A{\"\i}d%
\ \protect \BOthers {.}}{%
{\protect \APACyear {2017}}%
}]{%
ABP2017}
\APACinsertmetastar {%
ABP2017}%
\begin{APACrefauthors}%
A{\"\i}d, R.%
, Basei, M.%
\BCBL {}\ \BBA {} Pham, H.%
\end{APACrefauthors}%
\unskip\
\newblock
\APACrefYearMonthDay{2017}{}{}.
\newblock
{\BBOQ}\APACrefatitle {The coordination of centralised and distributed
  generation} {The coordination of centralised and distributed
  generation}.{\BBCQ}
\newblock
\APACjournalVolNumPages{arXiv preprint arXiv:1705.01302}{}{}{}.
\PrintBackRefs{\CurrentBib}

\bibitem [\protect \citeauthoryear {%
Avellaneda%
\ \BBA {} Cont%
}{%
Avellaneda%
\ \BBA {} Cont%
}{%
{\protect \APACyear {2010}}%
}]{%
avellaneda2010}
\APACinsertmetastar {%
avellaneda2010}%
\begin{APACrefauthors}%
Avellaneda, M.%
\BCBT {}\ \BBA {} Cont, R.%
\end{APACrefauthors}%
\unskip\
\newblock
\APACrefYearMonthDay{2010}{}{}.
\newblock
\APACrefbtitle {Transparency in over-the-counter interest rate derivatives
  markets} {Transparency in over-the-counter interest rate derivatives
  markets}\ \APACbVolEdTR {}{Report}.
\newblock
\APACaddressInstitution{}{Finance Concepts}.
\PrintBackRefs{\CurrentBib}

\bibitem [\protect \citeauthoryear {%
Bator%
}{%
Bator%
}{%
{\protect \APACyear {1957}}%
}]{%
Bator1957}
\APACinsertmetastar {%
Bator1957}%
\begin{APACrefauthors}%
Bator, F\BPBI M.%
\end{APACrefauthors}%
\unskip\
\newblock
\APACrefYearMonthDay{1957}{}{}.
\newblock
{\BBOQ}\APACrefatitle {The simple analytics of welfare maximization} {The
  simple analytics of welfare maximization}.{\BBCQ}
\newblock
\APACjournalVolNumPages{The American Economic Review}{47}{1}{22--59}.
\PrintBackRefs{\CurrentBib}

\bibitem [\protect \citeauthoryear {%
Bene{\v{s}}%
, Shepp%
\BCBL {}\ \BBA {} Witsenhausen%
}{%
Bene{\v{s}}%
\ \protect \BOthers {.}}{%
{\protect \APACyear {1980}}%
}]{%
BSW1980}
\APACinsertmetastar {%
BSW1980}%
\begin{APACrefauthors}%
Bene{\v{s}}, V\BPBI E.%
, Shepp, L\BPBI A.%
\BCBL {}\ \BBA {} Witsenhausen, H\BPBI S.%
\end{APACrefauthors}%
\unskip\
\newblock
\APACrefYearMonthDay{1980}{}{}.
\newblock
{\BBOQ}\APACrefatitle {Some solvable stochastic control problems} {Some
  solvable stochastic control problems}.{\BBCQ}
\newblock
\APACjournalVolNumPages{Stochastics: An International Journal of Probability
  and Stochastic Processes}{4}{1}{39--83}.
\PrintBackRefs{\CurrentBib}

\bibitem [\protect \citeauthoryear {%
Bensoussan%
, Long%
, Perera%
\BCBL {}\ \BBA {} Sethi%
}{%
Bensoussan%
\ \protect \BOthers {.}}{%
{\protect \APACyear {2012}}%
}]{%
bensoussan2012}
\APACinsertmetastar {%
bensoussan2012}%
\begin{APACrefauthors}%
Bensoussan, A.%
, Long, H.%
, Perera, S.%
\BCBL {}\ \BBA {} Sethi, S.%
\end{APACrefauthors}%
\unskip\
\newblock
\APACrefYearMonthDay{2012}{}{}.
\newblock
{\BBOQ}\APACrefatitle {Impulse control with random reaction periods: a central
  bank intervention problem} {Impulse control with random reaction periods: a
  central bank intervention problem}.{\BBCQ}
\newblock
\APACjournalVolNumPages{{Operations Research Letters}}{40}{6}{425--430}.
\PrintBackRefs{\CurrentBib}

\bibitem [\protect \citeauthoryear {%
Brezis%
}{%
Brezis%
}{%
{\protect \APACyear {2010}}%
}]{%
brezis2010functional}
\APACinsertmetastar {%
brezis2010functional}%
\begin{APACrefauthors}%
Brezis, H.%
\end{APACrefauthors}%
\unskip\
\newblock
\APACrefYear{2010}.
\newblock
\APACrefbtitle {Functional {A}nalysis, {S}obolev {S}paces and {P}artial
  {D}ifferential {E}quations} {Functional {A}nalysis, {S}obolev {S}paces and
  {P}artial {D}ifferential {E}quations}.
\newblock
\APACaddressPublisher{}{Springer}.
\PrintBackRefs{\CurrentBib}

\bibitem [\protect \citeauthoryear {%
Cadenillas%
\ \BBA {} Zapatero%
}{%
Cadenillas%
\ \BBA {} Zapatero%
}{%
{\protect \APACyear {2000}}%
}]{%
cadenillas2000}
\APACinsertmetastar {%
cadenillas2000}%
\begin{APACrefauthors}%
Cadenillas, A.%
\BCBT {}\ \BBA {} Zapatero, F.%
\end{APACrefauthors}%
\unskip\
\newblock
\APACrefYearMonthDay{2000}{}{}.
\newblock
{\BBOQ}\APACrefatitle {Classical and impulse stochastic control of the exchange
  rate using interest rates and reserves} {Classical and impulse stochastic
  control of the exchange rate using interest rates and reserves}.{\BBCQ}
\newblock
\APACjournalVolNumPages{Mathematical Finance}{10}{2}{141--156}.
\PrintBackRefs{\CurrentBib}

\bibitem [\protect \citeauthoryear {%
Carlen%
\ \BBA {} Protter%
}{%
Carlen%
\ \BBA {} Protter%
}{%
{\protect \APACyear {1992}}%
}]{%
carlen1992semimartingale}
\APACinsertmetastar {%
carlen1992semimartingale}%
\begin{APACrefauthors}%
Carlen, E.%
\BCBT {}\ \BBA {} Protter, P.%
\end{APACrefauthors}%
\unskip\
\newblock
\APACrefYearMonthDay{1992}{}{}.
\newblock
{\BBOQ}\APACrefatitle {On semimartingale decompositions of convex functions of
  semimartingales} {On semimartingale decompositions of convex functions of
  semimartingales}.{\BBCQ}
\newblock
\APACjournalVolNumPages{Illinois journal of mathematics}{36}{3}{420--427}.
\PrintBackRefs{\CurrentBib}

\bibitem [\protect \citeauthoryear {%
Carmona%
, Fouque%
\BCBL {}\ \BBA {} Sun%
}{%
Carmona%
\ \protect \BOthers {.}}{%
{\protect \APACyear {2015}}%
}]{%
CFS2013}
\APACinsertmetastar {%
CFS2013}%
\begin{APACrefauthors}%
Carmona, R.%
, Fouque, J\BHBI P.%
\BCBL {}\ \BBA {} Sun, L\BHBI H.%
\end{APACrefauthors}%
\unskip\
\newblock
\APACrefYearMonthDay{2015}{}{}.
\newblock
{\BBOQ}\APACrefatitle {Mean field games and systemic risk} {Mean field games
  and systemic risk}.{\BBCQ}
\newblock
\APACjournalVolNumPages{Communications in Mathematical
  Sciences}{13}{4}{911--933}.
\PrintBackRefs{\CurrentBib}

\bibitem [\protect \citeauthoryear {%
Chiarolla%
, Ferrari%
\BCBL {}\ \BBA {} Riedel%
}{%
Chiarolla%
\ \protect \BOthers {.}}{%
{\protect \APACyear {2013}}%
}]{%
CFR2013}
\APACinsertmetastar {%
CFR2013}%
\begin{APACrefauthors}%
Chiarolla, M\BPBI B.%
, Ferrari, G.%
\BCBL {}\ \BBA {} Riedel, F.%
\end{APACrefauthors}%
\unskip\
\newblock
\APACrefYearMonthDay{2013}{}{}.
\newblock
{\BBOQ}\APACrefatitle {Generalized {K}uhn--{T}ucker Conditions for {N}-Firm
  Stochastic Irreversible Investment under Limited Resources} {Generalized
  {K}uhn--{T}ucker conditions for {N}-firm stochastic irreversible investment
  under limited resources}.{\BBCQ}
\newblock
\APACjournalVolNumPages{SIAM Journal on Control and
  Optimization}{51}{5}{3863--3885}.
\PrintBackRefs{\CurrentBib}

\bibitem [\protect \citeauthoryear {%
Coleman%
}{%
Coleman%
}{%
{\protect \APACyear {1979}}%
}]{%
Coleman1979}
\APACinsertmetastar {%
Coleman1979}%
\begin{APACrefauthors}%
Coleman, J\BPBI L.%
\end{APACrefauthors}%
\unskip\
\newblock
\APACrefYearMonthDay{1979}{}{}.
\newblock
{\BBOQ}\APACrefatitle {Efficiency, utility, and wealth maximization}
  {Efficiency, utility, and wealth maximization}.{\BBCQ}
\newblock
\APACjournalVolNumPages{Hofstra L. Rev.}{8}{}{509}.
\PrintBackRefs{\CurrentBib}

\bibitem [\protect \citeauthoryear {%
Davis%
\ \BBA {} Norman%
}{%
Davis%
\ \BBA {} Norman%
}{%
{\protect \APACyear {1990}}%
}]{%
davis1990portfolio}
\APACinsertmetastar {%
davis1990portfolio}%
\begin{APACrefauthors}%
Davis, M\BPBI H.%
\BCBT {}\ \BBA {} Norman, A\BPBI R.%
\end{APACrefauthors}%
\unskip\
\newblock
\APACrefYearMonthDay{1990}{}{}.
\newblock
{\BBOQ}\APACrefatitle {Portfolio selection with transaction costs} {Portfolio
  selection with transaction costs}.{\BBCQ}
\newblock
\APACjournalVolNumPages{Mathematics of Operations Research}{15}{4}{676--713}.
\PrintBackRefs{\CurrentBib}

\bibitem [\protect \citeauthoryear {%
Davis%
, Panas%
\BCBL {}\ \BBA {} Zariphopoulou%
}{%
Davis%
\ \protect \BOthers {.}}{%
{\protect \APACyear {1993}}%
}]{%
davis1993european}
\APACinsertmetastar {%
davis1993european}%
\begin{APACrefauthors}%
Davis, M\BPBI H.%
, Panas, V\BPBI G.%
\BCBL {}\ \BBA {} Zariphopoulou, T.%
\end{APACrefauthors}%
\unskip\
\newblock
\APACrefYearMonthDay{1993}{}{}.
\newblock
{\BBOQ}\APACrefatitle {European option pricing with transaction costs}
  {European option pricing with transaction costs}.{\BBCQ}
\newblock
\APACjournalVolNumPages{SIAM Journal on Control and
  Optimization}{31}{2}{470--493}.
\PrintBackRefs{\CurrentBib}

\bibitem [\protect \citeauthoryear {%
De~Angelis%
\ \BBA {} Ferrari%
}{%
De~Angelis%
\ \BBA {} Ferrari%
}{%
{\protect \APACyear {2018}}%
}]{%
DF2018}
\APACinsertmetastar {%
DF2018}%
\begin{APACrefauthors}%
De~Angelis, T.%
\BCBT {}\ \BBA {} Ferrari, G.%
\end{APACrefauthors}%
\unskip\
\newblock
\APACrefYearMonthDay{2018}{}{}.
\newblock
{\BBOQ}\APACrefatitle {Stochastic nonzero-sum games: a new connection between
  singular control and optimal stopping} {Stochastic nonzero-sum games: a new
  connection between singular control and optimal stopping}.{\BBCQ}
\newblock
\APACjournalVolNumPages{Advances in Applied Probability}{50}{2}{347--372}.
\PrintBackRefs{\CurrentBib}

\bibitem [\protect \citeauthoryear {%
Dianetti%
\ \BBA {} Ferrari%
}{%
Dianetti%
\ \BBA {} Ferrari%
}{%
{\protect \APACyear {2020}}%
}]{%
dianetti2020nonzero}
\APACinsertmetastar {%
dianetti2020nonzero}%
\begin{APACrefauthors}%
Dianetti, J.%
\BCBT {}\ \BBA {} Ferrari, G.%
\end{APACrefauthors}%
\unskip\
\newblock
\APACrefYearMonthDay{2020}{}{}.
\newblock
{\BBOQ}\APACrefatitle {Nonzero-Sum Submodular Monotone-Follower Games:
  existence and Approximation of {N}ash Equilibria} {Nonzero-sum submodular
  monotone-follower games: existence and approximation of {N}ash
  equilibria}.{\BBCQ}
\newblock
\APACjournalVolNumPages{SIAM Journal on Control and
  Optimization}{58}{3}{1257--1288}.
\PrintBackRefs{\CurrentBib}

\bibitem [\protect \citeauthoryear {%
Duffie%
\ \BBA {} Stein%
}{%
Duffie%
\ \BBA {} Stein%
}{%
{\protect \APACyear {2015}}%
}]{%
duffie2015}
\APACinsertmetastar {%
duffie2015}%
\begin{APACrefauthors}%
Duffie, D.%
\BCBT {}\ \BBA {} Stein, J\BPBI C.%
\end{APACrefauthors}%
\unskip\
\newblock
\APACrefYearMonthDay{2015}{}{}.
\newblock
{\BBOQ}\APACrefatitle {Reforming {LIBOR} and other financial market benchmarks}
  {Reforming {LIBOR} and other financial market benchmarks}.{\BBCQ}
\newblock
\APACjournalVolNumPages{Journal of Economic Perspectives}{29}{2}{191--212}.
\PrintBackRefs{\CurrentBib}

\bibitem [\protect \citeauthoryear {%
Dupuis%
\ \BBA {} Ishii%
}{%
Dupuis%
\ \BBA {} Ishii%
}{%
{\protect \APACyear {1991}}%
}]{%
DI1991}
\APACinsertmetastar {%
DI1991}%
\begin{APACrefauthors}%
Dupuis, P.%
\BCBT {}\ \BBA {} Ishii, H.%
\end{APACrefauthors}%
\unskip\
\newblock
\APACrefYearMonthDay{1991}{}{}.
\newblock
{\BBOQ}\APACrefatitle {On {L}ipschitz continuity of the solution mapping to the
  {S}korokhod problem, with applications} {On {L}ipschitz continuity of the
  solution mapping to the {S}korokhod problem, with applications}.{\BBCQ}
\newblock
\APACjournalVolNumPages{Stochastics}{35}{1}{31--62}.
\PrintBackRefs{\CurrentBib}

\bibitem [\protect \citeauthoryear {%
Dupuis%
\ \BBA {} Ishii%
}{%
Dupuis%
\ \BBA {} Ishii%
}{%
{\protect \APACyear {1993}}%
}]{%
DI1993}
\APACinsertmetastar {%
DI1993}%
\begin{APACrefauthors}%
Dupuis, P.%
\BCBT {}\ \BBA {} Ishii, H.%
\end{APACrefauthors}%
\unskip\
\newblock
\APACrefYearMonthDay{1993}{}{}.
\newblock
{\BBOQ}\APACrefatitle {{SDEs with Oblique Reflection on Nonsmooth Domains}}
  {{SDEs with Oblique Reflection on Nonsmooth Domains}}.{\BBCQ}
\newblock
\APACjournalVolNumPages{{Annals of Probability}}{21}{1}{554 -- 580}.
\newblock
\begin{APACrefDOI} \doi{10.1214/aop/1176989415} \end{APACrefDOI}
\PrintBackRefs{\CurrentBib}

\bibitem [\protect \citeauthoryear {%
Evans%
}{%
Evans%
}{%
{\protect \APACyear {1990}}%
}]{%
evans1990weak}
\APACinsertmetastar {%
evans1990weak}%
\begin{APACrefauthors}%
Evans, L\BPBI C.%
\end{APACrefauthors}%
\unskip\
\newblock
\APACrefYear{1990}.
\newblock
\APACrefbtitle {Weak Convergence Methods for Nonlinear Partial Differential
  Equations} {Weak convergence methods for nonlinear partial differential
  equations}.
\newblock
\APACaddressPublisher{}{American Mathematical Society}.
\PrintBackRefs{\CurrentBib}

\bibitem [\protect \citeauthoryear {%
Ferrari%
, Riedel%
\BCBL {}\ \BBA {} Steg%
}{%
Ferrari%
\ \protect \BOthers {.}}{%
{\protect \APACyear {2017}}%
}]{%
FRS2017}
\APACinsertmetastar {%
FRS2017}%
\begin{APACrefauthors}%
Ferrari, G.%
, Riedel, F.%
\BCBL {}\ \BBA {} Steg, J\BHBI H.%
\end{APACrefauthors}%
\unskip\
\newblock
\APACrefYearMonthDay{2017}{}{}.
\newblock
{\BBOQ}\APACrefatitle {Continuous-time public good contribution under
  uncertainty: {a} stochastic control approach} {Continuous-time public good
  contribution under uncertainty: {a} stochastic control approach}.{\BBCQ}
\newblock
\APACjournalVolNumPages{Applied Mathematics \& Optimization}{75}{3}{429--470}.
\PrintBackRefs{\CurrentBib}

\bibitem [\protect \citeauthoryear {%
Fischer%
\ \BBA {} Livieri%
}{%
Fischer%
\ \BBA {} Livieri%
}{%
{\protect \APACyear {2016}}%
}]{%
FL2016}
\APACinsertmetastar {%
FL2016}%
\begin{APACrefauthors}%
Fischer, M.%
\BCBT {}\ \BBA {} Livieri, G.%
\end{APACrefauthors}%
\unskip\
\newblock
\APACrefYearMonthDay{2016}{}{}.
\newblock
{\BBOQ}\APACrefatitle {Continuous time mean-variance portfolio optimization
  through the mean field approach} {Continuous time mean-variance portfolio
  optimization through the mean field approach}.{\BBCQ}
\newblock
\APACjournalVolNumPages{ESAIM: Probability and Statistics}{20}{}{30--44}.
\PrintBackRefs{\CurrentBib}

\bibitem [\protect \citeauthoryear {%
Gilbarg%
\ \BBA {} Trudinger%
}{%
Gilbarg%
\ \BBA {} Trudinger%
}{%
{\protect \APACyear {2015}}%
}]{%
GT2015}
\APACinsertmetastar {%
GT2015}%
\begin{APACrefauthors}%
Gilbarg, D.%
\BCBT {}\ \BBA {} Trudinger, N\BPBI S.%
\end{APACrefauthors}%
\unskip\
\newblock
\APACrefYear{2015}.
\newblock
\APACrefbtitle {Elliptic {P}artial {D}ifferential {E}quations of {S}econd
  {O}rder} {Elliptic {P}artial {D}ifferential {E}quations of {S}econd {O}rder}.
\newblock
\APACaddressPublisher{}{Springer}.
\PrintBackRefs{\CurrentBib}

\bibitem [\protect \citeauthoryear {%
Gomes%
, Mohr%
\BCBL {}\ \BBA {} Souza%
}{%
Gomes%
\ \protect \BOthers {.}}{%
{\protect \APACyear {2010}}%
}]{%
GMS2010}
\APACinsertmetastar {%
GMS2010}%
\begin{APACrefauthors}%
Gomes, D\BPBI A.%
, Mohr, J.%
\BCBL {}\ \BBA {} Souza, R\BPBI R.%
\end{APACrefauthors}%
\unskip\
\newblock
\APACrefYearMonthDay{2010}{}{}.
\newblock
{\BBOQ}\APACrefatitle {Discrete time, finite state space mean field games}
  {Discrete time, finite state space mean field games}.{\BBCQ}
\newblock
\APACjournalVolNumPages{Journal de Math{\'e}matiques Pures et
  Appliqu{\'e}es}{93}{3}{308--328}.
\PrintBackRefs{\CurrentBib}

\bibitem [\protect \citeauthoryear {%
Guo%
\ \BBA {} Pham%
}{%
Guo%
\ \BBA {} Pham%
}{%
{\protect \APACyear {2005}}%
}]{%
GP2005}
\APACinsertmetastar {%
GP2005}%
\begin{APACrefauthors}%
Guo, X.%
\BCBT {}\ \BBA {} Pham, H.%
\end{APACrefauthors}%
\unskip\
\newblock
\APACrefYearMonthDay{2005}{}{}.
\newblock
{\BBOQ}\APACrefatitle {Optimal partially reversible investment with entry
  decision and general production function} {Optimal partially reversible
  investment with entry decision and general production function}.{\BBCQ}
\newblock
\APACjournalVolNumPages{Stochastic Processes and their
  Applications}{115}{5}{705--736}.
\PrintBackRefs{\CurrentBib}

\bibitem [\protect \citeauthoryear {%
Guo%
\ \BBA {} Xu%
}{%
Guo%
\ \BBA {} Xu%
}{%
{\protect \APACyear {2019}}%
}]{%
GX2019}
\APACinsertmetastar {%
GX2019}%
\begin{APACrefauthors}%
Guo, X.%
\BCBT {}\ \BBA {} Xu, R.%
\end{APACrefauthors}%
\unskip\
\newblock
\APACrefYearMonthDay{2019}{}{}.
\newblock
{\BBOQ}\APACrefatitle {Stochastic Games for Fuel Follower Problem: {N} versus
  {MFG}} {Stochastic games for fuel follower problem: {N} versus {MFG}}.{\BBCQ}
\newblock
\APACjournalVolNumPages{SIAM Journal on Control and
  Optimization}{57}{1}{659--692}.
\PrintBackRefs{\CurrentBib}

\bibitem [\protect \citeauthoryear {%
{H. M. Treasury}%
}{%
{H. M. Treasury}%
}{%
{\protect \APACyear {2012}}%
}]{%
wheatley}
\APACinsertmetastar {%
wheatley}%
\begin{APACrefauthors}%
{H. M. Treasury}.%
\end{APACrefauthors}%
\unskip\
\newblock
\APACrefYearMonthDay{2012}{}{}.
\newblock
\APACrefbtitle {{The Wheatley Review of {LIBOR}: Final Report}} {{The Wheatley
  Review of {LIBOR}: Final Report}}\ \APACbVolEdTR{}{\BTR{}}.
\newblock
\APACaddressInstitutionEqAuth{}{{H. M. Treasury}}.
\newblock
\begin{APACrefURL}
  \url{https://assets.publishing.service.gov.uk/government/uploads/system/uploads/attachment_data/file/191762/wheatley_review_libor_finalreport_280912.pdf}
  \end{APACrefURL}
\PrintBackRefs{\CurrentBib}

\bibitem [\protect \citeauthoryear {%
Hernandez-Hernandez%
, Simon%
\BCBL {}\ \BBA {} Zervos%
}{%
Hernandez-Hernandez%
\ \protect \BOthers {.}}{%
{\protect \APACyear {2015}}%
}]{%
hernandez2015zero}
\APACinsertmetastar {%
hernandez2015zero}%
\begin{APACrefauthors}%
Hernandez-Hernandez, D.%
, Simon, R\BPBI S.%
\BCBL {}\ \BBA {} Zervos, M.%
\end{APACrefauthors}%
\unskip\
\newblock
\APACrefYearMonthDay{2015}{}{}.
\newblock
{\BBOQ}\APACrefatitle {A zero-sum game between a singular stochastic controller
  and a discretionary stopper} {A zero-sum game between a singular stochastic
  controller and a discretionary stopper}.{\BBCQ}
\newblock
\APACjournalVolNumPages{Annals of Applied Probability}{25}{1}{46--80}.
\PrintBackRefs{\CurrentBib}

\bibitem [\protect \citeauthoryear {%
Huang%
, Malham{\'e}%
\BCBL {}\ \BBA {} Caines%
}{%
Huang%
\ \protect \BOthers {.}}{%
{\protect \APACyear {2006}}%
}]{%
HMC2006}
\APACinsertmetastar {%
HMC2006}%
\begin{APACrefauthors}%
Huang, M.%
, Malham{\'e}, R\BPBI P.%
\BCBL {}\ \BBA {} Caines, P\BPBI E.%
\end{APACrefauthors}%
\unskip\
\newblock
\APACrefYearMonthDay{2006}{}{}.
\newblock
{\BBOQ}\APACrefatitle {Large population stochastic dynamic games: closed-loop
  {M}c{K}ean-{V}lasov systems and the {N}ash certainty equivalence principle}
  {Large population stochastic dynamic games: closed-loop {M}c{K}ean-{V}lasov
  systems and the {N}ash certainty equivalence principle}.{\BBCQ}
\newblock
\APACjournalVolNumPages{Communications in Information \&
  Systems}{6}{3}{221--252}.
\PrintBackRefs{\CurrentBib}

\bibitem [\protect \citeauthoryear {%
Jeanblanc-Picqu{\'e}%
}{%
Jeanblanc-Picqu{\'e}%
}{%
{\protect \APACyear {1993}}%
}]{%
jeanblanc1993impulse}
\APACinsertmetastar {%
jeanblanc1993impulse}%
\begin{APACrefauthors}%
Jeanblanc-Picqu{\'e}, M.%
\end{APACrefauthors}%
\unskip\
\newblock
\APACrefYearMonthDay{1993}{}{}.
\newblock
{\BBOQ}\APACrefatitle {Impulse control method and exchange rate} {Impulse
  control method and exchange rate}.{\BBCQ}
\newblock
\APACjournalVolNumPages{Mathematical Finance}{3}{2}{161--177}.
\PrintBackRefs{\CurrentBib}

\bibitem [\protect \citeauthoryear {%
Kallsen%
\ \BBA {} Muhle-Karbe%
}{%
Kallsen%
\ \BBA {} Muhle-Karbe%
}{%
{\protect \APACyear {2017}}%
}]{%
kallsen2017general}
\APACinsertmetastar {%
kallsen2017general}%
\begin{APACrefauthors}%
Kallsen, J.%
\BCBT {}\ \BBA {} Muhle-Karbe, J.%
\end{APACrefauthors}%
\unskip\
\newblock
\APACrefYearMonthDay{2017}{}{}.
\newblock
{\BBOQ}\APACrefatitle {The general structure of optimal investment and
  consumption with small transaction costs} {The general structure of optimal
  investment and consumption with small transaction costs}.{\BBCQ}
\newblock
\APACjournalVolNumPages{Mathematical Finance}{27}{3}{659--703}.
\PrintBackRefs{\CurrentBib}

\bibitem [\protect \citeauthoryear {%
Karatzas%
}{%
Karatzas%
}{%
{\protect \APACyear {1983}}%
}]{%
karatzas1982}
\APACinsertmetastar {%
karatzas1982}%
\begin{APACrefauthors}%
Karatzas, I.%
\end{APACrefauthors}%
\unskip\
\newblock
\APACrefYearMonthDay{1983}{}{}.
\newblock
{\BBOQ}\APACrefatitle {A class of singular stochastic control problems} {A
  class of singular stochastic control problems}.{\BBCQ}
\newblock
\APACjournalVolNumPages{{Advances in Applied Probability}}{15}{2}{225 -- 254}.
\PrintBackRefs{\CurrentBib}

\bibitem [\protect \citeauthoryear {%
Kruk%
}{%
Kruk%
}{%
{\protect \APACyear {2000}}%
}]{%
kruk2000}
\APACinsertmetastar {%
kruk2000}%
\begin{APACrefauthors}%
Kruk, L.%
\end{APACrefauthors}%
\unskip\
\newblock
\APACrefYearMonthDay{2000}{}{}.
\newblock
{\BBOQ}\APACrefatitle {Optimal policies for {N}-dimensional singular stochastic
  control problems part {I}: the {S}korokhod problem} {Optimal policies for
  {N}-dimensional singular stochastic control problems part {I}: the
  {S}korokhod problem}.{\BBCQ}
\newblock
\APACjournalVolNumPages{SIAM Journal on Control and
  Optimization}{38}{5}{1603--1622}.
\PrintBackRefs{\CurrentBib}

\bibitem [\protect \citeauthoryear {%
Kwon%
\ \BBA {} Zhang%
}{%
Kwon%
\ \BBA {} Zhang%
}{%
{\protect \APACyear {2015}}%
}]{%
KZ2015}
\APACinsertmetastar {%
KZ2015}%
\begin{APACrefauthors}%
Kwon, H.%
\BCBT {}\ \BBA {} Zhang, H.%
\end{APACrefauthors}%
\unskip\
\newblock
\APACrefYearMonthDay{2015}{}{}.
\newblock
{\BBOQ}\APACrefatitle {Game of singular stochastic control and strategic exit}
  {Game of singular stochastic control and strategic exit}.{\BBCQ}
\newblock
\APACjournalVolNumPages{Mathematics of Operations Research}{40}{4}{869--887}.
\PrintBackRefs{\CurrentBib}

\bibitem [\protect \citeauthoryear {%
Lasry%
\ \BBA {} Lions%
}{%
Lasry%
\ \BBA {} Lions%
}{%
{\protect \APACyear {2007}}%
}]{%
LL2007}
\APACinsertmetastar {%
LL2007}%
\begin{APACrefauthors}%
Lasry, J\BHBI M.%
\BCBT {}\ \BBA {} Lions, P\BHBI L.%
\end{APACrefauthors}%
\unskip\
\newblock
\APACrefYearMonthDay{2007}{}{}.
\newblock
{\BBOQ}\APACrefatitle {Mean field games} {Mean field games}.{\BBCQ}
\newblock
\APACjournalVolNumPages{Japanese Journal of Mathematics}{2}{1}{229--260}.
\PrintBackRefs{\CurrentBib}

\bibitem [\protect \citeauthoryear {%
Menaldi%
\ \BBA {} Robin%
}{%
Menaldi%
\ \BBA {} Robin%
}{%
{\protect \APACyear {1983}}%
}]{%
MR1983}
\APACinsertmetastar {%
MR1983}%
\begin{APACrefauthors}%
Menaldi, J\BHBI L.%
\BCBT {}\ \BBA {} Robin, M.%
\end{APACrefauthors}%
\unskip\
\newblock
\APACrefYearMonthDay{1983}{}{}.
\newblock
{\BBOQ}\APACrefatitle {On some cheap control problems for diffusion processes}
  {On some cheap control problems for diffusion processes}.{\BBCQ}
\newblock
\APACjournalVolNumPages{Transactions of the American Mathematical
  Society}{278}{2}{771--802}.
\PrintBackRefs{\CurrentBib}

\bibitem [\protect \citeauthoryear {%
Menaldi%
\ \BBA {} Taksar%
}{%
Menaldi%
\ \BBA {} Taksar%
}{%
{\protect \APACyear {1989}}%
}]{%
MT1989}
\APACinsertmetastar {%
MT1989}%
\begin{APACrefauthors}%
Menaldi, J\BHBI L.%
\BCBT {}\ \BBA {} Taksar, M\BPBI I.%
\end{APACrefauthors}%
\unskip\
\newblock
\APACrefYearMonthDay{1989}{}{}.
\newblock
{\BBOQ}\APACrefatitle {Optimal correction problem of a multidimensional
  stochastic system} {Optimal correction problem of a multidimensional
  stochastic system}.{\BBCQ}
\newblock
\APACjournalVolNumPages{Automatica}{25}{2}{223--232}.
\PrintBackRefs{\CurrentBib}

\bibitem [\protect \citeauthoryear {%
Meyer%
}{%
Meyer%
}{%
{\protect \APACyear {1976}}%
}]{%
Meyer76}
\APACinsertmetastar {%
Meyer76}%
\begin{APACrefauthors}%
Meyer, P\BPBI A.%
\end{APACrefauthors}%
\unskip\
\newblock
\APACrefYearMonthDay{1976}{}{}.
\newblock
{\BBOQ}\APACrefatitle {Martingales locales changement de variables, formules
  exponentielles} {Martingales locales changement de variables, formules
  exponentielles}.{\BBCQ}
\newblock
\BIn{} \APACrefbtitle {{S}{\'e}minaire de {P}robabilit{\'e}s {X}
  {U}niversit{\'e} de {S}trasbourg} {{S}{\'e}minaire de {P}robabilit{\'e}s {X}
  {U}niversit{\'e} de {S}trasbourg}\ (\BPGS\ 291--331).
\newblock
\APACaddressPublisher{}{Springer}.
\PrintBackRefs{\CurrentBib}

\bibitem [\protect \citeauthoryear {%
Ramanan%
}{%
Ramanan%
}{%
{\protect \APACyear {2006}}%
}]{%
ramanan2006}
\APACinsertmetastar {%
ramanan2006}%
\begin{APACrefauthors}%
Ramanan, K.%
\end{APACrefauthors}%
\unskip\
\newblock
\APACrefYearMonthDay{2006}{}{}.
\newblock
{\BBOQ}\APACrefatitle {Reflected diffusions defined via the extended
  {Skorokhod} map} {Reflected diffusions defined via the extended {Skorokhod}
  map}.{\BBCQ}
\newblock
\APACjournalVolNumPages{Electronic journal of probability}{11}{}{934--992}.
\PrintBackRefs{\CurrentBib}

\bibitem [\protect \citeauthoryear {%
Rosen%
}{%
Rosen%
}{%
{\protect \APACyear {1987}}%
}]{%
rosen1987}
\APACinsertmetastar {%
rosen1987}%
\begin{APACrefauthors}%
Rosen, J.%
\end{APACrefauthors}%
\unskip\
\newblock
\APACrefYearMonthDay{1987}{}{}.
\newblock
{\BBOQ}\APACrefatitle {Joint continuity of the intersection local times of
  {Markov} processes} {Joint continuity of the intersection local times of
  {Markov} processes}.{\BBCQ}
\newblock
\APACjournalVolNumPages{{Annals of Probability}}{15}{}{659--675}.
\PrintBackRefs{\CurrentBib}

\bibitem [\protect \citeauthoryear {%
Soner%
\ \BBA {} Shreve%
}{%
Soner%
\ \BBA {} Shreve%
}{%
{\protect \APACyear {1989}}%
}]{%
SS1989}
\APACinsertmetastar {%
SS1989}%
\begin{APACrefauthors}%
Soner, H\BPBI M.%
\BCBT {}\ \BBA {} Shreve, S\BPBI E.%
\end{APACrefauthors}%
\unskip\
\newblock
\APACrefYearMonthDay{1989}{}{}.
\newblock
{\BBOQ}\APACrefatitle {Regularity of the value function for a two-dimensional
  singular stochastic control problem} {Regularity of the value function for a
  two-dimensional singular stochastic control problem}.{\BBCQ}
\newblock
\APACjournalVolNumPages{SIAM Journal on Control and
  Optimization}{27}{4}{876--907}.
\PrintBackRefs{\CurrentBib}

\bibitem [\protect \citeauthoryear {%
Sun%
}{%
Sun%
}{%
{\protect \APACyear {2018}}%
}]{%
sun2018}
\APACinsertmetastar {%
sun2018}%
\begin{APACrefauthors}%
Sun, L\BHBI H.%
\end{APACrefauthors}%
\unskip\
\newblock
\APACrefYearMonthDay{2018}{}{}.
\newblock
{\BBOQ}\APACrefatitle {Systemic risk and interbank lending} {Systemic risk and
  interbank lending}.{\BBCQ}
\newblock
\APACjournalVolNumPages{Journal of Optimization Theory and
  Applications}{179}{2}{400--424}.
\PrintBackRefs{\CurrentBib}

\bibitem [\protect \citeauthoryear {%
Wang%
\ \BBA {} Ewald%
}{%
Wang%
\ \BBA {} Ewald%
}{%
{\protect \APACyear {2010}}%
}]{%
WE2010}
\APACinsertmetastar {%
WE2010}%
\begin{APACrefauthors}%
Wang, W\BHBI K.%
\BCBT {}\ \BBA {} Ewald, C\BHBI O.%
\end{APACrefauthors}%
\unskip\
\newblock
\APACrefYearMonthDay{2010}{}{}.
\newblock
{\BBOQ}\APACrefatitle {Dynamic voluntary provision of public goods with
  uncertainty: {a} stochastic differential game model} {Dynamic voluntary
  provision of public goods with uncertainty: {a} stochastic differential game
  model}.{\BBCQ}
\newblock
\APACjournalVolNumPages{Decisions in Economics and Finance}{33}{2}{97--116}.
\PrintBackRefs{\CurrentBib}

\bibitem [\protect \citeauthoryear {%
Widder%
}{%
Widder%
}{%
{\protect \APACyear {1941}}%
}]{%
widder2015laplace}
\APACinsertmetastar {%
widder2015laplace}%
\begin{APACrefauthors}%
Widder, D\BPBI V.%
\end{APACrefauthors}%
\unskip\
\newblock
\APACrefYear{1941}.
\newblock
\APACrefbtitle {The {L}aplace {T}ransform} {The {L}aplace {T}ransform}.
\newblock
\APACaddressPublisher{}{Princeton university press}.
\PrintBackRefs{\CurrentBib}

\bibitem [\protect \citeauthoryear {%
Williams%
, Chow%
\BCBL {}\ \BBA {} Menaldi%
}{%
Williams%
\ \protect \BOthers {.}}{%
{\protect \APACyear {1994}}%
}]{%
WCM1994}
\APACinsertmetastar {%
WCM1994}%
\begin{APACrefauthors}%
Williams, S.%
, Chow, P.%
\BCBL {}\ \BBA {} Menaldi, J.%
\end{APACrefauthors}%
\unskip\
\newblock
\APACrefYearMonthDay{1994}{}{}.
\newblock
{\BBOQ}\APACrefatitle {Regularity of the Free Boundary in Singular Stochastic
  Control} {Regularity of the free boundary in singular stochastic
  control}.{\BBCQ}
\newblock
\APACjournalVolNumPages{Journal of Differential Equations}{111}{1}{175 - 201}.
\PrintBackRefs{\CurrentBib}

\bibitem [\protect \citeauthoryear {%
Zariphopoulou%
}{%
Zariphopoulou%
}{%
{\protect \APACyear {1992}}%
}]{%
zariphopoulou1992investment}
\APACinsertmetastar {%
zariphopoulou1992investment}%
\begin{APACrefauthors}%
Zariphopoulou, T.%
\end{APACrefauthors}%
\unskip\
\newblock
\APACrefYearMonthDay{1992}{}{}.
\newblock
{\BBOQ}\APACrefatitle {Investment-consumption models with transaction fees and
  {M}arkov-chain parameters} {Investment-consumption models with transaction
  fees and {M}arkov-chain parameters}.{\BBCQ}
\newblock
\APACjournalVolNumPages{SIAM Journal on Control and
  Optimization}{30}{3}{613--636}.
\PrintBackRefs{\CurrentBib}

\end{thebibliography}

\appendix
\section{Verification theorem}\label{app:verification}
\begin{theorem}\label{thm:verification}
Let $u\in W_{\rm loc}^{2,\infty}(\mathbb{R}^N)$ be a convex solution to the HJB equation \eqref{po_hjb} and $0\leq \partial^2_{\pmb{\nu}}u(\pmb{x}) \leq C$ (in the weak sense). Under Assumptions {\bf A1-A3},  $u$ is equal to the value function $v$ of    (\ref{centrol_controller}):
$$v(\pmb{x}) = \min_{\pmb{\xi}\in \mathcal{U}}J(\pmb{x};\pmb{\xi})=u(\pmb{x}).$$
In addition, if there exists   $\pmb{\xi}^*\in \mathcal{U}$ such that
\begin{itemize}
    \item $\pmb{X}_t^*=\pmb{x}+\pmb{\sigma} \pmb{B}_t +\pmb{\xi}^*_t \in \overline{\mathcal{C}}_N$ for every $t \ge 0$, $\mathbb{P}$-a.s.;
    \item $\pmb{\xi}^*_t =\pmb{\xi}_0^*+\int_0^t \pmb{N}_s d\eta^*_s $ with $\pmb{\xi}_0^* = \pi(\pmb{x}) - \pmb{x}$ and $\eta_t^* = \int_0^t {\bf 1}_{\{\pmb{X}_s^* \in \partial\mathcal{C}_N, \pmb{N}_s\in \bf{r}(\pmb{X}_s^*)\}}d\eta^*_s$ for every $t \ge 0$, $\mathbb{P}$-a.s.;
    \item $\pmb{\xi}^*$ is continuous if $\pmb{\xi}_0^*=0$;
\end{itemize}
where $\mathcal{C}_N= \left\{\pmb{x} \given \beta(\nabla u(\pmb{x}))<1\right\}$, and $\gamma$ and $\pmb{r}$ are defined in \eqref{eq:gamma_direction}-\eqref{cone_original} such that Assumption {\bf A5} holds, then $\pmb{\xi}^*$ is an optimal control.
\end{theorem}

\begin{proof}
By the Sobolev embedding \cite[Ch. 9, Cor. 9.15]{brezis2010functional}, $u \in \mathcal{C}^1(\mathbb{R}^N)$ since $u\in\mathcal{W}_{\rm loc}^{2,\infty}(\mathbb{R}^N)$. In addition, $u$ is convex and $0\leq \partial^2_{\pmb{\nu}}u(\pmb{x}) \leq C$, then apply the It\^{o}-Tanaka-Meyer formula \citep{carlen1992semimartingale} to the function $e^{-\rho t} u(\pmb{X}_t)$ of the semi-martingale $\pmb{X}_t = \pmb{x}+\pmb{\mu}t+\pmb{\sigma}\pmb{B}_t+\pmb{\xi}_t$,
\begin{eqnarray}
    e^{-\alpha T}u(\pmb{X}_t)-u(\pmb{x}) &=& \int_0^T e^{-\alpha t} \nabla u(\pmb{X}_t) d \pmb{B}_t +\int_0^T e^{-\alpha t} (\mathcal{L}u(\pmb{X}_t)-\alpha u(\pmb{X}_t))dt\nonumber\\
    &&+\int_{0}^T e^{-\alpha T} (\sum_{i=1}^N \partial_{x^i}u(\pmb{X}_t)d\xi_t^{i,+}-\sum_{i=1}^N \partial_{x^i}u(\pmb{X}_t)d\xi_t^{i,-})\nonumber\\
    &&+\sum_{0 \leq t \leq T}  e^{-\alpha T} \left( \Delta u(\pmb{X}_t)-\sum_{i=1}^N \partial_{x^i}u(\pmb{X}_t)\Delta X_t^i\right),\label{ito}
\end{eqnarray}
with the notation $\Delta \phi_t := \phi_t-\phi_{t-}$. Since $u$ is  a convex solution to the HJB equation \eqref{po_hjb},  we have $\mathbb{P}$-a.s. for all $0\leq t \leq T$,
\begin{eqnarray}
&&  \rho u(\pmb{X}_t) - \mathcal{L} u(\pmb{X}_t) -H((\pmb{X}_t)) \leq 0,\label{verification_ineqs1}\\
  && \partial_{x^i} u(\pmb{X}_t) d\xi_t^{i,-}\leq L_iK_i^{-}d\xi_t^{i,-}, -L_iK_i^{+}d\xi_t^{i,+} \leq \partial_{x^i} u(\pmb{X}_t) d\xi_t^{i,+},\label{verification_ineqs2}\\
&&   \Delta u(\pmb{X}_t)-\sum_{i=1}^N \partial_{x^i}u(\pmb{X}_t)\Delta X_t^i \ge 0.\label{verification_ineqs3}
\end{eqnarray}
Taking expectation on both sides of \eqref{ito}, we have for any admissible policy $\pmb{\xi}$,
\begin{eqnarray}
    e^{-\alpha T} \mathbb{E}[u(\pmb{X}_t)] + \mathbb{E}\int_0^T e^{-\rho t}\Big(H(\pmb{X}_t)dt + K_i^+d\xi_{t}^{i,+}+K_i^-d\xi_{t}^{i,-}\Big) \ge u(\pmb{x}).
\end{eqnarray}
Since $0\leq \partial^2_{\pmb{\nu}}u(\pmb{x}) \leq C$, there exists constant $K=K(C)>0$ such that $|u(\pmb{x})|\leq K(1+\|\pmb{x}\|^2)$. Hence
$ \mathbb{E}[u(\pmb{X}_t)]  \leq 9K\left(1+\|\pmb{x}\|^2+\|\pmb{\sigma}\|^2\|\pmb{B}_T\|^2+\|\pmb{\xi}_T\|^2\right)$.

Now we show that $\mathbb{E}[\|\pmb{\xi}_T\|^2]=o(e^{\rho t})$. If this does not hold, then standard arguments (e.g. \cite[P 39]{widder2015laplace}) can show that there exists $i\in\{1,2,\cdots,N\}$ such that $\mathbb{E}[\int_0^{\infty}e^{-\rho t} (d\xi_t^{i,+}+d\xi_t^{i,-})]=\infty$, which violates the condition in the definition of admissible control set $\mathcal{U}_N$.
Hence by letting $T\rightarrow \infty$ we have
\begin{eqnarray}\label{eq:verification_oneside}
    \mathbb{E}\int_0^{\infty} e^{-\rho t}\Big(H(\pmb{X}_t)dt + K_i^+d\xi_{t}^{i,+}+K_i^-d\xi_{t}^{i,-}\Big) \ge u(\pmb{x}).
\end{eqnarray}
Under Assumption {\bf A1-A3}, Theorem \ref{prop:regularity_uniqueness} holds and hence $ u(\pmb{x})=v(\pmb{x})$ for all $\pmb{x}\in \mathbb{R}^N$.

To achieve the equality in \eqref{eq:verification_oneside}, it suffices to achieve the equalities in conditions \eqref{verification_ineqs1}-\eqref{verification_ineqs3}, which requires the following properties from the optimal control process $\pmb{\xi}^*$:
\begin{itemize}
    \item  $\pmb{X}_t^*=\pmb{x}+\pmb{\sigma} \pmb{B}_t +\pmb{\xi}^*_t \in \overline{\mathcal{C}}_N$  hence $\rho u(\pmb{X}^*_t) - \mathcal{L} u(\pmb{X}^*_t) -H(\pmb{X}^*_t)=0$ for every $t \ge 0$, $\mathbb{P}$-a.s.;
    \item The only possible jump is at time $0$ when $\pmb{x}\notin \mathcal{C}_N$. Under Assumption {\bf A5} and the convexity of $u$, we can show that $u(\pmb{x}) = u(\pi(\pmb{x}))+l(\pmb{x}-\pi(\pmb{x}))$ with
     $l(\pmb{y}) = \sum_i l_i(y_i)$, where 
\begin{eqnarray} 
    l_{i}(y_i)=\left\{
                \begin{array}{ll}
                 L_iK_i^- y_i, \quad \mbox{ if } y_i \geq 0,\\
                 -L_iK^+_i y_i, \quad \mbox{ if } y_i <0.
                \end{array}
              \right.
\end{eqnarray}
The proof is the same as the one for Theorem \ref{epsilon_out}. And hence the equality in \eqref{verification_ineqs3} holds.
    \item By the definition of $\pmb{\xi}^*$, $d\pmb{\xi}_t^*\neq 0$ only when $\pmb{X}^*_{t-}\notin \mathcal{C}_N$. Hence the equality in \eqref{verification_ineqs2} holds.
\end{itemize}
\end{proof}

\end{document}